\documentclass[10pt]{amsart}
      \usepackage[mathscr]{eucal}
      \usepackage{amsmath,amsfonts}
\def\rg{\hbox to 30pt{\rightarrowfill}}
\def\lg{\hbox to 30pt{\leftarrowfill}}

      \parskip\smallskipamount
          \newtheorem{theorem}{Theorem}[section]
      \newtheorem{definition}[theorem]{Definition}
      \newtheorem{proposition}[theorem]{Proposition}
      \newtheorem{corollary}[theorem]{Corollary}
      
      \newtheorem{example}[theorem]{Example}

      \makeatletter
      \@addtoreset{equation}{section}
      \makeatother

      \newcommand{\CC}{{\mathbb C}}
      \newcommand{\NN}{{\mathbb N}}
      
      \newcommand{\ZZ}{{\mathbb Z}}
      \newcommand{\DD}{{\mathbb D}}
      
      \newcommand{\FF}{{\mathbb F}}
      \newcommand{\TT}{{\mathbb T}}

      \newcommand{\cA}{{\mathcal A}}
      \newcommand{\cB}{{\mathcal B}}
      
      \newcommand{\cD}{{\mathcal D}}
      \newcommand{\cE}{{\mathcal E}}
      \newcommand{\cF}{{\mathcal F}}
      \newcommand{\cG}{{\mathcal G}}
      \newcommand{\cH}{{\mathcal H}}
      \newcommand{\cK}{{\mathcal K}}
      
      \newcommand{\cM}{{\mathcal M}}

      \newcommand{\cR}{{\mathcal R}}
      \newcommand{\cS}{{\mathcal S}}
      \newcommand{\cT}{{\mathcal T}}

      \newdimen\expt
      \def\boxit#1{\setbox0\hbox{$\displaystyle{#1}$}
            \hbox{\lower.4\expt
       \hbox{\lower3\expt\hbox{\lower\dp0
            \hbox{\vbox{\hrule height.4\expt
       \hbox{\vrule width.4\expt\hskip3\expt
            \vbox{\vskip3\expt\box0\vskip2\expt}%
       \hskip3\expt\vrule width.4\expt}\hrule height.4\expt}}}}}}
      
\begin{document}
       \pagestyle{myheadings}
      \markboth{ Gelu Popescu}{ Multivariable  Bohr
      inequalities }

      \title [ Bohr inequalities  for free holomorphic functions on  polyballs ]
      { Bohr inequalities   for free holomorphic functions on  polyballs }
        \author{Gelu Popescu}
      \date{February 24, 2017}
      \thanks{Research supported in  part by  NSF grant DMS 1500922}

      \address{Department of Mathematics, The University of Texas
      at San Antonio \\ San Antonio, TX 78249, USA}
      \email{\tt gelu.popescu@utsa.edu}

      \begin{abstract}
      Multivariable operator theory is used to provide Bohr inequalities for free holomorphic functions with operator coefficients on the regular polyball ${\bf B_n}$, ${\bf n}=(n_1,\ldots, n_k)\in \NN^k$, which is a noncommutative analogue of the scalar polyball $(\CC^{n_1})_1\times \cdots \times (\CC^{n_k})_1$. The Bohr radius $K_{mh}({\bf B_n})$ (resp.~$K_{h}({\bf B_n})$) associated with the multi-homogeneous (resp. homogeneous) power series expansions of the free holomorphic functions  are the main objects of study in this paper.  We extend a theorem of Bombieri and Bourgain for the disc $\DD:=\{z\in \CC: |z|<1\}$ to the polyball, and   obtain the estimations
     $\frac{1}{3\sqrt{k}}< K_{mh}({\bf B_n})< \frac{2\sqrt{\log k}}{\sqrt{k}}$ if  $k>1$,  extending Boas-Khavinson  result for the scalar polydisk $\DD^k$.

With respect to the homogeneous power series expansion, we  prove that  $K_{h}({\bf B_n})=1/3$,  extending the classical result,  and
       obtain  analogues of Carath\' eodory,  Fej\' er,  and Egerv\' ary-Sz\' azs inequalities  for   free holomorhic functions  with  operator coefficients and  positive  real parts on the polyball. These results are used to provide multivariable   analogues of Landau's   inequality   and  Bohr's inequality  when the norm is replaced by the numerical radius of  an operator.

 When specialized to the regular polydisk ${\bf D}^k$ (which corresponds to the case $n_1=\cdots =n_k=1$), we obtain new results concerning Bohr, Landau, Fej\' er, and Harnack inequalities   for operator-valued holomorphic functions  and $k$-pluriharmonic functions on the scalar  polydisc $\DD^k$.
The results of the paper can be used to obtain Bohr type inequalities for the noncommutative ball algebra $\boldsymbol\cA_{\bf n}$, the Hardy algebra ${\bf F}^\infty_{\bf n}$, and the $C^*$-algebra  $C^*({\bf S})$,  generated by the universal model ${\bf S}$ of the polyball ${\bf B_n}$.

 \end{abstract}

      \maketitle

\section*{Contents}
{\it

\quad Introduction

\begin{enumerate}
   \item[1.]   Preliminaries on Berezin transforms on noncommutative polyballs
   \item[2.] Bohr inequalities for free holomorphic functions on  polyballs
 \item[3.]  Bohr inequalities  for free  holomorphic functions with $F(0)\geq 0$ and $\Re F\leq I$
 \item[4.] The Bohr radius $K_{mh}({\bf B_n})$ and Bombieri-Bourgain theorem for the polyball
\item[5.]  The Bohr radius $K_{h}({\bf B_n})$ for the Hardy space $H^\infty({\bf B_n})$
\item[6.]   Fej\' er and  Bohr  inequalities  for  multivariable polynomials with operator coefficients
\item[7.]   Harnack   inequalities for free $k$-pluriharmonic functions
   \end{enumerate}

\quad References

}

      \section*{Introduction}

 Bohr's inequality
\cite{Bo}
asserts  that if  $f(z):=\sum\limits_{k=0}^\infty a_kz^k$ is an  analytic function on the open unit disc
$\DD:=\{z\in\CC: \ |z|<1\}$ such that  $\|f\|_\infty\leq 1$, then
$$\sum_{k=0}^\infty r^k |a_k|\leq 1\quad \text{  for }\  0\leq r\leq \frac{1} {3}.
$$
M. Riesz, Schur, and Weiner showed, independently, that $\frac{1} {3}$ is the best possible constant. Other proofs were  later
obtained
 by Sidon \cite{S} and Tomic \cite{T}.
 Dixon \cite{D} used Bohr's inequality in connection with
the long-standing problem of characterizing Banach algebras satisfying  the   von Neumann inequality \cite{vN} (see also \cite{Pa-book} and \cite{Pi}).

In 1997, Boas and Khavinson \cite{BK} introduced the Bohr radius $K_N$ for the Hardy space  $H^\infty(\DD^N)$ of bounded holomorphic function on the $N$-dimensional polydisc and showed that, if $N>1$, then
$$
\frac{1}{3\sqrt{N}}<K_N< \frac{2\sqrt{\log N}}{\sqrt{N}}.
$$
Inspired by this result, several authors  (see \cite{Ai}, \cite{AV}, \cite{BPS}, \cite{BK}, \cite{DFOOS}, \cite{DMS},
 \cite{DT}, \cite{PPoS}, \cite{Po-Bohr}, \cite{Po-domains} and the references therein) have considered multivariable analogues of
Bohr's inequality.
  Due to the remarkable  work by Defant, Frerick, Ortega-Cerdà,  Ouna\" ies, and Seip \cite{DFOOS}, and by
  Bayart, Pellegrino, and Seoane-Sep\' ulveda   \cite{BPS}, we know now the asymptotic  behaviour of the Bohr radius  $K_N$, i.e.
   $$
   \lim_{N\to \infty}\frac{K_N}{\sqrt{(\log N)/N}}=1.
   $$
 In \cite{DMS}, Defant, Maestre, and Schwarting   obtained  upper and lower estimates for the Bohr radius in the setting of  holomorphic functions defined on $\DD^N$ with values in Banach spaces.
 Noncommutative multivariable analogues of  Bohr's inequality were obtained in \cite{PPoS} and  \cite{Po-Bohr} for the class of noncommutative holomorphic functions on the open unit ball
 $$
[B(\cH)^k]_1:=\left\{(X_1,\ldots, X_k)\in B(\cH)^k:\ \|X_1
X_1^*+\cdots +X_kX_k^* \|^{1/2} <1\right\},
$$
where  $k\in \NN:=\{1,2,\ldots\}$ and  $B(\cH)$ is the algebra
  of all bounded linear operators on a Hilbert space $\cH$.

  The main goal of the  present paper is to  study  the Bohr phenomenon  in the setting of  free holomorphic functions on  the noncommutative polyball ${\bf B_n}$, ${\bf n}=(n_1,\ldots, n_k)\in \NN^k$,  which is a noncommutative analogue of the scalar polyball
$(\CC^{n_1})_1\times \cdots \times (\CC^{n_k})_1$, where $(\CC^{n})_1:=\{{\bf z}\in \CC^n: \ \|{\bf z}\|_2<1\}$.

To present our  results, we need some definitions.
  We denote by  $B(\cH)^{n_1}\times_c\cdots \times_c B(\cH)^{n_k}$, where $n_i \in\NN$,
   the set of all tuples  ${\bf X}:=({ X}_1,\ldots, { X}_k)$ in $B(\cH)^{n_1}\times\cdots \times B(\cH)^{n_k}$
     with the property that the entries of ${X}_s:=(X_{s,1},\ldots, X_{s,n_s})$  are commuting with the entries of
      ${X}_t:=(X_{t,1},\ldots, X_{t,n_t})$  for any $s,t\in \{1,\ldots, k\}$, $s\neq t$.
  Note that  the operators $X_{s,1},\ldots, X_{s,n_s}$ are not necessarily commuting.
  Define  the polyball
  $${\bf P_n}(\cH):=[B(\cH)^{n_1}]_1\times_c \cdots \times_c [B(\cH)^{n_k}]_1.
  $$
    If $A$ is  a positive invertible operator, we write $A>0$. The {\it regular polyball} on the Hilbert space $\cH$  is defined by
$$
{\bf B_n}(\cH):=\left\{ {\bf X}\in {\bf P_n}(\cH) : \ {\bf \Delta_{X}}(I)> 0  \right\},
$$
where
 the {\it defect mapping} ${\bf \Delta_{X}}:B(\cH)\to  B(\cH)$ is given by
$
{\bf \Delta_{X}}:=\left(id -\Phi_{X_1}\right)\circ \cdots \circ\left(id -\Phi_{ X_k}\right),
$
 and
$\Phi_{X_i}:B(\cH)\to B(\cH)$  is the completely positive linear map defined by
$$\Phi_{X_i}(Y):=\sum_{j=1}^{n_i}   X_{i,j} Y X_{i,j} ^*, \qquad Y\in B(\cH).
$$
 Note that if $k=1$, then ${\bf B_n}(\cH)$ coincides with the noncommutative unit ball $[B(\cH)^{n_1}]_1$.
   We remark that the scalar representation of
  the  ({\it abstract})
 {\it regular polyball} ${\bf B}_{\bf n}:=\{{\bf B_n}(\cH):\ \cH \text{\ is a Hilbert space} \}$ is
   ${\bf B_n}(\CC)={\bf P_n}(\CC)= (\CC^{n_1})_1\times \cdots \times (\CC^{n_k})_1$.
A multivariable  operator model theory and a theory of free holomorphic functions on polydomains which admit universal operator models have been recently developed  in \cite{Po-Berezin3} and \cite{Po-Berezin-poly}. An important feature of these theories is that they  are related, via noncommutative Berezin transforms, to the study of the operator algebras generated by the universal models associated with the domains, as well as to the theory of functions in several complex variable (\cite{Ru1}, \cite{Ru2}). These results played a crucial role in   our work on
  the curvature invariant \cite{Po-curvature-polyball}, the Euler characteristic  \cite{Po-Euler-charact}, and the  group of  free holomorphic automorphisms  on  noncommutative regular polyballs \cite{Po-automorphisms-polyball}, and will play an important role in the present paper.

For each $i\in\{1,\ldots, k\}$, let $\FF_{n_i}^+$ be the free semigroup with generators $g_1^i,\ldots, g^i_{n_i}$ and identity $g^i_0$.
Let $Z_i:=(Z_{i,1},\ldots, Z_{i,n_i})$ be
an  $n_i$-tuple of noncommuting indeterminates and assume that, for any
$p,q\in \{1,\ldots, k\}$, $p\neq q$, the entries in $Z_p$ are commuting
 with the entries in $Z_q$. We set $Z_{i,\alpha_i}:=Z_{i,j_1}\cdots Z_{i,j_p}$
  if $\alpha_i\in \FF_{n_i}^+$ and $\alpha_i=g_{j_1}^i\cdots g_{j_p}^i$, and
   $Z_{i,g_0^i}:=1$. If  $\boldsymbol\alpha:=(\alpha_1,\ldots, \alpha_k)\in \FF_{n_1}^+\times \cdots \times\FF_{n_k}^+$, we denote ${\bf Z}_{\boldsymbol\alpha}:= {Z}_{1,\alpha_1}\cdots {Z}_{k,\alpha_k}$.
Let $\ZZ$ be the set of all integers and $\ZZ_+$ be the set of all nonnegative integers.
  A formal power series $\varphi=\sum\limits_{\boldsymbol\alpha\in \FF_{n_1}^+\times \cdots \times\FF_{n_k}^+} a_{(\boldsymbol\alpha)} {\bf Z}_{\boldsymbol\alpha}$
  in ideterminates ${\bf Z}=\{Z_{i,j}\}$ and scalar coefficients $a_{(\boldsymbol\alpha)}\in \CC$
  is called   {\it free holomorphic function}  on the
{\it abstract  polyball}
$ {\bf B_n}:=
\{ {\bf  B_n}(\cH):\ \cH \text{ is a Hilbert space}\}$ if the series
$$
\varphi({\bf X} ):=\sum\limits_{{\bf p}\in \ZZ_+^k}\sum\limits_{\boldsymbol\alpha\in \Lambda_{\bf p}} a_{(\boldsymbol\alpha)}  {\bf X}_{\boldsymbol\alpha} \qquad (\text{\it multi-homogeneous expansion})
$$
is convergent in the operator norm topology for any ${\bf X}=\{X_{i,j}\}\in {\bf B_n}(\cH)$   and any Hilbert space $\cH$. Here we use the notation
$\Lambda_{\bf p}:=\{\boldsymbol\alpha=(\alpha_1,\ldots, \alpha_k)\in \FF_{n_1}^+\times\cdots \times \FF_{n_k}^+: \ |\alpha_i|=p_i \}$, where ${\bf p}=(p_1,\ldots, p_k)\in \ZZ_+^k$ and $|\alpha_i|$ is the length of $\alpha_i$.
In this case, we proved \cite{Po-automorphisms-polyball} that
  $$
 \varphi({\bf X} )
  =
\sum_{q=0}^\infty \sum_{\boldsymbol\alpha\in  \Gamma_{ q}} a_{(\boldsymbol\alpha)}   {\bf X}_{\boldsymbol\alpha}\qquad (\text{\it homogeneous expansion})
$$
where $\Gamma_{q}:=\{\boldsymbol\alpha=(\alpha_1,\ldots, \alpha_k)\in \FF_{n_1}^+\times\cdots \times \FF_{n_k}^+: \  |\alpha_1|+\cdots +|\alpha_k|=q\}$ and
 the convergence of the series  is in the operator norm topology. In fact, this result holds true for free holomorphic functions with operator coefficients.

The Bohr radius $K_{mh}({\bf B_n})$ for the Hardy space $H^\infty({\bf B_n})$ of all bounded free holomorphic functions on ${\bf B_n}$, with respect to the multi-homogeneous expansion  of its elements, is the largest $r\geq 0$ such that
 \begin{equation*}
\sum\limits_{{\bf p}\in \ZZ_+^k}\left\|\sum\limits_{\boldsymbol\alpha\in \Lambda_{\bf p}} a_{(\boldsymbol\alpha)}  {\bf X}_{\boldsymbol\alpha}\right\|
\leq \|\varphi\|_\infty, \qquad {\bf X}\in   r {\bf B_n}(\cH)^-,
\end{equation*}
  for any   $\varphi\in H^\infty({\bf B_n})$. Similarly, we define the Bohr radius
  $K_{h}({\bf B_n})$ for the Hardy space $H^\infty({\bf B_n})$  with respect to the homogeneous expansion  of its elements. Note that when $k=1$ the two definitions coincide.
When the Hardy space  $H^\infty({\bf B_n})$ is replaced by the subspace $ H_0^\infty({\bf B_n}):=\{f\in  H^\infty({\bf B_n}): \ f(0)=0\}$ the corresponding Bohr radii are denoted by
$K^0_{mh}({\bf B_n})$ and $K^0_{h}({\bf B_n})$, respectively.

In Section 2, we obtain Weiner type inequalities for the  coefficients of bounded free holomorphic functions on polyballs, which are used to obtain Bohr inequalities for  bounded free holomorphic functions with operator coefficients. As a   consequence, we show that
the Bohr radius $K_{mh}({\bf B_n})$   satisfies the inequalities
$$ 1-\left(\frac{2}{3}\right)^{1/k}\leq K_{mh}({\bf B_n})\leq \frac{1}{3}, \qquad k\geq 1.
 $$
Let $F\in H^\infty({\bf B_n})$ have the  representation
$
F({\bf X} ):=\sum_{{\bf p}\in \ZZ_+^k}\sum_{\boldsymbol\alpha\in \Lambda_{\bf p}} a_{(\boldsymbol\alpha)}  {\bf X}_{\boldsymbol\alpha}
$
and let
$$
\cD(F,r):=\sum\limits_{{\bf p}\in \ZZ_+^k}r^{|{\bf p}|}\left\|\sum\limits_{\boldsymbol\alpha\in \Lambda_{\bf p}} a_{(\boldsymbol\alpha)}  {\bf S}_{\boldsymbol\alpha}\right\|
$$
be the associated majorant series, where ${\bf S}=\{{\bf S}_{ij}\}$ is the universal model of the polyball (see Section 1). Define
$$
d_{\bf B_n}(r):=\sup \frac{\cD(F,r)}{\|F\|_\infty}, \qquad r\in [0,1),
$$
where the supremum is taken over all $F\in H^\infty({\bf B_n})$  with $F$ not identically $0$. The results of this section show that $d_{\bf B_n}(r)=1$ if $0\leq r\leq 1-\left(\frac{2}{3}\right)^{1/k}$.  While we obtain upper bounds for
$m_{\bf B_n}(r)$,
the precise value of $d_{\bf B_n}(r)$ as $ 1-\left(\frac{2}{3}\right)^{1/k}<r<1$ remains unknown, in general.
   Progress on this problem was made, in the classical case of the disc $\DD$,
 by Bombieri and Bourgain in \cite{BB}, where they proved  that
 $d_{\DD}(r)\sim \frac{1}{\sqrt{1-r^2}}$ as $r\to 1$.
We  extend their result proving  that  $d_{\bf B_n}(r)$ behaves asymptotically as  $\left(\frac{1}{\sqrt{1-r^2}}\right)^k$ if $r\to 1$, i.e.
 $$
 \lim_{r\to 1} \frac{d_{\bf B_n}(r)}{\left(\frac{1}{\sqrt{1-r^2}}\right)^k}=1.
 $$
In particular, the result applies to  the scalar polydisc $\DD^k$.

In Section 3, we obtain an analogue of Landau's  inequality \cite{LG} for  bounded free holomorhic functions with operator coefficients on the polyball  (see Theorem \ref{W-O}). This is used to obtain Bohr inequalities for free holomorphic functions
 $F:{\bf B_n}(\cH)\to B(\cK)\otimes_{min}B(\cH)$
  such that $ F(0)\geq 0$ and  $\Re F({\bf X})\leq I$ for any ${\bf X}\in {\bf B_n}(\cH)$.
The result plays a crucial role in Section 4, where we prove that the Bohr radius $ K_{mh}({\bf B_n})$
satisfies the inequalities
$$
\frac{1}{3\sqrt{k}}< K_{mh}({\bf B_n})< \frac{2\sqrt{\log k}}{\sqrt{k}}, \qquad k>1,
$$
and  obtain the asymptotic upper bound
$$
\limsup_{k\to\infty}\frac{K_{mh}({\bf B_n})}{ \sqrt{(\log k)/k}}\leq 1.
$$

Section 5 concerns the Bohr radius $K_{h}({\bf B_n})$ for the Hardy space $H^\infty({\bf B_n})$, with respect to the homogeneous expansion of its elements. Using the results of Section 3, we prove that
$$
K_{h}({\bf B_n})=\frac{1}{3},
$$
which extends the classical result to our multivariable noncommutative setting.
Let $F\in H^\infty({\bf B_n})$ with representation
$
F({\bf X} ):=\sum_{q=0}^\infty \sum_{\boldsymbol\alpha\in \Gamma_q}  a_{(\boldsymbol \alpha)} {\bf X}_{\boldsymbol \alpha}
$
and let
$$
\cM(F,r):=\sum_{q=0}^\infty r^q \left\|\sum_{\boldsymbol\alpha\in \Gamma_q}  a_{(\boldsymbol \alpha)} {\bf S}_{\boldsymbol \alpha}\right\|\quad \text{ and }\quad  m_{\bf B_n}(r):=\sup \frac{\cM(F,r)}{\|F\|_\infty}, \qquad r\in [0,1),
$$
where the supremum is taken over all $F\in H^\infty({\bf B_n})$  with $F$ not identically $0$.
The results of this section show that $m_{\bf B_n}(r)=1$ if $r\in[0,\frac{1}{3}]$.  While we obtain upper bounds for
$m_{\bf B_n}(r)$,
the precise value of $m_{\bf B_n}(r)$ as $\frac{1}{3}<r<1$ remains unknown, in general.

Concerning  the  Bohr radius $K_{mh}^0({\bf B_n})$,  we show  that it satisfies    the inequalities
$$
\sqrt{1-\left(\frac{1}{2}\right)^{1/k}}\leq K_{mh}^0({\bf B_n})\leq \frac{1}{\sqrt{2}},   \qquad  \text{ if } \ k\geq 1,
$$
and
$$
\frac{1}{2\sqrt{k}}< K_{mh}^0({\bf B_n})< \frac{2\sqrt{\log k}}{\sqrt{k}}, \qquad \text{ if } \ k>1.
$$
Estimations for the Bohr radius $K_{h}^0({\bf B_n})$  are also obtained.

In Section 6, we   obtain  analogues of Carath\' eodory's   inequality \cite{Ca}, and Fej\' er  \cite{Fe} and Egerv\' ary-Sz\' azs inequalities \cite{ES} for   free holomorhic functions  with  operator coefficients and  positive  real parts on the polyball  (see Theorem \ref{Fejer}). These results are use to provide (see Theorem \ref{OM}) an analogue of Landau's   inequality \cite{LG}   and  Bohr type inequalities  when the norm is replaced by the numerical radius of  an operator, i.e.
$$
\omega(T):=\sup\{|\left<Th,h\right>| : \ h\in\cH, \|h\|=1\}, \qquad T\in B(\cH).
$$
In particular, we obtain the following   numerical radius versions of Landau's inequality and  Bohr's inequality for  free holomorphic functions on polyballs.
If $m\in \NN\cup \{\infty\}$ and
   $f({\bf X}):= \sum\limits_{q=1}^m\sum\limits_{\boldsymbol\alpha\in \Gamma_q} a_{(\boldsymbol \alpha)} {\bf X}_{\boldsymbol \alpha}
$
 is a free holomorphic function with $f(0)\geq 0$ and $\Re f({\bf X})\leq I$ on the polyball ${\bf B_n}$,  then $$ \omega\left(\sum_{\boldsymbol\alpha\in \Gamma_q} a_{(\boldsymbol\alpha)} {\bf S}_{\boldsymbol\alpha}\right)\leq 2(1-a_0)\cos \frac{\pi}{\left[\frac{m}{q}\right]+2},\qquad q\in \{1,\ldots, m\},$$
  and
$$
  \sum_{q=0}^m  \omega\left(\sum_{\boldsymbol\alpha\in \Gamma_q} a_{(\boldsymbol\alpha)} {\bf S}_{\boldsymbol\alpha}\right)r^q\leq 1,\qquad r\in [0,t_m],
$$
    where  $t_m\in (0,1)$ is the solution of the equation
$$
\sum_{q=1}^m t^q \cos \frac{\pi}{\left[\frac{m}{q}\right]+2}=\frac{1}{2},
$$
 and $[x]$ is the integer part of $x$. Moreover, the sequence $\{t_m\}$ is  strictly decreasing and converging  to $\frac{1}{3}$. When $m=\infty$, we have $t_\infty=\frac{1}{3}$.
As a consequence of these results,   we deduce  that if $f$ is a
  holomorphic function   on the polydisc $\DD^k$ such that $\Re f({\bf z})\leq 1$ for  ${\bf z}\in \DD^k$
 and   $f({\bf a})\geq 0$ for some ${\bf a}=(a_1, \ldots, a_k)\in \DD^k$, then
$$
\sum_{i=1}^k (1-|a_i|^2)\left|\left(\frac{\partial f}{\partial z_i}\right)({\bf a})\right|
\leq  2(1-f({\bf a})),
$$
which is an extension of Landau's inequality to the polydisk.

In Section 7, we provide Harnack type inequalities for positive free $k$-pluriharmonic function  with operator coefficients  on polyballs. In particular,   if
$F$ is a positive free $k$-pluriharmonic function  with scalar coefficients   and of   degree $m_i\in \NN\cup \{\infty\}$ with respect to the variables $X_{i,1},\ldots, X_{i,n_i}$,  then we prove that
 $$
   F({\bf X})\leq F(0)
   \prod_{i=1}^k \left(1+2\sum_{p_i=1}^{m_i}\rho_i^{p_i} \cos \frac{\pi}{\left[\frac{m_i}{p_i}\right]+2}\right)\leq F(0)\prod_{i=1}^k \frac{1+\rho_i}{1-\rho_i}
 $$
 for any ${\bf X}\in \boldsymbol\rho{\bf B}_{\bf n}(\cH)^-$    and  $\boldsymbol\rho:=(\rho_1,\ldots, \rho_k)\in [0,1)^k$, where $[x]$ is the integer part of $x$.

We remark that when $n_1=\cdots= n_k=1$ the free holomorphic functions on the regular polydisc ${\bf D}^k:={\bf B_n}$ can be identified with the holomorphic functions on the scalar polydisc $\DD^k$, and  the Hardy space $H^\infty({\bf B_n})$ can be identified with the Hardy space $H^\infty(\DD^k)$. As a consequence all the results  of the present paper hold true in this particular setting. In this way, we recover some known results but at the same time we provide new results concerning Bohr, Landau,  Fej\' er, and Harnack inequalities   for operator-valued holomorphic functions  and $k$-pluriharmonic functions on the polydisc.

We should also mention that our results can be used to obtain Bohr type inequalities for the noncommutative ball algebra $\boldsymbol\cA_{\bf n}$, the Hardy algebra ${\bf F}^\infty_{\bf n}$, and the $C^*$-algebra  $C^*({\bf S})$,  generated by the universal model ${\bf S}=\{{\bf S}_{i,j}\}$ of the polyball ${\bf B_n}$.

\section{Preliminaries on Berezin transforms on noncommutative polyballs}

Let $H_{n_i}$ be
an $n_i$-dimensional complex  Hilbert space with orthonormal basis $e^i_1,\ldots, e^i_{n_i}$.
  We consider the {\it full Fock space}  of $H_{n_i}$ defined by
$F^2(H_{n_i}):=\CC 1 \oplus\bigoplus_{p\geq 1} H_{n_i}^{\otimes p},$
where  $H_{n_i}^{\otimes p}$ is the
(Hilbert) tensor product of $p$ copies of $H_{n_i}$. Let $\FF_{n_i}^+$ be the unital free semigroup on $n_i$ generators
$g_{1}^i,\ldots, g_{n_i}^i$ and the identity $g_{0}^i$.
  Set $e_\alpha^i :=
e^i_{j_1}\otimes \cdots \otimes e^i_{j_p}$ if
$\alpha=g^i_{j_1}\cdots g^i_{j_p}\in \FF_{n_i}^+$
 and $e^i_{g^i_0}:= 1\in \CC$.
  The length of $\alpha\in
\FF_{n_i}^+$ is defined by $|\alpha|:=0$ if $\alpha=g_0^i$  and
$|\alpha|:=p$ if
 $\alpha=g_{j_1}^i\cdots g_{j_p}^i$, where $j_1,\ldots, j_p\in \{1,\ldots, n_i\}$.
 We  define
 the {\it left creation  operator} $S_{i,j}$ acting on the  Fock space $F^2(H_{n_i})$  by setting
$
S_{i,j} e_\alpha^i:=  e^i_j\otimes e^i_{ \alpha}$, $\alpha\in \FF_{n_i}^+,
$
 and
  the operator  ${\bf S}_{i,j}$ acting on the Hilbert tensor  product
$F^2(H_{n_1})\otimes\cdots\otimes F^2(H_{n_k})$ by setting
$${\bf S}_{i,j}:=\underbrace{I\otimes\cdots\otimes I}_{\text{${i-1}$
times}}\otimes \,S_{i,j}\otimes \underbrace{I\otimes\cdots\otimes
I}_{\text{${k-i}$ times}},
$$
where  $i\in\{1,\ldots,k\}$ and  $j\in\{1,\ldots,n_i\}$.  We denote ${\bf S}:=({\bf S}_1,\ldots, {\bf S}_k)$, where  ${\bf S}_i:=({\bf S}_{i,1},\ldots,{\bf S}_{i,n_i})$, or ${\bf S}:=\{{\bf S}_{i,j}\}$.   The noncommutative Hardy algebra ${\bf F}_{\bf n}^\infty$ (resp.~the polyball algebra $\boldsymbol\cA_{\bf n}$) is the weakly closed (resp.~norm closed) non-selfadjoint  algebra generated by $\{{\bf S}_{i,j}\}$ and the identity.
Similarly, we   define
 the {\it right creation  operator} $R_{i,j}:F^2(H_{n_i})\to
F^2(H_{n_i})$     by setting
 $
R_{i,j} e_\alpha^i:=  e^i_ {\alpha }\otimes e^i_j$ for $ \alpha\in \FF_{n_i}^+,
$
 and
  the operator  ${\bf R}_{i,j}$ acting on the Hilbert tensor  product
$F^2(H_{n_1})\otimes\cdots\otimes F^2(H_{n_k})$ by setting
$${\bf R}_{i,j}:=\underbrace{I\otimes\cdots\otimes I}_{\text{${i-1}$
times}}\otimes \,R_{i,j}\otimes \underbrace{I\otimes\cdots\otimes
I}_{\text{${k-i}$ times}}.
$$
  The polyball algebra $\boldsymbol\cR_{\bf n}$ is the norm closed non-selfadjoint  algebra generated by $\{{\bf R}_{i,j}\}$ and the identity.

We recall (see \cite{Po-poisson}, \cite{Po-Berezin-poly}) some basic properties for  the   noncommutative Berezin     transforms associated  with regular polyballs.
 Let  ${\bf X}=({ X}_1,\ldots, { X}_k)\in {\bf B_n}(\cH)^-$ with $X_i:=(X_{i,1},\ldots, X_{i,n_i})$.
We  use the notation
$X_{i,\alpha_i}:=X_{i,j_1}\cdots X_{i,j_p}$
  if  $\alpha_i=g_{j_1}^i\cdots g_{j_p}^i\in \FF_{n_i}^+$ and
   $X_{i,g_0^i}:=I$.
The {\it noncommutative Berezin kernel} associated with any element
   ${\bf X}$ in the noncommutative polyball ${\bf B_n}(\cH)^-$ is the operator
   $${\bf K_{X}}: \cH \to F^2(H_{n_1})\otimes \cdots \otimes  F^2(H_{n_k}) \otimes  \overline{{\bf \Delta_{X}}(I) (\cH)}$$
   defined by
   $$
   {\bf K_{X}}h:=\sum_{\beta_i\in \FF_{n_i}^+, i=1,\ldots,k}
   e^1_{\beta_1}\otimes \cdots \otimes  e^k_{\beta_k}\otimes {\bf \Delta_{X}}(I)^{1/2} X_{1,\beta_1}^*\cdots X_{k,\beta_k}^*h, \qquad h\in \cH,
   $$
   where $ {\bf \Delta_{X}}(I)$ is the defect operator.
A very  important property of the Berezin kernel is that
     ${\bf K_{X}} { X}^*_{i,j}= ({\bf S}_{i,j}^*\otimes I)  {\bf K_{X}}$ for any  $i\in \{1,\ldots, k\}$ and $ j\in \{1,\ldots, n_i\}.
    $
    The {\it Berezin transform at} ${\bf X}\in {\bf B_n}(\cH)$ is the map $ \boldsymbol{\cB_{\bf X}}: B(\otimes_{i=1}^k F^2(H_{n_i}))\to B(\cH)$
 defined by
\begin{equation*}
 {\boldsymbol\cB_{\bf X}}[g]:= {\bf K^*_{\bf X}} (g\otimes I_\cH) {\bf K_{\bf X}},
 \qquad g\in B(\otimes_{i=1}^k F^2(H_{n_i})).
 \end{equation*}
  If $g$ is in   the $C^*$-algebra $C^*({\bf S})$ generated by ${\bf S}_{i,1},\ldots,{\bf S}_{i,n_i}$, where $i\in \{1,\ldots, k\}$, we  define the Berezin transform at  ${\bf X}\in {\bf B_n}(\cH)^-$  by
  $${\boldsymbol\cB_{\bf X}}[g]:=\lim_{r\to 1} {\bf K^*_{r\bf X}} (g\otimes I_\cH) {\bf K_{r\bf X}},
 \qquad g\in  C^*({\bf S}),
 $$
 where the limit is in the operator norm topology.
In this case, the Berezin transform at ${\bf X}$ is a unital  completely positive linear  map such that
 $${\boldsymbol\cB_{\bf X}}({\bf S}_{\boldsymbol\alpha} {\bf S}_{\boldsymbol\beta}^*)={\bf X}_{\boldsymbol\alpha} {\bf X}_{\boldsymbol\beta}^*, \qquad \boldsymbol\alpha, \boldsymbol\beta \in \FF_{n_1}^+\times \cdots \times\FF_{n_k}^+,
 $$
 where  ${\bf S}_{\boldsymbol\alpha}:= {\bf S}_{1,\alpha_1}\cdots {\bf S}_{k,\alpha_k}$ if  $\boldsymbol\alpha:=(\alpha_1,\ldots, \alpha_k)\in \FF_{n_1}^+\times \cdots \times\FF_{n_k}^+$.
Using the noncommutative Berezin transforms, we proved in \cite{Po-automorphisms-polyball} that a formal power series $\varphi:=\sum\limits_{\boldsymbol\alpha\in \FF_{n_1}^+\times \cdots \times\FF_{n_k}^+} a_{(\boldsymbol\alpha)}  {\bf Z}_{\boldsymbol\alpha}$ is a  free holomorphic function (with scalar coefficients) on the
abstract  polyball
${\bf B_n}$  if and only if  the series
$\sum\limits_{(p_1,\ldots, p_k)\in \ZZ_+^k}\|\sum\limits_{{ \alpha_i\in \FF_{n_i}^+, |\alpha_i|=p_i}\atop{i\in \{1,\ldots, k\} }}a_{(\boldsymbol\alpha)}  r_1^{p_1}\cdots r_k^{p_k}{\bf S}_{\boldsymbol\alpha}\|
$
converges for any $r_i\in [0, 1)$,   which is equivalent to the condition that the series
$
\sum_{q=0}^\infty \|\sum_{{\boldsymbol\alpha\in \FF_{n_1}^+\times \cdots \times\FF_{n_k}^+ }\atop {|\alpha_1|+\cdots +|\alpha_k|=q}} a_{(\boldsymbol\alpha)}  r^q {\bf S}_{\boldsymbol\alpha}\|
$
is convergent for any $r\in [0,1)$.
In \cite{Po-Berezin-poly},
we identified the  noncommutative algebra
${\bf F}_{\bf n}^\infty$ with the  Hardy subalgebra $H^\infty({\bf B_n})$ of   bounded free holomorphic functions  on
${\bf B_n}$ with scalar coefficients.
More precisely, we proved that  the map
$ \Phi:H^\infty({\bf B_n})\to {\bf F}_{n}^\infty $ defined by
$
\Phi\left(\varphi\right):=\text{\rm SOT-}\lim_{r\to 1}\varphi(r{\bf S}),
$
is a completely isometric isomorphism of operator algebras, where
$\varphi(r{\bf S}):=\sum_{q=0}^\infty \sum_{{\boldsymbol\alpha\in \FF_{n_1}^+\times \cdots \times\FF_{n_k}^+ }\atop {|\alpha_1|+\cdots +|\alpha_k|=q}} r^q a_{(\alpha)} {\bf S}_{\boldsymbol\alpha}$ and the convergence of the series is in the operator norm topology.
Moreover, if  $\varphi$
is  a free holomorphic function on the abstract  polyball ${\bf B_n}$, then the following statements are equivalent:
 \begin{enumerate}
 \item[(i)]$\varphi\in H^\infty({\bf B_n})$;
\item[(ii)] $\sup\limits_{0\leq r<1}\|\varphi(r{\bf S})\|<\infty$;
\item[(iii)]
there exists $\psi\in {\bf F}_{\bf n}^\infty $ with $\varphi({\bf X})={\boldsymbol\cB}_{\bf X}[\psi]$, where ${\boldsymbol\cB}_{\bf X}$ is the  noncommutative Berezin
transform  associated with the abstract   polyball ${\bf B_n}$. Moreover, $\psi$ is uniquely determined by $\varphi$, namely,
$\psi=\text{\rm SOT-}\lim_{r\to 1}\varphi(r{\bf S})$
and
\begin{equation*}
\|\psi\|=\sup_{0\leq r<1}\|\varphi(r{\bf S})\|=
\lim_{r\to 1}\|\varphi(r{\bf S})\|=\|\varphi\|_\infty.
\end{equation*}
\end{enumerate}

We use the notation $\widehat \varphi:=\psi$ and call $\widehat \varphi$ the ({\it model}) {\it boundary function} of $\varphi$ with respect to the universal model ${\bf S}$ of the regular polyball. Similar results hold for free holomorphic functions with operator coefficients.
More information on  noncommutative Berezin transforms and multivariable operator theory on noncommutative balls and  polydomains can be found in \cite{Po-poisson}, \cite{Po-unitary},    and \cite{Po-Berezin-poly}.
For basic results on completely positive (resp.~bounded)  maps  we refer the reader to \cite{Pa-book} and \cite{Pi}.

If ${\bf z}:=(z_1,\ldots, z_k)\in \CC^k$ and ${\bf p}:=(p_1,\ldots, p_k)\in \ZZ_+^k$, we write ${\bf z}^{\bf p}$ for the monomial $z_1^{p_1}\cdots z_k^{p_k}$ and use the notation $|{\bf p}|:=p_1+\cdots +p_k$. Similarly, if ${\bf X}:=(X_1,\ldots, X_k)\in B(\cH)^k$, then ${\bf X}^{\bf p}:=X_1^{p_1}\cdots X_k^{p_k}$.
When ${\bf n}=(n_1,\ldots, n_k)$ with $n_1=\cdots=n_k=1$, we call ${\bf D}^k:={\bf B_n}$ the {\it regular polydisk}. Note that the scalar representation ${\bf D}^k(\CC)$ coincides with the scalar polydisk $\DD^k$.

\begin{proposition} \label{prelim} If $\{a_{\bf p}\}_{{\bf p}\in \ZZ_+^k}\subset \CC$, then
$f({\bf z}):=\sum_{{{\bf p} \in \ZZ_+^k}} a_{\bf p} {\bf z}^{\bf p}$ is  a holomorphic function on the scalar  polydisk $\DD^k$ if and only if
$F({\bf X} ):=\sum_{{{\bf p} \in \ZZ_+^k}} a_{\bf p} {\bf X}^{\bf p}$  is a free holomorphic function on the regular polydisc ${\bf D}^k$. Moreover, the following statements hold.
\begin{enumerate}
\item[(i)]  $f\in H^\infty(\DD^k)$ if and only if
$F\in H^\infty({\bf D}^k)$, in which case  $\|f\|_\infty=\|F\|_\infty$.
\item[(ii)] $\Re f({\bf z})\leq 1$ for any ${\bf z}\in \DD^k$
if and only if $\Re F(r{\bf S})\leq I$ for any $r\in [0,1)$, where
${\bf S}=(S_1,\ldots, S_k)$ is   the universal model of the regular polydisk ${\bf D}^k$.
\item[(iii)] The Banach algebra $H^\infty(\DD^k)$ is isometrically isomorphic to  $H^\infty({\bf D}^k)$, which is isometrically embedded in $H^\infty({\bf B_n})$.
\end{enumerate}
\end{proposition}
\begin{proof} Note that
$f({\bf z}):=\sum_{{{\bf p} \in \ZZ_+^k}} a_{\bf p} {\bf z}^{\bf p}$ is  a holomorphic function on the scalar  polydisk $\DD^k$ if and only if $\sum_{{\bf p}\in \ZZ_+^k} r^{|{\bf p}|} |a_{\bf p}|<\infty$ for any $r\in [0,1)$ (see \cite{Ru1}). Due to the remarks preceding the proposition, the latter condition is equivalent to
  the fact that
$F({\bf X}):=\sum_{{\bf p} \in \ZZ_+^k} a_{\bf p}{\bf X}^{\bf p} $
  is  a free holomorphic   function on the regular polydisc ${\bf D}^k$. Let  ${\bf S}=(S_1,\ldots, S_k)$ be the  universal model of the regular polydisk ${\bf D}^k$ and note that $S_1,\ldots, S_k$ are unitarily equivalent to the multiplication operators  by the coordinate functions $z_1,\ldots, z_k$ on the Hardy space $H^2(\DD^k)$.
  If $f\in H^\infty(\DD^k)$, then standard arguments imply
  $
  \sup_{0\leq r<1}\|F(r{\bf S})\|\leq \sup_{{\bf z}\in \DD^k}|f({\bf z})|<\infty,
  $
 which due to the remarks preceding the proposition  shows that $F\in H^\infty({\bf D}^k)$.
 Conversely, if we assume that $F\in H^\infty({\bf D}^k)$, then $\sup\limits_{0\leq r<1}\|F(r{\bf S})\|<\infty$. If ${\bf z}\in \DD^k$, let $r\in (0,1)$ such that $\frac{1}{r}{\bf z}\in \DD^k$. Applying the Berezin transform ${\boldsymbol\cB}_{\frac{1}{r}{\bf z}}$ to $F(r{\bf S})$, we obtain
 $
 f({\bf z})={\boldsymbol\cB}_{\frac{1}{r}{\bf z}}[F(r{\bf S})]$
 which implies $\sup_{{\bf z}\in \DD^k}|f({\bf z})|\leq \sup\limits_{0\leq r<1}\|F(r{\bf S})\|<\infty.$ Now it is clear that  $\|f\|_\infty=\|F\|_\infty$.

 To prove part (ii),  assume that $\Re f({\bf z})\leq 1$ for any ${\bf z}\in \DD^k$. We use the natural identification of the Hardy spaces $H^2(\DD^k)$ with $H^2(\TT^k)$, and $H^\infty(\DD^k)$ with $H^\infty (\TT^k)$. Under this identification, the shifts $S_1,\ldots, S_n$  are the multiplications by the coordinate functions $\xi_1,\ldots, \xi_k$ on $H^2(\TT^k)$.
Note that, for each $h\in H^2(\TT^k)$, we have
$$
\left<[2I-F(r{\bf S})^*-F(r{\bf S})]h, h\right>_{H^2(\TT^k)}
=\int_{\TT^k} [2-\overline{f(r\boldsymbol\xi)}-f(r\boldsymbol\xi)]|h(\boldsymbol\xi)|^2 dm_k(\boldsymbol\xi)\geq 0,\qquad \boldsymbol\xi\in \TT^k,
$$
where $m_k$ is the normalized Lebesgue measure on $\TT^k$.
Therefore, $\Re F(r{\bf S})\leq I$ for any $r\in [0,1)$.
Conversely, using the Berezin transform at ${\bf z}\in \DD^k$, we have
$$
\overline{f(r{\bf z})}+f(r{\bf z})=\boldsymbol B_{\bf z}[F(r{\bf S})^*-F(r{\bf S})]\leq 2,
$$
for any $r\in [0,1)$ and ${\bf z}\in \DD^k$, which completes the proof of part (ii).

The fact that the  Banach algebra $H^\infty(\DD^k)$ is isometrically isomorphic to  $H^\infty({\bf D}^k)$ follows from item (i). Note that if
$F({\bf Y})= \sum_{{\bf p} \in \ZZ_+^k} a_{\bf p}{\bf Y}^{\bf p}$, ${\bf Y}\in {\bf D}^k$,
is a free holomorphic   function in $H^\infty({\bf D}^k)$, then the function
$G({\bf X}):=\sum_{{\bf p}=(p_1,\ldots, p_k)\in \ZZ_+^k} a_{\bf p}X_{1,1}^{p_1}X_{2,1}^{p_2}\cdots X_{k,1}^{p_k}$,
 where ${\bf X}= \{X_{i,j}\}\in {\bf B_n}(\cH)$,
 is in the noncommutative Hardy algebra $H^\infty({\bf B_n})$ and $\|F\|_\infty=\|G\|_\infty$. Hence, part (iii) follows. The proof is complete.
\end{proof}
We remark that a result similar to Proposition \ref{prelim} holds for free holomorphic functions with coefficients bounded linear operators acting on  a separable Hilbert space.

\section{Bohr inequalities for free holomorphic functions  on  polyballs}

We  define the right (resp.~minimal) sets in $\FF_{n_1}^+\times\cdots \times \FF_{n_k}^+$,  mention some of their properties and give several examples.
The exhaustions of
 $\FF_{n_1}^+\times\cdots \times \FF_{n_k}^+$  by right minimal  sets or by orthogonal sets will play an important role throughout the paper. We associate with each such an exhaustion, a Bohr type inequality for the Hardy space $H^\infty({\bf B_n})$.
 We  obtain Weiner type inequalities for the  coefficients of bounded free holomorphic functions on polyballs, which are used to obtain Bohr inequalities for  bounded free holomorphic functions with operator coefficients.  We also extend a theorem of Bombieri and Bourgain for the disc   to the polyball and obtain estimations for the Bohr radii $K_{mh}({\bf B_n})$ and $K_{mh}^0({\bf B_n})$.

  If $\omega, \gamma\in \FF_n^+$,
we say that $\gamma<_r \omega$ if there is $\sigma\in
\FF_n^+\backslash\{g_0\}$ such that $\omega=\sigma \gamma$. In this
case,  we set $\omega\backslash_r \gamma:=\sigma$.  We also use the notation
$\gamma\leq_r \omega$  when $\gamma<_r \omega$  or $\gamma=\omega$. Note that if $\gamma=\omega$, then $\omega\backslash_r \gamma:=g_0$.
Similarly, we say
that $\gamma
<_{l}\omega$ if there is $\sigma\in
\FF_n^+\backslash\{g_0\}$ such that $\omega= \gamma \sigma$ and set
$\omega\backslash_l \gamma:=\sigma$. The notation $\gamma
\leq_{l}\omega$ is clear.
 We denote by
$\widetilde\alpha$  the reverse of $\alpha\in \FF_n^+$, i.e.
  $\widetilde \alpha= g_{i_k}\cdots g_{i_1}$ if
   $\alpha=g_{i_1}\cdots g_{i_k}\in\FF_n^+$.
  Note that $\gamma\leq_r \omega$ if and only if $\widetilde \gamma\leq_l\widetilde \omega$. In this case, we have
 $ \widetilde{\omega\backslash_r\gamma}=\widetilde\omega \backslash_l\widetilde\gamma$.
Let $\boldsymbol\omega=(\omega_1,\ldots, \omega_k)$ and $\boldsymbol\gamma=(\gamma_1,\ldots, \gamma_k)$ be in $\FF_{n_1}^+\times\cdots \times \FF_{n_k}^+$. We say that $\boldsymbol\omega\leq_r\boldsymbol\gamma$,   if  $\omega_i\leq_r \gamma_i$ for each $i\in \{1,\ldots, k\}$.
Similarly, we say that  $\boldsymbol\omega\leq_l\boldsymbol\gamma$, if
$\widetilde{\boldsymbol\omega}\leq_r \widetilde{\boldsymbol\gamma}$, where
$\widetilde{\boldsymbol\omega}=(\widetilde{\omega_1},\ldots, \widetilde{\omega_k})$.

\begin{definition}
A subset $\Lambda$ of  $\FF_{n_1}^+\times\cdots \times \FF_{n_k}^+$ is called {\it right minimal} if, for any $\boldsymbol\omega, \boldsymbol\gamma\in \Lambda$,
$\boldsymbol\omega\leq_r\boldsymbol\gamma$ implies $\boldsymbol\omega=\boldsymbol\gamma$.
We   say that $\Lambda$ is {\it left minimal} if, for any $\boldsymbol\omega, \boldsymbol\gamma\in \Lambda$,
$\boldsymbol\omega\leq_l\boldsymbol\gamma$ implies $\boldsymbol\omega=\boldsymbol\gamma$.
If a set $\Lambda$ is both left and right minimal, we call it minimal.
\end{definition}

Note that $\Lambda$ is right minimal if and only if
$\widetilde\Lambda:=\{\widetilde\omega:\ \omega \in \Lambda\}$ is left minimal.
Here are a few characterizations of right minimal sets. Since the proof is straightforward, we shall omit it.
\begin{proposition}\label{Prop1}  Let ${\bf S}:=\{{\bf S}_{i,j}\}$ be the universal model of the polyball ${\bf B_n}$, and let $\{e^i_{\alpha}\}_{\alpha\in \FF_{n_i}^+}$ be the orthogonal basis for the Fock space $F^2(H_{n_i})$.
Then the following statements hold.
\begin{enumerate}
\item[(i)]  For $\alpha=g_{i_1}\cdots g_{i_p}\in \FF_n^+$, we denote by $D_\alpha^r$ the set of all right divisors of $\alpha$, i.e.
$$
D_\alpha^r:=\{g_0, g_{i_p}, g_{i_{p-1}}g_{i_p},\ldots, g_{i_1}\cdots g_{i_p}\}.
$$
If $\beta \in \FF_n^+\backslash D_\alpha^r$  and
$\alpha \in \FF_n^+\backslash D_\beta^r$, then  $\{\alpha, \beta\}$ is a right minimal set. Moreover, a set $\Lambda\subset \FF_n^+$ is right minimal if and only if, for any $\alpha, \beta\in \Lambda$ with $\alpha\neq \beta$, we have  $\beta \in \FF_n^+\backslash D_\alpha^r$  and
$\alpha \in \FF_n^+\backslash D_\beta^r$.

\item[(ii)]   $\Lambda\subset \FF_{n_1}^+\times\cdots \times \FF_{n_k}^+$ is a right minimal  set if and only if, for any $\boldsymbol\beta, \boldsymbol\gamma\in \Lambda$ and $\boldsymbol \alpha \in \FF_{n_1}^+\times\cdots \times \FF_{n_k}^+$,
$$\left< e_{\boldsymbol \alpha}\otimes  e_{\boldsymbol \beta}, e_{\boldsymbol \gamma}\right>=1
$$
if and only if $\boldsymbol \alpha=(g_0^1,\ldots, g_0^k)$ and $\boldsymbol \beta=\boldsymbol \gamma$, where $e_{\boldsymbol \beta}:=e^1_{\beta_1}\otimes \cdots \otimes  e^k_{\beta_k}$ if  $\boldsymbol \beta=(\beta_1,\ldots, \beta_k)$.

\end{enumerate}
\end{proposition}

\begin{example} \label{Ex1} The following statements hold.
\begin{enumerate}
\item[(i)]
If $p\in \NN\cup\{0\}$, then
$\Lambda_p:=\{\alpha\in \FF_n^+:\ |\alpha|=p\}$ is a minimal set in $\FF_n^+$.
\item[(ii)]
If $p_1,\ldots, p_k\in \NN\cup\{0\}$ and $\Lambda_{p_i}:=\{\alpha\in \FF_{n_i}^+:\ |\alpha|=p_i\}$, then $\Lambda_{p_1}\times\cdots \times \Lambda_{p_k}$ is a minimal set in  $\FF_{n_1}^+\times\cdots \times \FF_{n_k}^+$.
\item[(iii)] If $p\in \NN\cup\{0\}$, then
$$
\Gamma_p:=\{(\alpha_1,\ldots, \alpha_k)\in \FF_{n_1}^+\times\cdots \times \FF_{n_k}^+: \  |\alpha_1|+\cdots +|\alpha_k|=p\}
$$
is a minimal set in  $\FF_{n_1}^+\times\cdots \times \FF_{n_k}^+$.
\item[(iv)] Any subset of a (left, right) minimal set is a (left, right) minimal set.
\item[(v)] If  $i\in \{1,\ldots, k-1\}$ and $\Lambda\subset \FF_{n_1}^+\times\cdots \times \FF_{n_i}^+$ and
$\Gamma\subset \FF_{n_{i+1}}^+\times\cdots \times \FF_{n_k}^+$ are (left, right)  minimal sets, then so if $\Lambda\times \Gamma$.
\end{enumerate}
\end{example}

\begin{definition} Let ${\bf S}:=\{{\bf S}_{i,j}\}$ be the universal model of the polyball ${\bf B_n}$.
 A set $\Lambda\subset \FF_{n_1}^+\times\cdots \times \FF_{n_k}^+$ is  called {\it orthogonal} if the isometries  ${\bf S}_{\boldsymbol\alpha}$,
$\boldsymbol\alpha\in \Lambda$, have orthogonal ranges in $F^2(H_{n_1})\otimes\cdots\otimes F^2(H_{n_k})$, i.e. ${\bf S}_{\boldsymbol\beta}^* {\bf S}_{\boldsymbol\alpha}=0$ for any $\boldsymbol\alpha, \boldsymbol\beta\in \Lambda$ with
$\boldsymbol\alpha\neq \boldsymbol\beta$.
\end{definition}

Here are a few properties for the orthogonal sets. Since the proof is straightforward, we shall omit it.

\begin{proposition} \label{Prop2} Let $\Lambda$ be a subset of $ \FF_{n_1}^+\times\cdots \times \FF_{n_k}^+$.
\begin{enumerate}
\item[(i)] $\Lambda$ is an orthogonal set if and only if, for any $\boldsymbol\alpha=(\alpha_1,\ldots, \alpha_k), \boldsymbol\beta=(\beta_1,\ldots \beta_k)\in \Lambda$ with
$\boldsymbol\alpha\neq \boldsymbol\beta$,  there is $i\in \{1,\ldots, k\}$ such that $\alpha_i\neq g_0^i$ and $\beta_i$ is not a left divisor of $\alpha_i$, i.e. there is no $\gamma_i\in \FF_{n_i}$ such that $\alpha_i=\beta_i\gamma_i$.
\item[(ii)] If $\Lambda$ is an  orthogonal set, then $\Lambda$ is left minimal.
\item[(iii)] If $\Lambda$ is an orthogonal set and $\widetilde \Lambda=\Lambda$, then $\Lambda$ is minimal.
\end{enumerate}
\end{proposition}

\begin{example}\label{Ex2}  The following statements hold.
\begin{enumerate}
\item[(i)]
If $p\in \NN$, then
$\Lambda_p:=\{\alpha\in \FF_n^+:\ |\alpha|=p\}$ is an orthogonal set in $\FF_n^+$.
\item[(ii)] Let $\Lambda\subset \FF_n^+$ with $\Lambda\neq \{g_0\}$. Then $\Lambda$ is orthogonal if and only it is left minimal.
\item[(iii)]
If $p_1,\ldots, p_k\in \NN\cup\{0\}$ and $\Lambda_{p_i}:=\{\alpha\in \FF_{n_i}^+:\ |\alpha|=p_i\}$, then $\Lambda_{p_1}\times\cdots \times \Lambda_{p_k}$ is an orthogonal  set in  $\FF_{n_1}^+\times\cdots \times \FF_{n_k}^+$ if at least one $p_i\geq 1$.

\item[(iv)] Any subset of an orthogonal set is  orthogonal.
\item[(v)] If  $i\in \{1,\ldots, k-1\}$ and $\Lambda\subset \FF_{n_1}^+\times\cdots \times \FF_{n_i}^+$ and
$\Gamma\subset \FF_{n_{i+1}}^+\times\cdots \times \FF_{n_k}^+$ are orthogonal sets, then so if $\Lambda\times \Gamma$.
\end{enumerate}
\end{example}

We remark that there are minimal and orthogonal sets in $\FF_{n_1}^+\times\cdots \times \FF_{n_k}^+$ which are  infinite. Indeed, it is easy to see that
$$
\Lambda=\{g_2g_1 g_2, g_2g_1^2 g_2, g_2g_1^3 g_2,\ldots\}\subset \FF_n^+
$$
is such a set in $\FF_{n}^+$.
Taking cartesian products of this kind of sets, we obtain   minimal and orthogonal sets in $\FF_{n_1}^+\times\cdots \times \FF_{n_k}^+$, which are infinite.

In what follows we obtain a Weiner type inequality for the coefficients of bounded free holomorhic functions  on the regular  polyball. Without loss of generality, we assume throughout this paper that $\cH$ and $\cK$ are separable Hilbert spaces.

\begin{proposition}\label{W} Let $\Lambda$ be a right minimal  subset of $ \FF_{n_1}^+\times\cdots \times \FF_{n_k}^+$ that does not contain the neutral element ${\bf g}_0=(g_0^1,\ldots, g_0^k)$ and let $F:{\bf B_n}(\cH)\to B(\cK)\otimes_{min}B(\cH)$ be a bounded free holomorphic function with  representation
$$
F({\bf X})=\sum_{\boldsymbol \alpha\in  \FF_{n_1}^+\times\cdots \times \FF_{n_k}^+} A_{(\boldsymbol \alpha)}\otimes  {\bf X}_{\boldsymbol \alpha}
$$
such that $\|F\|_\infty\leq 1$ and $F(0)$ is a scalar operator, i.e. $F(0)=a_0 I$ for some $a_0\in \CC$.
Then
$$
\left\|\sum_{\boldsymbol \alpha\in \Lambda}
A_{(\boldsymbol \alpha) }^*A_{(\boldsymbol \alpha) }\right\|^{1/2}\leq 1-|a_0|^2.
$$
\end{proposition}
\begin{proof}
Let $\{e_{\boldsymbol \alpha}\}_{\boldsymbol \alpha\in \FF_{n_1}^+\times\cdots \times \FF_{n_k}^+}$ be the orthonormal basis of the Hilbert space $F^2(H_{n_1})\otimes \cdots \otimes  F^2(H_{n_k})$ and let $\cE$ be the closed linear span of the vectors $1, e_{\boldsymbol \beta}$, where $\boldsymbol \beta\in \Lambda$. If  $\widehat F:=\text{\rm SOT-}\lim_{r\to 1}F(r{\bf S})$ is the boundary function of $F$ with respect to ${\bf S}$,
  we have
$$\left<P_{\cK\otimes \cE} \widehat F|_{\cK\otimes \cE} (x\otimes 1),y\otimes 1\right>= \lim_{r\to 1}\left<  F(r{\bf S}) x\otimes 1,y\otimes 1\right>=a_0\left<x,y\right>,
$$
$$\left<P_{\cK\otimes \cE} \widehat F|_{\cK\otimes \cE} (x\otimes e_{\boldsymbol \beta}),y\otimes 1\right>= \lim_{r\to 1}\left<  F(r{\bf S}) (x\otimes  e_{\boldsymbol \beta}),y\otimes 1\right>=0,
$$
and
$$\left<P_{\cK\otimes \cE} \widehat F|_{\cK\otimes \cE} (x\otimes 1), y\otimes e_{\boldsymbol \beta}\right>= \lim_{r\to 1}\left<  F(r{\bf S}) (x\otimes 1),y\otimes e_{\boldsymbol \beta}\right>=\left<A_{(\boldsymbol\beta)}x, y\right>,
$$
for any $\boldsymbol \beta\in \Lambda$ and $x,y\in \cK$. Since $\Lambda$ is a right minimal  subset of $ \FF_{n_1}^+\times\cdots \times \FF_{n_k}^+$,  Proposition \ref{Prop1}, part (ii), implies
\begin{equation*}
\begin{split}
\left<P_{\cK\otimes \cE} \widehat F |_{\cK\otimes \cE}(x\otimes  e_{\boldsymbol \beta}), y\otimes e_{\boldsymbol \gamma}\right>&=\lim_{r\to 1}\left<  F(r{\bf S})  (x\otimes e_{\boldsymbol \beta}),y\otimes e_{\boldsymbol \gamma}\right>\\
&=\sum_{\boldsymbol \alpha\in \FF_{n_1}^+\times\cdots \times \FF_{n_k}^+}\left< A_{\boldsymbol \alpha}x, y \right> \left<e_{\boldsymbol\alpha}\otimes e_{\boldsymbol\beta}, e_{\boldsymbol\gamma}\right>=a_0\left<x, y \right> \delta_{\boldsymbol \beta,\boldsymbol \gamma}
\end{split}
\end{equation*}
for any $\boldsymbol \beta,\boldsymbol\gamma\in \Lambda$ and $x,y\in \cK$.
Consequently, the matrix representation of the contraction $P_{\cK\otimes \cE} \widehat F |_{\cK\otimes \cE}$ with respect to the decomposition $$\cK\otimes \cE=(\cK\otimes 1)\oplus \bigoplus_{\boldsymbol \beta\in \Lambda}(\cK\otimes e_{\boldsymbol \beta})$$
is
$$
\left(\begin{matrix}
 a_0 I_\cK& [0\ \ \ \cdots \ \  \ 0]\\
 \left[
 \begin{matrix}A_{(\boldsymbol \alpha)}\\
 \vdots\\
 \boldsymbol \alpha\in \Lambda
 \end{matrix}
 \right]&
 \left[
 \begin{matrix}a_0 I_\cK&\cdots& 0\\
 \vdots&\ddots& \vdots\\
 0&\cdots& a_0 I_\cK
 \end{matrix}
 \right]
 \end{matrix}
 \right).
 $$
According to Lemma 3.4 from  \cite{PPoS}, if $\cH, \cK$ are Hilbert spaces, $B$ is a bounded linear operator from $\cM$ to $\cG$, and $\lambda\in \CC$, then  $\left(\begin{matrix}\lambda I_\cM&0\\B&\lambda I_\cG\end{matrix}\right)$ is a contraction if and only if  $\|B\|\leq 1- |\lambda|^2$. Applying this result to our setting, we deduce that
$$
\left\|\sum_{\boldsymbol \alpha\in \Lambda}
A_{(\boldsymbol \alpha) }^*A_{(\boldsymbol \alpha) }\right\|^{1/2}\leq 1-|a_0|^2.
$$
The proof is complete.
\end{proof}

In what follows, we need some notation.
Let ${\boldsymbol\rho}:=(\rho_1,\ldots, \rho_k)$, $\rho_i:=(\rho_{i,1},\ldots, \rho_{i,n_i})$,  with
    $\rho_{i,j}\geq 0$. We also use the abbreviation $\boldsymbol\rho=(\rho_{i,j})$.   We denote
$\rho_{i,\alpha_i}:=\rho_{i,j_1}\cdots \rho_{i,j_p}$
  if  $\alpha_i=g_{j_1}^i\cdots g_{j_p}^i\in \FF_{n_i}^+$ and
   $\rho_{i,g_0^i}:=1$, and  ${\boldsymbol\rho}_{\boldsymbol\alpha}:= {\rho}_{1,\alpha_1}\cdots {\rho }_{k,\alpha_k}$ if  $\boldsymbol\alpha:=(\alpha_1,\ldots, \alpha_k)\in \FF_{n_1}^+\times \cdots \times\FF_{n_k}^+$.
If ${\bf X}=\{X_{i,j}\}\in B(\cH)^{n_1+\cdots, n_k}$, set $\boldsymbol\rho {\bf X}:=(\rho_{i,j}X_{i,j})$.
   Moreover, if ${\boldsymbol \gamma }:=(\gamma_1,\ldots, \gamma_k)$, $\gamma_i\geq 0$, we set ${\boldsymbol \gamma }{\bf X}:=(\gamma_1{X}_1,\ldots, \gamma_k{X}_k)$, where
    $\gamma_iX_i:=(\gamma_iX_{i,1},\ldots, \gamma_iX_{i,j})$.
 When $r\geq 0$, the notation $r{\bf X}=(rX_{i,j})$ is clear.

An exhaustion of
 $\FF_{n_1}^+\times\cdots \times \FF_{n_k}^+$  by right minimal  sets is a countable collection $\{\Sigma_m\}_{m=0}^\infty$ of non-empty, disjoint, right minimal sets $\Sigma_m$ with $\Sigma_0=\{{\bf g}\}$ and  the property that $\cup_{m=0}^\infty \Sigma_m =\FF_{n_1}^+\times\cdots \times \FF_{n_k}^+$.

\begin{theorem} \label{Bohr1}  Let $F:{\bf B_n}(\cH)\to B(\cK)\otimes_{min}B(\cH)$ be a bounded free holomorphic function with  representation
$$
F({\bf X})=\sum_{\boldsymbol \alpha\in  \FF_{n_1}^+\times\cdots \times \FF_{n_k}^+} A_{(\boldsymbol \alpha)}\otimes  {\bf X}_{\boldsymbol \alpha}
$$
such that $\|F\|_\infty\leq 1$ and $F(0)$ is a scalar operator.
If $\{\Sigma_m\}_{m=0}^\infty$ is an exhaustion of
 $\FF_{n_1}^+\times\cdots \times \FF_{n_k}^+$  by right  minimal  sets and
 $\boldsymbol\rho=(\rho_{i,j})\in [0,1)^{n_1+\cdots n_k}$  is such that
$$
\sum\limits_{m=1}^\infty \left\|\sum_{ \boldsymbol\alpha \in \Sigma_m}
 {\boldsymbol\rho}_{\boldsymbol\alpha}^2{\bf S}_{\boldsymbol \alpha}{\bf S}_{\boldsymbol \alpha}^*\right\|^{1/2}\leq \frac{1}{2},
 $$
then
$$
\sum\limits_{ m=0}^\infty\sup_{{\bf X}\in \boldsymbol\rho {\bf B_n(\cH)}^-}\left\|\sum\limits_ {\boldsymbol\alpha\in \Sigma_m} A_{(\boldsymbol \alpha)} \otimes {\bf X}_{\boldsymbol \alpha}\right\|
\leq
  1.
$$
\end{theorem}
\begin{proof}
  Since $\Sigma_m$ is a right minimal set,  we can use Proposition \ref{W} and deduce that
\begin{equation*}
\begin{split}
\left\|\sum_{\boldsymbol\alpha \in \Sigma_m}{\boldsymbol\rho}_{\boldsymbol\alpha}A_{(\boldsymbol\alpha)}\otimes  {\bf S}_{\boldsymbol\alpha}\right\|
&=
\left\|\left[ {\boldsymbol\rho}_{\boldsymbol\alpha}{\bf S}_{\boldsymbol \alpha}\otimes I:\ \boldsymbol\alpha \in \Sigma_m\right]
\left[\begin{matrix} I\otimes A_{(\boldsymbol\alpha)}\\
:\\
\boldsymbol\alpha\in \Sigma_m\end{matrix}\right]\right\|\\
&\leq
 \left\|\sum_{ \boldsymbol\alpha \in \Sigma_m}
 {\boldsymbol\rho}_{\boldsymbol\alpha}^2{\bf S}_{\boldsymbol \alpha}{\bf S}_{\boldsymbol \alpha}^*\right\|^{1/2} \left\|\sum_{\boldsymbol \alpha\in \Sigma_m}
A_{(\boldsymbol \alpha) }^*A_{(\boldsymbol \alpha) }\right\|^{1/2}\\
 &\leq (1-|a_0|^2)\left\|\sum_{ \boldsymbol\alpha \in \Sigma_m}
 {\boldsymbol\rho}_{\boldsymbol\alpha}^2{\bf S}_{\boldsymbol \alpha}{\bf S}_{\boldsymbol \alpha}^*\right\|^{1/2}.
 \end{split}
\end{equation*}
Now, using the noncommutative von Neumann   inequality for regular  polyballs (see \cite{Po-von}, \cite{Po-poisson}),  we deduce that
\begin{equation*}
\begin{split}
\sum\limits_{ m=0}^\infty\sup_{{\bf X}\in \boldsymbol\rho {\bf B_n(\cH)}^-}\left\|\sum\limits_ {\boldsymbol\alpha\in \Sigma_m} A_{(\boldsymbol \alpha)}\otimes {\bf X}_{\boldsymbol \alpha}\right\|
&\leq
\sum\limits_{ m=0}^\infty \left\|\sum\limits_{\boldsymbol\alpha \in \Sigma_m}  {\boldsymbol\rho}_{\boldsymbol\alpha}A_{(\boldsymbol \alpha)}\otimes {\bf S}_{\boldsymbol \alpha}\right\|\\
&
\leq|a_0|+ (1-|a_0|^2)\sum\limits_{ m=1}^\infty\left\|\sum_{ \boldsymbol\alpha \in \Sigma_m}
 {\boldsymbol\rho}_{\boldsymbol\alpha}^2{\bf S}_{\boldsymbol \alpha}{\bf S}_{\boldsymbol \alpha}^*\right\|^{1/2}\\
 &\leq |a_0|+ \frac{1-|a_0|^2}{2}\leq 1.
\end{split}
\end{equation*}
The proof is complete.
\end{proof}

Acoording to Example \ref{Ex2}, part (iii), if $p_1,\ldots, p_k\in \NN\cup\{0\}$ and $\Lambda_{p_i}:=\{\alpha\in \FF_{n_i}^+:\ |\alpha|=p_i\}$, then $\Lambda_{\bf p}:=\Lambda_{p_1}\times\cdots \times \Lambda_{p_k}$ is an orthogonal  set in  $\FF_{n_1}^+\times\cdots \times \FF_{n_k}^+$ if at least one $p_i\geq 1$.
On the other hand, due to Example \ref{Ex1}, part (ii), the set $\Lambda_{\bf p}$ is minimal.
   Therefore, $\{\Lambda_{\bf p}\}_{{\bf p}\in \ZZ_+^k}$ is an exhaustion of $\FF_{n_1}^+\times\cdots \times \FF_{n_k}^+$ by minimal and orthogonal sets.
The following definition concerns the polyball ${\bf B_n}$, the Hardy space $H^\infty({\bf B_n})$ of all bounded free holomorphic functions on the polyball with scalar coefficients,  and the representation  of its elements $F$ in terms of  {\it multi-homogeneous} polynomials, i.e.
\begin{equation*}
F({\bf X})=\sum\limits_{{\bf p}\in \ZZ_+^k} \left(\sum\limits_ {\boldsymbol\alpha\in \Lambda_{\bf p}} a_{(\boldsymbol \alpha)}  {\bf X}_{\boldsymbol \alpha}\right), \qquad {\bf X}\in {\bf B_n}(\cH),
\end{equation*}
for any Hilbert space $\cH$, where the convergence is in the operator norm topology.
 In this case, the {\it Bohr radius} for the polyball ${\bf B_n}$  is  denoted  by $K_{mh}({\bf B_n})$ and is the largest $r\geq 0$ such that
   \begin{equation*}
\sum\limits_{{\bf p}\in \ZZ_+^k} \left\|\sum\limits_ {\boldsymbol\alpha\in \Lambda_{\bf p}} a_{(\boldsymbol \alpha)}  {\bf X}_{\boldsymbol \alpha}\right\|
\leq \|F\|_\infty, \qquad {\bf X}\in   r {\bf B_n}(\cH)^-,
\end{equation*}
  for any   $F\in H^\infty({\bf B_n})$.
Due to the noncommutative von Neumann inequality, the latter inequality  is equivalent to
  $\sum\limits_{{\bf p}\in \ZZ_+^k} \left\|\sum\limits_ {\boldsymbol\alpha\in \Lambda_{\bf p}} a_{(\boldsymbol \alpha)}  r^{|\boldsymbol\alpha|}{\bf S}_{\boldsymbol \alpha}\right\|
\leq \|F\|_\infty$, where   $|\boldsymbol \alpha|:=|\alpha_1|+\cdots +|\alpha_k|$, if $\boldsymbol\alpha=(\alpha_1,\ldots, \alpha_k)$ is in $\FF_{n_1}^+\times \cdots \times\FF_{n_k}^+$. Now, we obtain the first estimations for the Bohr radius  $K_{mh}({\bf B_n})$.

\begin{theorem} \label{Bohr2}  Let $F:{\bf B_n}(\cH)\to B(\cK)\otimes_{min}B(\cH)$ be a bounded free holomorphic function with  representation
$$
F({\bf X})=\sum_{\boldsymbol \alpha\in  \FF_{n_1}^+\times\cdots \times \FF_{n_k}^+} A_{(\boldsymbol \alpha)}\otimes  {\bf X}_{\boldsymbol \alpha}
$$
such that $\|F\|_\infty\leq 1$ and $F(0)$ is a scalar operator.

 \begin{enumerate}
  \item[(i)]
If $\boldsymbol\rho:=(\rho_{i,j})$ with $\rho_{i,j}=r_i\in [0,1)$ for   $j\in \{1,\ldots, n_i\}$,  then
$$
\sum\limits_{{\bf p}\in \ZZ_+^k} \left\|\sum\limits_ {\boldsymbol\alpha\in \Lambda_{\bf p}} A_{(\boldsymbol \alpha)} \otimes {\bf X}_{\boldsymbol \alpha}\right\|
\leq
 |a_0|+(1-|a_0|^2)\left[\prod_{i=1}^k (1-r_i)^{-1}-1\right]
$$
for any $ {\bf X}\in   \boldsymbol \rho {\bf B_n}(\cH)^-$.
Consequently, if
$\prod_{i=1}^k (1-r_i)\geq \frac{2}{3}$, then
$$
\sum\limits_{{\bf p}\in \ZZ_+^k} \left\|\sum\limits_ {\boldsymbol\alpha\in \Lambda_{\bf p}} A_{(\boldsymbol \alpha)}\otimes  {\bf X}_{\boldsymbol \alpha}\right\|
\leq1, \qquad {\bf X}\in   \boldsymbol \rho {\bf B_n}(\cH)^-.
$$
\item[(ii)] The Bohr radius $K_{mh}({\bf B_n})$ satisfies the inequalities
 $$ 1-\left(\frac{2}{3}\right)^{1/k}\leq K_{mh}({\bf B_n})\leq \frac{1}{3}.
 $$
\end{enumerate}
\end{theorem}
\begin{proof} Let  $\{\Sigma_m\}_{m=0}^\infty$ be an exhaustion of
 $\FF_{n_1}^+\times\cdots \times \FF_{n_k}^+$  by  minimal and orthogonal sets and
  let $\boldsymbol\rho=(\rho_{i,j})\in [0,1)^{n_1+\cdots + n_k}$  be  such that
$
\sum\limits_{m=1}^\infty \sup_{ \boldsymbol \alpha\in \Sigma_m}
 \boldsymbol\rho_{\boldsymbol \alpha} \leq \frac{1}{2}.
 $
 Since $\Sigma_m$ is an orthogonal set, the set
$\{{\bf S}_{\boldsymbol \alpha}\}_{\boldsymbol \alpha\in \Sigma_m}$  consists of  isometries with orthogonal ranges. Consequently, we have
\begin{equation*}
 \left\|\sum_{ \boldsymbol\alpha \in \Sigma_m}
 \boldsymbol\rho_{\boldsymbol \alpha}^2{\bf S}_{\boldsymbol \alpha}{\bf S}_{\boldsymbol \alpha}^*\right\|=\sup_{ \boldsymbol\alpha \in \Sigma_m}
\boldsymbol\rho_{\boldsymbol \alpha}^2.
\end{equation*}
As in the proof of  Theorem \ref{Bohr1},
 in the particular case when the exhaustion $\{\Sigma_m\}$ coincides with $\{\Lambda_{\bf p}\}_{{\bf p}\in \ZZ_+^k}$ and $\rho_{i,j}=r_i\in [0,1)$,   we deduce that
\begin{equation*}
\begin{split}
\sum\limits_{{\bf p}\in \ZZ_+^k} \left\|\sum\limits_ {\boldsymbol\alpha\in \Lambda_{\bf p}} A_{(\boldsymbol \alpha)} \otimes {\bf X}_{\boldsymbol \alpha}\right\|
&\leq |a_0|+ (1-|a_0|^2)\sum\limits_{(p_1,\ldots, p_k)\in \ZZ_+^k\backslash\{0\}} r_1^{p_1}\cdots r_k^{p_k}\\
&=
|a_0|+(1-|a_0|^2)\left[\prod_{i=1}^k (1-r_i)^{-1}-1\right].
\end{split}
\end{equation*}
Consequently, if
$\prod_{i=1}^k (1-r_i)\geq \frac{2}{3}$, then
$$
\sum\limits_{{\bf p}\in \ZZ_+^k} \left\|\sum\limits_ {\boldsymbol\alpha\in \Lambda_{\bf p}} A_{(\boldsymbol \alpha)}\otimes  {\bf X}_{\boldsymbol \alpha}\right\|
\leq1, \qquad {\bf X}\in   \boldsymbol \rho {\bf B_n}(\cH)^-.
$$
To prove part (ii),
take $r_1=\cdots =r_k=r$ with the property that $\prod_{i=1}^k(1-r_i)=(1-r)^k\geq \frac{2}{3}$. Applying part (i) of the theorem to $F\in H^\infty({\bf B_n})$, we deduce that
 $ 1-\left(\frac{2}{3}\right)^{1/k}\leq K_{mh}({\bf B_n})$.
 The second inequality in (ii)   is due to the classical result and  the fact, due to Proposition \ref{prelim}, that each function in $H^\infty(\DD)$ can be seen as a function in $H^\infty({\bf B_n})$. Indeed, we have $K_{mh}({\bf B_n})\leq K_{mh}(\DD)=\frac{1}{3}$.
 The proof is complete.
\end{proof}

The next result for the polydisc is a consequence of Proposition \ref{prelim} and Theorem \ref{Bohr2}.

\begin{corollary}  \label{pdisk} Let
$$
f(z)=\sum_{(p_1,\ldots, p_k)\in \ZZ_+^k}  A_{p_1,\ldots, p_k} z_1^{p_1}\cdots z_k^{p_k},\qquad z=(z_1,\ldots, z_k)\in \DD^k,
$$
 be an operator-valued   analytic function on the  polydisk \, $\DD^k$ such that  $\|f\|_\infty\leq 1$ and $f(0)=a_0I$, $a_0\in \CC$.  Then
$$
\sum_{(p_1,\ldots, p_k)\in \ZZ_+^k} \| A_{p_1,\ldots, p_k}\| r_1^{p_1}\cdots r_k^{p_k}\leq 1
$$
for any $r_i\in [0, 1)$ with $\prod_{i=1}^k (1-r_i)\geq \frac{2}{3}$.
\end{corollary}

Here is an extension of the classical result of Bohr to operator-valued bounded analytic function in the disc.
\begin{corollary} \label{disk} Let $f(z)=\sum_{n=0}^\infty A_n z^n$, $A_n\in B(\cH)$, be an operator-valued bounded analytic function in $\DD$ such that $A_0=a_0 I$, $a_0\in \CC$.
Then
$$\sum_{n=0}^\infty \|A_n\| r^n\leq \|f\|_\infty
$$
 for any $r\in [0, \frac{1}{3}]$, and $\frac{1}{3}$ is the best possible constant. Moreover, the inequality  is strict unless $f$ is a constant.
\end{corollary}
\begin{proof} The first part of the corollary is due to Corollary \ref{pdisk} and the classical result. To prove the second part, assume that $\|f\|_\infty=1$. As in the proof of Theorem \ref{Bohr2}, we have
$$
\sum_{n=0}^\infty \|A_n\| r^n\leq
|a_0|+\frac{1-|a_0|^2}{2}\leq 1=\|f\|_\infty.
$$
If the equality holds, then $|a_0|=1$. Due to Proposition \ref{W}, we have
$\|A_n^*A_n\|^{1/2}\leq 1-|a_0|^2$ if $n\geq 1$. Consequently $A_n=0$ for $n\geq 1$. This completes the proof.
\end{proof}

\begin{proposition} \label{Cro} Let $F:{\bf B_n}(\cH)\to B(\cK)\otimes_{min}B(\cH)$ be a bounded free holomorphic function with  representation
$$
F({\bf X})=\sum_{{\bf p}\in \ZZ_+^k}\sum_{\boldsymbol\alpha\in \Lambda_{\bf p}} A_{(\boldsymbol\alpha)} \otimes {\bf X}_{\boldsymbol\alpha}
$$
such that $\|F\|_\infty\leq 1$ and $F(0)$ is a scalar operator.
If $\boldsymbol\rho:=(\rho_{i,j})$ with $\rho_{i,j}=r_i\in [0,1)$ for   $j\in \{1,\ldots, n_i\}$,   then
$$
\sum\limits_{{\bf p}\in \ZZ_+^k}\sup_{{\bf X}\in \boldsymbol\rho {\bf B_n(\cH)}^-}\left\|\sum\limits_ {\boldsymbol\alpha\in \Lambda_{\bf p}} A_{(\boldsymbol \alpha)} \otimes  {\bf X}_{\boldsymbol \alpha}\right\|
\leq
 C(\boldsymbol\rho)\|F\|_\infty,
$$
where
$$
C(\boldsymbol\rho):=
\begin{cases} 1&\quad \text{if }\  c\leq \frac{1}{2}\\
c+\frac{1}{4c} &\quad \text{if }\  c>\frac{1}{2}
\end{cases}
$$
and $c:=\prod_{i=1}^k (1-r_i)^{-1}-1$.
\end{proposition}
\begin{proof} Without loss of generality, we may assume that $\|F\|_\infty\leq 1$ and, consequently $|a_0|\leq 1$.
Note that $c:=\prod_{i=1}^k (1-r_i)^{-1}-1\geq 0$.
It is easy to see that, if $0\leq c\leq \frac{1}{2}$, then
$\sup \{x+ (1-x^2)c: 0\leq x\leq 1\}\leq 1$. On the other hand, if $c> \frac{1}{2}$, then
$\sup \{x+ (1-x^2)c: 0\leq x\leq 1\}=c+\frac{1}{4c}$.
Consequently, using  Theorem \ref{Bohr2}, part (i), we deduce that
$$
\sum\limits_{{\bf p}\in \ZZ_+^k}\sup_{{\bf X}\in \boldsymbol\rho {\bf B_n(\cH)}^-}\left\|\sum\limits_ {\boldsymbol\alpha\in \Lambda_{\bf p}} A_{(\boldsymbol \alpha)} \otimes {\bf X}_{\boldsymbol \alpha}\right\|
\leq
 |a_0|+(1-|a_0|^2)c\leq C(\boldsymbol \rho),
$$
which completes the proof.
\end{proof}

Let $F\in H^\infty({\bf B_n})$ have the representation
$
F({\bf X} ):=\sum_{{\bf p}\in \ZZ_+^k}\sum_{\boldsymbol\alpha\in \Lambda_{\bf p}} a_{(\boldsymbol\alpha)}  {\bf X}_{\boldsymbol\alpha}
$
and let
$$
\cD(F,r):=\sum\limits_{{\bf p}\in \ZZ_+^k}r^{|{\bf p}|}\left\|\sum\limits_{\boldsymbol\alpha\in \Lambda_{\bf p}} a_{(\boldsymbol\alpha)}  {\bf S}_{\boldsymbol\alpha}\right\|
$$
be the associated majorant series, where ${\bf S}=\{{\bf S}_{ij}\}$ is the universal model of the polyball. Define
$$
d_{\bf B_n}(r):=\sup \frac{\cD(F,r)}{\|F\|_\infty}, \qquad r\in [0,1),
$$
where the supremum is taken over all $F\in H^\infty({\bf B_n})$  with $F$ not identically $0$.
For the open unit disk, Bombieri and Bourgain (see  \cite{BB}) proved  that
 $d_{\DD}(r)\sim \frac{1}{\sqrt{1-r^2}}$ as $r\to 1$.
In what follows we extend their result to the regular polyball ${\bf B_n}$. In particular, the following  result holds for the scalar polydisk $\DD^k$.

\begin{theorem}\label{Bo-Bo} Let
$$
F({\bf X})=\sum\limits_{{\bf p}\in \ZZ_+^k} \left(\sum\limits_ {\boldsymbol\alpha\in \Lambda_{\bf p}} a_{(\boldsymbol \alpha)}   {\bf X}_{\boldsymbol \alpha}\right), \qquad {\bf X}\in {\bf B_n}(\cH),
$$
be any bounded free holomorphic function  on the polyball. If $\boldsymbol\rho:=(\rho_{i,j})$ with $\rho_{i,j}=r_i\in [0,1)$ for   $j\in \{1,\ldots, n_i\}$,   then
$$
\sum\limits_{{\bf p}\in \ZZ_+^k}\sup_{{\bf X}\in \boldsymbol\rho {\bf B_n(\cH)}^-}\left\|\sum\limits_ {\boldsymbol\alpha\in \Lambda_{\bf p}} a_{(\boldsymbol \alpha)} \ {\bf X}_{\boldsymbol \alpha}\right\|
\leq
 K(\boldsymbol\rho)\|F\|_\infty,\qquad F\in H^\infty({\bf B_n}),
$$
where
$$ K(\boldsymbol\rho):=min\left\{ C(\boldsymbol\rho),
\prod_{i=1}^k (1-r^2_i)^{-1/2}\right\}
$$
and $C(\boldsymbol\rho)$ is defined in Proposition \ref{Cro}.
 Moreover, $d_{\bf B_n}(r)$  has the following properties:
 \begin{enumerate}
 \item[(i)] $ 1\leq d_{\bf B_n}(r)\leq \min\left\{c(r),\left(\frac{1}{\sqrt{1-r^2}}\right)^k\right\},
     $
  where
$$
c(r):=
\begin{cases} 1&\quad \text{if }\  0\leq r\leq 1-\left(\frac{2}{3}\right)^{1/k}\\
c+\frac{1}{4c} &\quad \text{if }\  1-\left(\frac{2}{3}\right)^{1/k}<r<1
\end{cases}
$$
and $c:= (1-r)^{-k}-1$.
 \item[(ii)]
   $ d_{\bf B_n}(r)$ behaves asymptotically as  $\left(\frac{1}{\sqrt{1-r^2}}\right)^k$ if $r\to 1$, i.e.
$$
 \lim_{r\to 1} \frac{d_{\bf B_n}(r)}{\left(\frac{1}{\sqrt{1-r^2}}\right)^k}=1.
 $$
 \end{enumerate}
\end{theorem}
\begin{proof}
Let ${\bf p}:=(p_1,\ldots, p_k)\in \ZZ_+^k$ be such that at  least one $p_i\neq 0$.
Since $\Lambda_{\bf p}$ is an orthogonal set, the set
$\{{\bf S}_{\boldsymbol \alpha}\}_{\boldsymbol \alpha\in \Lambda_{\bf p}}$  consists of  isometries with orthogonal ranges. Consequently, we have
\begin{equation*}
 \left\|\sum_{ \boldsymbol\alpha \in \Lambda_{\bf p}}
 \boldsymbol\rho_{\boldsymbol \alpha}^2{\bf S}_{\boldsymbol \alpha}{\bf S}_{\boldsymbol \alpha}^*\right\|=\sup_{ \boldsymbol\alpha \in \Lambda_{\bf p}}
\boldsymbol\rho_{\boldsymbol \alpha}^2=r_1^{2p_1}\cdots r_k^{2p_k},
\end{equation*}
where ${\bf S}=\{{\bf S}_{i,j}\}$ is the universal model of the polyball.
Due to the noncommutative von Neumann inequality and using the Cauchy Schwarz inequality, we obtain
\begin{equation*}
\begin{split}
\sum\limits_{{\bf p}\in \ZZ_+^k}\sup_{{\bf X}\in \boldsymbol\rho {\bf B_n(\cH)}^-}\left\|\sum\limits_ {\boldsymbol\alpha\in \Lambda_{\bf p}} a_{(\boldsymbol \alpha)} \ {\bf X}_{\boldsymbol \alpha}\right\|
&\leq
\sum\limits_{{\bf p}\in \ZZ_+^k} \left\|\sum\limits_ {\boldsymbol\alpha=(\alpha_1,\ldots, \alpha_k)\in \Lambda_{\bf p}} a_{(\boldsymbol \alpha)} r_1^{|\alpha_1|}\cdots r_k^{|\alpha_k|} {\bf S}_{\boldsymbol \alpha}\right\|\\
&\leq
 \sum\limits_{{\bf p}\in \ZZ_+^k} \left\|\sum\limits_ {\boldsymbol\alpha\in \Lambda_{\bf p}} r_1^{2|\alpha_1|}\cdots r_k^{2|\alpha_k|} {\bf S}_{\boldsymbol \alpha}{\bf S}_{\boldsymbol \alpha}^*\right\|^{1/2}\left(\sum\limits_ {\boldsymbol\alpha\in \Lambda_{\bf p}}
 |a_{(\boldsymbol \alpha)}|^2\right)^{1/2}\\
 &\leq
 \left(\sum\limits_{{\bf p}\in \ZZ_+^k} \left\|\sum\limits_ {\boldsymbol\alpha\in \Lambda_{\bf p}} r_1^{2|\alpha_1|}\cdots r_k^{2|\alpha_k|} {\bf S}_{\boldsymbol \alpha}{\bf S}_{\boldsymbol \alpha}^*\right\|\right)^{1/2}
 \left(\sum\limits_{{\bf p}\in \ZZ_+^k}  \sum\limits_ {\boldsymbol\alpha\in \Lambda_{\bf p}}
 |a_{(\boldsymbol \alpha)}|^2\right)^{1/2}\\
 &\leq
 \left(\sum\limits_{{\bf p}\in \ZZ_+^k} r_1^{2p_1}\cdots r_k^{2p_k}\right)^{1/2}\|F\|_2\\
 &\leq
 \prod_{i=1}^k (1-r^2_i)^{-1/2}\|F\|_\infty.
\end{split}
\end{equation*}
Using now Proposition \ref{Cro}, we deduce the inequality in   the theorem.

Now, we prove the second part of the theorem. First, note that $d_{\bf B_n}(r)\geq 1$ for any $r\in [0,1)$.
Part (i) is due to the first part of the theorem when one takes $r_1=\cdots =r_k=r$.  Note that if $0\leq r\leq 1-\left(\frac{2}{3}\right)^{1/k}$, then $d_{\bf B_n}(r)=1$. On the other hand,
it is easy to see that if $r$ is sufficiently close to $1$,  then we have $c(r)>\left(\frac{1}{\sqrt{1-r^2}}\right)^k$.

 To prove part (ii), note that due to the  result of Bombieri and Bourgain, for any $\epsilon >0$, there is $\delta>0$ such that
 $$
 \left| \frac{d_{\DD}(r)}{\frac{1}{\sqrt{1-r^2}}}-1\right|<\epsilon
 $$
 for any $r\in [0,1)$ such that $|r-1|<\delta$. Fix such an $r$. For each $i\in \{1,\ldots, k\}$,  there exists $f_i\in H^\infty(\DD)$ such that
 $$
 \frac{\cD(f_i, r)}{\|f_i\|_\infty}>(1-\epsilon) \frac{1}{\sqrt{1-r^2}}.
 $$
 Hence, we deduce that
 \begin{equation}\label{D}
 \prod_{i=1}^k \frac{\cD(f_i, r)}{\|f_i\|_\infty}>(1-\epsilon)^k \left(\frac{1}{\sqrt{1-r^2}}\right)^k.
 \end{equation}
 Note that the function $f(z_1,\ldots, z_k):=f_1(z_1)\cdots f_k(z_k)$ is in the Hardy space of the polydisk  $H^\infty(\DD^k)$.
 Due to Proposition \ref{prelim},
 the Banach algebra $H^\infty(\DD^k)$ is isometrically isomorphic to  $H^\infty({\bf D}^k)$, which is isometrically embedded in $H^\infty({\bf B_n})$.
 Let $G\in H^\infty({\bf B_n})$ be the corresponding  element to $f$, via   this embedding, and note that
  $G(r{\bf S}):=f_1(r{\bf S}_{1,1})\cdots f_k(r{\bf S}_{k,1})$. A careful examination reveals that
 $$
 \cD(G, r)=\cD(f_1, r)\cdots \cD(f_k, r) \quad \text{ and }\quad \|G\|_\infty=\|f_1\|_\infty \cdots \|f_k\|_\infty.
 $$
 Hence, using inequality \eqref{D} and  part (i) of the theorem, we obtain
 $$
 \left(\frac{1}{\sqrt{1-r^2}}\right)^k\geq d_{\bf B_n}(r)\geq \frac{\cD(G,r)}{\|G\|_\infty}=\prod_{i=1}^k \frac{\cD(f_i, r)}{\|f_i\|_\infty}>(1-\epsilon)^k \left(\frac{1}{\sqrt{1-r^2}}\right)^k
 $$
 for any $r\in [0,1)$ such that $|r-1|<\delta$.
 This  completes the proof.
 \end{proof}

Now, we consider the case when $F(0)=0$.

\begin{corollary} \label{00}
Let $F:{\bf B_n}(\cH)\to B(\cH)$ be a bounded free holomorphic function with $ F(0)=0$ and representation
$$
F({\bf X})=\sum\limits_{{\bf p}\in \ZZ_+^k} \left(\sum\limits_ {\boldsymbol\alpha\in \Lambda_{\bf p}} a_{(\boldsymbol \alpha)}   {\bf X}_{\boldsymbol \alpha}\right), \qquad {\bf X}\in {\bf B_n}(\cH),
$$
and let  $\{\Sigma_m\}_{m=0}^\infty$ be an exhaustion of
 $\FF_{n_1}^+\times\cdots \times \FF_{n_k}^+$  by  minimal and orthogonal sets.
 \begin{enumerate}
 \item[(i)]
 If $\boldsymbol\rho=(\rho_{i,j})\in [0,1)^{n_1+\cdots + n_k}$,
then
$$
\sum\limits_{ m=0}^\infty\sup_{{\bf X}\in \boldsymbol\rho {\bf B_n(\cH)}^-}\left\|\sum\limits_ {\boldsymbol\alpha\in \Sigma_m} a_{(\boldsymbol \alpha)} {\bf X}_{\boldsymbol \alpha}\right\|
\leq
  \left(\sum\limits_{m=1}^\infty \sup_{ \boldsymbol \alpha\in \Sigma_m}
 \boldsymbol\rho_{\boldsymbol \alpha}\right)^{1/2}\|F\|_\infty.
$$
\item[(ii)]
If $\boldsymbol\rho:=(\rho_{i,j})$ with $\rho_{i,j}=r_i\in [0,1)$ for   $j\in \{1,\ldots, n_i\}$,  then
$$
\sum\limits_{{\bf p}\in \ZZ_+^k}\sup_{{\bf X}\in \boldsymbol\rho {\bf B_n(\cH)}^-} \left\|\sum\limits_ {\boldsymbol\alpha\in \Lambda_{\bf p}} a_{(\boldsymbol \alpha)}  {\bf X}_{\boldsymbol \alpha}\right\|
\leq K(\boldsymbol\rho)\|F\|_\infty
$$
where $$ K(\boldsymbol\rho):=min\left\{ \left[\prod_{i=1}^k (1-r_i)^{-1}-1\right],
\left[\prod_{i=1}^k (1-r^2_i)^{-1}-1\right]^{1/2}\right\}.
$$
\item[(iii)]
If $\boldsymbol\rho:=(\rho_{i,j})$ with $\rho_{i,j}=r_i\in [0,1)$ for   $j\in \{1,\ldots, n_i\}$ and $\prod_{i=1}^k (1-r_i^2)\geq \frac{1}{2}$,   then
$$
\sum\limits_{{\bf p}\in \ZZ_+^k}\sup_{{\bf X}\in \boldsymbol\rho {\bf B_n(\cH)}^-}\left\|\sum\limits_ {\boldsymbol\alpha\in \Lambda_{\bf p}} a_{(\boldsymbol \alpha)}  {\bf X}_{\boldsymbol \alpha}\right\|
\leq
 \|F\|_\infty,
$$
In particular, the inequality holds if $r_1=\cdots =r_k \leq \sqrt{1-\left(\frac{1}{2}\right)^{1/k}}$.
\end{enumerate}
\end{corollary}
\begin{proof}The proof of (i) and (ii) is very similar to that of Theorem \ref{Bo-Bo}, but we have to take into account that $f(0)=0$.
Part (iii) follows from part (ii).
\end{proof}

If we replace the Hardy space
$H^\infty({\bf B_n})$ with
the subspace
$$H_0^\infty({\bf B_n}):=\{f\in H^\infty({\bf B_n}) : \ f(0)=0\},$$   the corresponding Bohr radius is denoted by $K_{mh}^0({\bf B_n})$.

\begin{corollary} \label{Kmh0} The Bohr radius $K_{mh}^0({\bf B_n})$ satisfies the inequalities
$$
\sqrt{1-\left(\frac{1}{2}\right)^{1/k}}\leq K_{mh}^0({\bf B_n})\leq \frac{1}{\sqrt{2}}.
$$
\end{corollary}
\begin{proof}
The left hand inequality is due to part (iii) of Corollary \ref{00}.
On the other hand, Proposition \ref{prelim} shows that $H_0^\infty(\DD)$ is isometrically embedded in
 $H_0^\infty({\bf B_n})$  via the map
$$
\sum_{n=1}^\infty a_n z^n\mapsto \sum_{n=1}^\infty a_n X_{1,1}^n.
$$
 Consequently, $K_{mh}^0({\bf B_n})\leq K_{mh}^0(\DD)$.
Since $K_{mh}^0(\DD)=\frac{1}{\sqrt{2}}$ (e.g. \cite{PPoS}), the result follows.
\end{proof}

Now, we consider the general case of arbitrary  free holomorphic functions on polyballs, with operator coefficients. Our Bohr type result in this setting is the following.

\begin{theorem} Let    $F:{\bf B_n}(\cH)\to B(\cK)\otimes_{min}B(\cH)$ be a   free holomorphic function with
 $\|F\|_\infty\leq 1$ and representation
$$
F({\bf X})=\sum_{\boldsymbol \alpha\in  \FF_{n_1}^+\times\cdots \times \FF_{n_k}^+} A_{(\boldsymbol \alpha)}\otimes  {\bf X}_{\boldsymbol \alpha}.
$$
If  $\{\Sigma_m\}_{m=0}^\infty$ is an exhaustion of
 $\FF_{n_1}^+\times\cdots \times \FF_{n_k}^+$  by right minimal  sets, then
 $$
\sum_{\boldsymbol \alpha\in \Sigma_m} A_{(\boldsymbol \alpha)}^*A_{(\boldsymbol \alpha)}\leq I-A_0^*A_0
$$
for any $m\in \NN$.
If, in addition, each $\Sigma_m$ is an  orthogonal set, then
  the following statements hold.
\begin{enumerate}

\item[(i)] For any  $\boldsymbol\rho=(\rho_{i,j})\in [0,1)^{n_1+\cdots n_k}$,
$$
\sup_{{\bf X}\in \boldsymbol\rho{\bf B_n}} \left\|\sum_{\boldsymbol \alpha\in \Sigma_m} A_{(\boldsymbol \alpha)}\otimes {\bf X}_{\boldsymbol \alpha}\right\|\leq \|(I-A_0^* A_0)^{1/2}\|\sup_{\boldsymbol\alpha\in \Sigma_m} \boldsymbol\rho_{\boldsymbol\alpha} .
$$
\item[(ii)] If $\boldsymbol\rho:=(\rho_{i,j})$ with $\rho_{i,j}=r_i\in [0,1)$ for   $j\in \{1,\ldots, n_i\}$,   then
\begin{equation*}
\begin{split}
 \sum\limits_{{\bf p}\in \ZZ_+^k}\sup_{{\bf X}\in  \boldsymbol\rho {\bf B_n(\cH)}^-} \left\|\sum\limits_ {\boldsymbol\alpha\in \Lambda_{\bf p}} A_{(\boldsymbol \alpha)}\otimes   {\bf X}_{\boldsymbol \alpha}\right\|
 &\leq \|A_0\|+\|(I-A_0^*A_0)^{1/2}\|\left[\prod_{i=1}^k (1-r_i)^{-1}-1\right].
\end{split}
\end{equation*}
\end{enumerate}

\end{theorem}
\begin{proof}
Let $\cE$ be the closed linear span of the vectors $1, e_{\boldsymbol \beta}$, where $\boldsymbol \beta\in \Sigma_m$. As in the proof of Proposition \ref{W}
,
taking into account that $\Sigma_m$ is a  right minimal  subset of $\FF_{n_1}^+\times\cdots \times \FF_{n_k}^+$, one  can show that
 the operator matrix   of $P_{\cK\otimes \cE}F(r{\bf S})_{\cK\otimes \cE}$, $r\in [0,1)$,  with respect to the decomposition
$$\cK\otimes \cE=(\cK\otimes 1)\oplus \bigoplus_{\boldsymbol \beta\in \Sigma_m}(\cK\otimes e_{\boldsymbol \beta})$$
is
$$
\left(\begin{matrix}
  A_{0}& [0\quad \  \cdots   \quad \  0]\\
 \left[
 \begin{matrix} r^{|\boldsymbol \alpha|}A_{(\boldsymbol \alpha)}\\
 \vdots\\
 \boldsymbol \alpha\in \Sigma_m
 \end{matrix}
 \right]&
 \left[
 \begin{matrix}  A_{0}&\cdots& 0\\
 \vdots&\ddots& \vdots\\
 0&\cdots&  A_{0}
 \end{matrix}
 \right]
 \end{matrix}
 \right).
$$
Since $P_{\cK\otimes \cE}F(r{\bf S})_{\cK\otimes \cE}$ is a contraction, so is the operator matrix $\left[
 \begin{matrix}A_0\\ r^{|\boldsymbol \alpha|}A_{(\boldsymbol \alpha)}\\
 \vdots\\
 \boldsymbol \alpha\in \Sigma_m
 \end{matrix}
 \right].$
  Hence, it is clear that
 $$
\sum_{\boldsymbol \alpha\in \Sigma_m} A_{(\boldsymbol \alpha)}^*A_{(\boldsymbol \alpha)}\leq I-A_0^*A_0.
$$
Now we assume, in addition, that each $\Sigma_m$ is  an orthogonal set, i.e. $\{{\bf S}_{\boldsymbol \alpha}\}_{\boldsymbol\alpha\in \Sigma_m}$ is  a sequence of isometries with orthogonal ranges. As in the proof of Theorem \ref{Bohr2}, we can use part (i) of the theorem to deduce the   inequality of part (ii).
\end{proof}

\bigskip

\section{Bohr inequalities  for free  holomorphic functions with $F(0)\geq 0$ and $\Re f\leq I$}

In this section,   we obtain an analogue of Landau's  inequality \cite{LG} for  bounded free holomorhic functions with operator coefficients on the polyball  and   use it  to obtain Bohr inequalities for free holomorphic functions on polyballs with operator coefficients
  such that $ F(0)\geq 0$ and  $\Re F({\bf X})\leq I$. The results play an important role in the next sections.

\begin{theorem} \label{W-O} Let $\Lambda$ be a right minimal  subset of $ \FF_{n_1}^+\times\cdots \times \FF_{n_k}^+$ that does not contain the neutral element $(g_0^1,\ldots, g_0^k)$ and let $F:{\bf B_n}(\cH)\to B(\cK)\otimes_{min}B(\cH)$ be a   free holomorphic function with representation
$$
F({\bf X})=\sum_{\boldsymbol \alpha\in  \FF_{n_1}^+\times\cdots \times \FF_{n_k}^+} A_{(\boldsymbol \alpha)}\otimes  {\bf X}_{\boldsymbol \alpha}
$$
  such that $ F(0)\geq 0$ and  $\Re F({\bf X})\leq I$ for any ${\bf X}\in {\bf B_n}(\cH)$.
Then
the operator matrix
$$
 P_{\Lambda}:=\left(\begin{matrix}
 2(I-A_{0})& [A_{(\boldsymbol \alpha)}^*:\ \boldsymbol \alpha\in \Lambda]\\
 \left[
 \begin{matrix}A_{(\boldsymbol \alpha)}\\
 \vdots\\
 \boldsymbol \alpha\in \Lambda
 \end{matrix}
 \right]&
 \left[
 \begin{matrix} 2(I-A_{0})&\cdots& 0\\
 \vdots&\ddots& \vdots\\
 0&\cdots& 2(I-A_{0})
 \end{matrix}
 \right]
 \end{matrix}
 \right)
$$
is positive and
 $$
\sum_{\boldsymbol \alpha\in \Lambda} A_{(\boldsymbol \alpha)}^*A_{(\boldsymbol \alpha)}\leq 4\|I-A_0\|(I-A_0).
$$
If, in addition,  $\Lambda$ is an orthogonal set, then
$$\sup_{{\bf X}\in {\bf B_n}} \left\|\sum_{\boldsymbol \alpha\in \Lambda} A_{(\boldsymbol \alpha)}\otimes {\bf X}_{\boldsymbol \alpha}\right\|= \left\|\sum_{\boldsymbol \alpha\in \Lambda} A_{(\boldsymbol \alpha)}\otimes {\bf S}_{\boldsymbol \alpha}\right\|\leq 2\|I-A_0\|.$$
\end{theorem}

\begin{proof}
First, we recall that $F$ is a free holomorphic function on the polyball ${\bf B_n}$ if and only if
  the series
$$
\sum_{q=0}^\infty \sum_{{\boldsymbol\alpha\in \FF_{n_1}^+\times \cdots \times\FF_{n_k}^+ }\atop {|\boldsymbol\alpha|=q}} A_{(\boldsymbol\alpha)} \otimes r^q {\bf S}_{\boldsymbol\alpha}
$$
is convergent in the operator norm topology for any $r\in [0,1)$. This shows that $F(r{\bf S})$ exists and it is in the polyball algebra $\boldsymbol\cA_{\bf n}$. In particular, this implies for each $r\in [0,1)$, the series $\sum_{\boldsymbol\alpha\in \FF_{n_1}^+\times \cdots \times\FF_{n_k}^+}r^{|\boldsymbol \alpha|}A_{(\boldsymbol\alpha)}^*A_{(\boldsymbol\alpha)}$ is convergent in the strong operator topology.
 Note that  $\Re F({\bf X})\leq I$ for any ${\bf X}\in {\bf B_n}(\cH)$ if and only if
$$
2I-F(r{\bf S})-F(r{\bf S})^*\geq 0
$$
for any $r\in [0,1)$.  Indeed,   using the fact that the noncommutative Berezin transform on the polyball ${\bf B_n}$ is a completely positive map and
$$
2I-F({\bf X})-F({\bf X})^*=(id\otimes {\boldsymbol\cB}_{\frac{1}{r}{\bf X}})[2I-F(r{\bf S})-F(r{\bf S})^*],
$$
where $r\in [0,1)$ is such that $\frac{1}{r}{\bf X}\in {\bf B_n}(\cH)$, the result follows.
Note also that $F(0)=A_0\otimes I\geq 0$, thus $A_0\geq 0$.

Let $\cE$ be the closed linear span of the vectors $1, e_{\boldsymbol \beta}$, where $\boldsymbol \beta\in \Lambda$. Setting
$G({\bf X}):=2I-F({\bf X})-F({\bf X})^*$ and taking into account that $\{e_{\boldsymbol \alpha}\}_{\boldsymbol \alpha\in \FF_{n_1}^+\times\cdots \times \FF_{n_k}^+}$ is an orthonormal basis for the Hilbert tensor product $F^2(H_{n_1})\otimes \cdots \otimes  F^2(H_{n_k})$, we have
$$\left<P_{\cK\otimes\cE} G(r{\bf S})|_{\cK\otimes\cE} (x\otimes 1),y\otimes 1\right>=  \left< 2(I-A_0)x,y\right>,
$$
$$\left<P_{\cK\otimes\cE} G(r{\bf S})|_{\cK\otimes\cE} (x\otimes e_{\boldsymbol \beta}),y\otimes 1\right>= -r^{|\boldsymbol \beta|} \left< A_{(\boldsymbol \beta)}^*x,y\right>,
$$
and
$$\left<P_{\cK\otimes\cE} G(r{\bf S})|_{\cK\otimes\cE} (x\otimes 1),y\otimes e_{\boldsymbol \beta}\right>= -r^{|\boldsymbol \beta|} \left< A_{(\boldsymbol \beta)}x,y\right>
$$
for any $\boldsymbol \beta\in \Lambda$, where $|\boldsymbol \beta|:=|\beta_1|+\cdots +|\beta_k|$, if $\boldsymbol\beta=(\beta_1,\ldots, \beta_k)\in F_{n_1}^+\times \cdots \times \FF_{n_k}^+$. Since $\Lambda$ is a right minimal  subset of $ \FF_{n_1}^+\times\cdots \times  \FF_{n_k}^+$,  Proposition \ref{Prop1}, part (ii), implies
\begin{equation*}
\begin{split}
\left<P_{\cK\otimes\cE} G(r{\bf S})|_{\cK\otimes\cE} (x\otimes e_{\boldsymbol\beta}),y\otimes e_{\boldsymbol \gamma}\right>
&=
2\delta_{\boldsymbol \beta\boldsymbol \gamma}\left<x, y\right>-\sum_{\boldsymbol\alpha\in \FF_{n_1}^+\times\cdots \times \FF_{n_k}^+}r^{|\boldsymbol \alpha|}\left<A_{(\boldsymbol\alpha)}x,y\right>\left< e_{\boldsymbol \alpha}\otimes  e_{\boldsymbol \beta}, e_{\boldsymbol \gamma}\right>\\
&\qquad \qquad
-\sum_{\boldsymbol\alpha\in \FF_{n_1}^+\times\cdots \times \FF_{n_k}^+}r^{|\boldsymbol \alpha|}\left<A_{(\boldsymbol\alpha)}^*x,y\right>\left<  e_{\boldsymbol \beta},  e_{\boldsymbol \alpha}\otimes e_{\boldsymbol \gamma}\right>\\
&=2\delta_{\boldsymbol \beta\boldsymbol \gamma}\left<x, y\right>
-\delta_{\boldsymbol \beta\boldsymbol \gamma}\left<A_0x, y\right>-
\delta_{\boldsymbol \beta\boldsymbol \gamma}\left<A_0^*x, y\right>\\
&=
\left<2(I-A_0)x, y\right>\delta_{\boldsymbol \beta\boldsymbol \gamma}
\end{split}
\end{equation*}
for any $\boldsymbol \beta,\boldsymbol\gamma\in \Lambda$.
Consequently, the matrix of $G(r{\bf S})$ with respect to the decomposition
$$\cK\otimes \cE=(\cK\otimes 1)\oplus \bigoplus_{\boldsymbol \beta\in \Lambda}(\cK\otimes e_{\boldsymbol \beta})$$
is
$$ Q_\Lambda(r):=
\left(\begin{matrix}
 2(I-A_{0})& [-r^{|\boldsymbol \alpha|}A_{(\boldsymbol \alpha)}^*:\ \boldsymbol \alpha\in \Lambda]\\
 \left[
 \begin{matrix}-r^{|\boldsymbol \alpha|}A_{(\boldsymbol \alpha)}\\
 \vdots\\
 \boldsymbol \alpha\in \Lambda
 \end{matrix}
 \right]&
 \left[
 \begin{matrix} 2(I-A_{0})&\cdots& 0\\
 \vdots&\ddots& \vdots\\
 0&\cdots& 2(I-A_{0})
 \end{matrix}
 \right]
 \end{matrix}
 \right)
$$
and it is  positive  for any $r\in [0,1)$. Consider the unitary operator
$U:=\left(\begin{matrix} I_\cK&0\\
0&-I_{\cK'}\end{matrix}\right)$, where $\cK':= \bigoplus_{\boldsymbol \beta\in \Lambda}(\cK\otimes e_{\boldsymbol \beta})$, and note that $P_\Lambda(r):=U^*Q_\Lambda(r) U\geq 0$, $r\in [0,1)$.
Let $\{V_{(\boldsymbol \alpha)}\}_{\boldsymbol\alpha\in \Lambda}$ by a sequence of isometries with orthogonal ranges acting on a Hilbert space $\cG$. Applying Lemma 2.2 from \cite{Po-Bohr} to our setting, we deduce that $P_\Lambda(r)\geq 0$ if and only if the operator matrix
\begin{equation}
\label{V}
\left(\begin{matrix} I_\cG\otimes 2(I-A_0)& \sum_{\boldsymbol \alpha\in \Lambda}
V_{(\boldsymbol \alpha)}\otimes r^{|\boldsymbol \alpha|}A_{(\boldsymbol \alpha)}\\
\sum_{\boldsymbol \alpha\in \Lambda}
V_{(\boldsymbol \alpha)}^*\otimes r^{|\boldsymbol \alpha|}A_{(\boldsymbol \alpha)}^*& I_\cG\otimes 2(I-A_0)
\end{matrix}
\right)
\end{equation}
is positive.
It is well-known (see eg. \cite{Pa-book}) that if $P,M$ are bounded operators on a Hilbert space $\cM$ such that   $\left(\begin{matrix}P&M\\M^*&P\end{matrix}
\right)\geq 0$, then $M^*M\leq \|P\|P$. Applying this result to the matrix \eqref{V}, we deduce that
$$
\left(\sum_{\boldsymbol \alpha\in \Lambda}
V_{(\boldsymbol \alpha)}^*\otimes r^{|\boldsymbol \alpha|}A_{(\boldsymbol \alpha)}^*\right)\left(\sum_{\boldsymbol \alpha\in \Lambda}
V_{(\boldsymbol \alpha)}\otimes r^{|\boldsymbol \alpha|}A_{(\boldsymbol \alpha)}\right)\leq \|I_\cG\otimes 2(I-A_0)\|I_\cG\otimes 2(I-A_0),
$$
which, taking into account that $V_{(\boldsymbol \alpha)}^*V_{(\boldsymbol \beta)}=\delta_{\boldsymbol \alpha\boldsymbol \beta}I$ for any $\boldsymbol \alpha,\boldsymbol \beta\in \Lambda$, is equivalent to
$$
I_\cG\otimes \sum_{\boldsymbol \alpha\in \Lambda}r^{2|\boldsymbol \alpha|} A_{(\boldsymbol \alpha)}^*A_{(\boldsymbol \alpha)}\leq I_\cG\otimes 4\|I-A_0\|(I-A_0).
$$
Hence, taking $r\to 1$ we deduce that
\begin{equation}
\label{AAAA}
\sum_{\boldsymbol \alpha\in \Lambda} A_{(\boldsymbol \alpha)}^*A_{(\boldsymbol \alpha)}\leq 4\|I-A_0\|(I-A_0)
\end{equation}
and $P_\Lambda=U^*Q_\Lambda(1) U\geq 0$, which proves the first part of the theorem.
To prove the last part of the theorem, we assume that $\Lambda$ is, in addition,  an orthogonal set, i.e. $\{{\bf S}_{\boldsymbol \alpha}\}_{\boldsymbol\alpha\in \Lambda}$ is  a sequence of isometries with orthogonal ranges. Due to the von Neumann type inequality for regular polyballs \cite{Po-poisson} and using relation \eqref{AAAA},   we have
\begin{equation*}
\begin{split}
\left\|\sum_{\boldsymbol \alpha\in \Lambda_0} A_{(\boldsymbol \alpha)}\otimes {\bf X}_{\boldsymbol \alpha}\right\|
&\leq \left\|\sum_{\boldsymbol \alpha\in \Lambda_0} A_{(\boldsymbol \alpha)}\otimes r^{|\boldsymbol\alpha|}{\bf S}_{\boldsymbol \alpha}\right\|\\
&=\left\|[I\otimes r^{|\boldsymbol\alpha|}{\bf S}_{\boldsymbol \alpha}: \ {\boldsymbol \alpha}\in \Lambda_0]\left[\begin{matrix} A_{(\boldsymbol \alpha)}\otimes I\\:\\ \boldsymbol \alpha\in \Lambda_0\end{matrix}\right]\right\|\\
&\leq \left\|\text{\rm diag}_{\Lambda_0}(r^{2|\boldsymbol\alpha|}I)\right\|^{1/2}
\left\|\sum_{\boldsymbol \alpha\in \Lambda_0} A_{(\boldsymbol \alpha)}^*A_{(\boldsymbol \alpha)}\right\|^{1/2}\\
&=\left\|\sum_{\boldsymbol \alpha\in \Lambda_0} A_{(\boldsymbol \alpha)}^*A_{(\boldsymbol \alpha)}\right\|^{1/2}\sup_{\boldsymbol\alpha\in \Lambda_0}r^{|\boldsymbol\alpha|}\\
&\leq 2\|I-A_0\| \, r^{\inf_{\boldsymbol \alpha\in \Lambda_0}|\boldsymbol \alpha|},
\end{split}
\end{equation*}
for any ${\bf X}\in r{\bf B_n}(\cH)$, $r\in [0,1)$, and any finite subset $\Lambda_0\subset \Lambda$.
Consequently,
$$
\left\|\sum_{\boldsymbol \alpha\in \Lambda} A_{(\boldsymbol \alpha)}\otimes {\bf X}_{\boldsymbol \alpha}\right\|\leq 2\|I-A_0\|$$
for any ${\bf X}\in {\bf B_n}(\cH)$, which completes the proof.
\end{proof}

We use Theorem \ref{W-O} the following Bohr type result in a very general setting.

\begin{theorem}  \label{CAX} Let   $F:{\bf B_n}(\cH)\to B(\cK)\otimes_{min}B(\cH)$ be a   free holomorphic function with representation
$
F({\bf X})=\sum_{\boldsymbol \alpha\in  \FF_{n_1}^+\times\cdots \times \FF_{n_k}^+} A_{(\boldsymbol \alpha)}\otimes  {\bf X}_{\boldsymbol \alpha}
$
  such that $ F(0)\geq 0$ and  $\Re F({\bf X})\leq I$ for any ${\bf X}\in {\bf B_n}(\cH)$.
  If  $\{\Sigma_m\}_{m=0}^\infty$ is an exhaustion of
 $\FF_{n_1}^+\times\cdots \times \FF_{n_k}^+$  by  minimal and orthogonal sets, then the following statements hold.
  \begin{enumerate}
  \item[(i)] If $\{C_{(\boldsymbol\alpha)}\}_{\boldsymbol \alpha\in    \FF_{n_1}^+\times\cdots \times \FF_{n_k}^+}$ is a sequence of bounded linear operators on a Hilbert space $\cE$ such that
      $$
      \|C_0\|\leq 1 \quad \text{ and } \quad \sum_{m=1}^\infty \sup_{\alpha\in \Sigma_m}\|C_{(\boldsymbol\alpha)}\|\leq \frac{1}{2},
      $$
      then $$
      \sum_{m=0}^\infty  \sup_{{\bf X}\in   {\bf B_n}(\cH)^-} \left\|\sum_{\boldsymbol \alpha\in\Sigma_m}C_{(\boldsymbol\alpha)}\otimes  A_{(\boldsymbol \alpha)}\otimes {\bf X}_{\boldsymbol \alpha}\right\|\leq \|A_0\|+\|1-A_0\|.$$
      \item[(ii)]
      If
      $$\|C_0\|\leq 1\quad \text{  and }\quad
      \sum_{m=1}^\infty\left\|\sum_{\alpha\in \Sigma_m}
      C_{(\boldsymbol\alpha)}^*C_{(\boldsymbol\alpha)}\right\|^{1/2}\leq \frac{1}{2},$$
      then
      $$
      \left\|\sum_{m=0}^\infty\sum_{\boldsymbol \alpha\in \Sigma_m}C_{(\boldsymbol\alpha)}\otimes  A_{(\boldsymbol \alpha)}\otimes {\bf S}_{\boldsymbol \alpha}\right\|\leq 1.
      $$

\end{enumerate}
\end{theorem}

\begin{proof}

Let  $\cF$  be a finite subset of $\Sigma_m$.
According to Theorem \ref{W-O}, we have
$$
\sum_{\boldsymbol \alpha\in \cF} A_{(\boldsymbol \alpha)}^*A_{(\boldsymbol \alpha)}\leq 4\|I-A_0\|(I-A_0).
$$
Hence, and  using the noncommutative von Neumann inequality,  we deduce that
\begin{equation*}
\begin{split}
\left\|\sum_{\boldsymbol \alpha\in \cF} C_{(\boldsymbol \alpha)}\otimes A_{(\boldsymbol \alpha)}\otimes {\bf X}_{\boldsymbol \alpha}\right\|
&\leq  \left\|[C_{(\boldsymbol \alpha)}\otimes I\otimes {\bf S}_{\boldsymbol \alpha}: \ {\boldsymbol \alpha}\in \cF]\left[\begin{matrix}I\otimes  A_{(\boldsymbol \alpha)}\otimes I\\:\\ \boldsymbol \alpha\in \cF\end{matrix}\right]\right\|\\
&\leq \left\|\text{\rm diag}_{\cF}(C_{(\boldsymbol \alpha)}^*C_{(\boldsymbol \alpha)})\right\|^{1/2}
\left\|\sum_{\boldsymbol \alpha\in \cF} A_{(\boldsymbol \alpha)}^*A_{(\boldsymbol \alpha)}\right\|^{1/2}\\
&=\left\|\sum_{\boldsymbol \alpha\in \cF} A_{(\boldsymbol \alpha)}^*A_{(\boldsymbol \alpha)}\right\|^{1/2}\sup_{\boldsymbol\alpha\in \cF}\|C_{(\boldsymbol \alpha)}\|\\
&\leq 2\|I-A_0\| \, \sup_{\boldsymbol\alpha\in \cF}\|C_{(\boldsymbol \alpha)}\|.
\end{split}
\end{equation*}
Consequently, $\left\|\sum_{\boldsymbol \alpha\in \Sigma_m} C_{(\boldsymbol \alpha)}\otimes A_{(\boldsymbol \alpha)}\otimes {\bf X}_{\boldsymbol \alpha}\right\|\leq
 2\|I-A_0\| \, \sup_{\boldsymbol\alpha\in \Sigma_m}\|C_{(\boldsymbol \alpha)}\|
$ for any $m\in \NN$.
Now, using the hypothesis, we have
\begin{equation*}
\begin{split}
\sum_{m=0}^\infty  \sup_{{\bf X}\in   {\bf B_n}(\cH)^-} \left\|\sum_{\boldsymbol \alpha\in\Sigma_m}C_{(\boldsymbol\alpha)}\otimes  A_{(\boldsymbol \alpha)}\otimes {\bf X}_{\boldsymbol \alpha}\right\|&\leq \|A_0\|+2\|1-A_0\|\sum_{m=1}^\infty \sup_{\alpha\in \Sigma_m}\|C_{(\boldsymbol\alpha)}\|\\
&\leq\|A_0\|+\|1-A_0\|.
\end{split}
\end{equation*}
To prove part (ii), denote  $d_{\Sigma_m}:=\left\|[C_{(\boldsymbol\alpha)}^*:\ \boldsymbol\alpha\in \Sigma_m]\right\|$ for $m\in \NN$, and note that
$$ Q_{\Lambda_m}:=
\left(\begin{matrix}
 d_{\Sigma_m}I_\cE& [C_{(\boldsymbol \alpha)}^*:\ \boldsymbol \alpha\in \Sigma_m]\\
 \left[
 \begin{matrix}C_{(\boldsymbol \alpha)}\\
 \vdots\\
 \boldsymbol \alpha\in \Sigma_m
 \end{matrix}
 \right]&
 \left[
 \begin{matrix} d_{\Sigma_m}I_\cE&\cdots& 0\\
 \vdots&\ddots& \vdots\\
 0&\cdots& d_{\Sigma_m}I_\cE
 \end{matrix}
 \right]
 \end{matrix}
 \right)
$$
is a positive operator matrix. Since $P_{\Sigma_m}\geq 0$, due to Theorem \ref{W-O}, we have $P_{\Sigma_m}\otimes Q_{\Sigma_m}\geq 0$. Compressing the operator matrix $P_{\Sigma_m}\otimes Q_{\Sigma_m}$ to the appropriate entries, we obtain that the operator matrix
\begin{equation}\label{2d}
\begin{split}
\left(\begin{matrix}
 I_\cE\otimes 2d_{\Sigma_m}(I-A_0)& [C_{(\boldsymbol \alpha)}^*\otimes A_{(\boldsymbol \alpha)}^*:\ \boldsymbol \alpha\in \Sigma_m]\\
 \left[
 \begin{matrix}C_{(\boldsymbol \alpha)}\otimes A_{(\boldsymbol \alpha)}\\
 \vdots\\
 \boldsymbol \alpha\in \Sigma_m
 \end{matrix}
 \right]&
 \left[
 \begin{matrix} I_\cE\otimes 2d_{\Sigma_m}(I-A_0)&\cdots& 0\\
 \vdots&\ddots& \vdots\\
 0&\cdots& I_\cE\otimes 2d_{\Sigma_m}(I-A_0)
 \end{matrix}
 \right]
 \end{matrix}
 \right)
\end{split}
\end{equation}
is positive. Since $\Sigma_m$ is an orthogonal set, the isometries $\{{\bf S}_{\boldsymbol\alpha}\}_{\boldsymbol\alpha\in \Sigma_m}$ have orthogonal ranges.
Applying  now Lemma 2.2 from \cite{Po-Bohr} to the operator matrix \eqref{2d}, we deduce that
\begin{equation}
\label{V2}
\left(\begin{matrix} I_\cE\otimes 2d_{\Sigma_m}(I-A_0)\otimes I & \sum_{\boldsymbol \alpha\in \Sigma_m}
C_{(\boldsymbol \alpha)}\otimes A_{(\boldsymbol \alpha)}\otimes {\bf S}_{\boldsymbol\alpha}\\
\sum_{\boldsymbol \alpha\in \Sigma_m}
C_{(\boldsymbol \alpha)}^*\otimes A_{(\boldsymbol \alpha)}^*\otimes {\bf S}_{\boldsymbol\alpha}^*& I_\cE\otimes 2d_{\Sigma_m}(I-A_0)\otimes I
\end{matrix}
\right)
\end{equation}
is positive. Since $0\leq I-A_0\leq I$, we deduce that
\begin{equation*}
\left(\begin{matrix} 2d_{\Sigma_m}I_\cE\otimes  I_\cH\otimes I & \sum_{\boldsymbol \alpha\in \Sigma_m}
C_{(\boldsymbol \alpha)}\otimes A_{(\boldsymbol \alpha)}\otimes {\bf S}_{\boldsymbol\alpha}\\
\sum_{\boldsymbol \alpha\in \Sigma_m}
C_{(\boldsymbol \alpha)}^*\otimes A_{(\boldsymbol \alpha)}^*\otimes {\bf S}_{\boldsymbol\alpha}^*& 2d_{\Sigma_m}I_\cE\otimes I_\cH\otimes I
\end{matrix}
\right)\geq 0,
\end{equation*}
which implies
$$
\left\|\sum_{\boldsymbol \alpha\in \Sigma_m}
C_{(\boldsymbol \alpha)}\otimes A_{(\boldsymbol \alpha)}\otimes {\bf S}_{\boldsymbol\alpha}\right\|\leq d_{\Sigma_m}, \qquad m\in \NN.
$$
Due to the hypothesis, we have $\sum_{m=1}^\infty \sum_{\boldsymbol \alpha\in \Sigma_m} d_{\Sigma_m}\leq \frac{1}{2}$, which shows that the series
$$
\sum_{m=1}^\infty \left\|\sum_{\boldsymbol \alpha\in \Sigma_m}
C_{(\boldsymbol \alpha)}\otimes A_{(\boldsymbol \alpha)}\otimes {\bf S}_{\boldsymbol\alpha}\right\|
$$
is convergent.
Since $A_0\geq 0$ and $\left(\begin{matrix}I_\cE&C_0\\C_0^*&I_\cE\end{matrix}
\right)\geq 0$
we also have
$\left(\begin{matrix}I_\cE\otimes A_0\otimes I&C_0\otimes A_0\otimes I\\C_0^*\otimes A_0\otimes I&I_\cE\otimes A_0\otimes I\end{matrix}
\right)\geq 0$. Taking the sum of the matrices \eqref{V2} and the latter one, we obtain

\begin{equation*}
\left(\begin{matrix} I_\cE\otimes \left[A_0+\sum_{m=1}^\infty 2d_{\Sigma_m}(I-A_0)\right]\otimes I& \sum_{m=0}^\infty\sum_{\boldsymbol \alpha\in \Sigma_m}
C_{(\boldsymbol \alpha)}\otimes A_{(\boldsymbol \alpha)}\otimes {\bf S}_{\boldsymbol\alpha}\\
\sum_{m=0}^\infty\sum_{\boldsymbol \alpha\in \Sigma_m}
C_{(\boldsymbol \alpha)}^*\otimes A_{(\boldsymbol \alpha)}^*\otimes {\bf S}_{\boldsymbol\alpha}^*& I_\cE\otimes \left[A_0+\sum_{m=1}^\infty2d_{\Sigma_m}(I-A_0)\right]\otimes I
\end{matrix}
\right)\geq 0.
\end{equation*}
Since $0\leq A_0+\sum_{m=1}^\infty 2d_{\Sigma_m}(I-A_0)\leq I$, we deduce that
\begin{equation*}
\left(\begin{matrix} I_\cE\otimes I_\cK\otimes I& \sum_{m=0}^\infty\sum_{\boldsymbol \alpha\in \Sigma_m}
C_{(\boldsymbol \alpha)}\otimes A_{(\boldsymbol \alpha)}\otimes {\bf S}_{\boldsymbol\alpha}\\
\sum_{m=0}^\infty\sum_{\boldsymbol \alpha\in \Sigma_m}
C_{(\boldsymbol \alpha)}^*\otimes A_{(\boldsymbol \alpha)}^*\otimes {\bf S}_{\boldsymbol\alpha}^*& I_\cE\otimes I_\cK\otimes I
\end{matrix}
\right)\geq 0,
\end{equation*}
which implies
$$
      \left\|\sum_{m=0}^\infty\sum_{\boldsymbol \alpha\in \Sigma_m}C_{(\boldsymbol\alpha)}\otimes  A_{(\boldsymbol \alpha)}\otimes {\bf S}_{\boldsymbol \alpha}\right\|\leq 1
      $$
and completes  the proof.
\end{proof}

\begin{corollary} \label{im} Let  $\{\Sigma_m\}_{m=0}^\infty$ be an exhaustion of
 $\FF_{n_1}^+\times\cdots \times \FF_{n_k}^+$  by  minimal and orthogonal sets, and
let   $F:{\bf B_n}(\cH)\to B(\cK)\otimes_{min}B(\cH)$ be a   free holomorphic function with representation
$$
F({\bf X})=\sum_{\boldsymbol \alpha\in  \FF_{n_1}^+\times\cdots \times \FF_{n_k}^+} A_{(\boldsymbol \alpha)}\otimes  {\bf X}_{\boldsymbol \alpha}
$$
  such that $ F(0)\geq 0$ and  $\Re F({\bf X})\leq I$ for any ${\bf X}\in {\bf B_n}(\cH)$.
  Then the following statements hold.

 \begin{enumerate}
\item[(i)] For any   $\boldsymbol\rho=(\rho_{i,j})\in [0,1)^{n_1+\cdots n_k}$,
$$
\sup_{{\bf X}\in \boldsymbol\rho{\bf B_n}} \left\|\sum_{\boldsymbol \alpha\in \Sigma_m} A_{(\boldsymbol \alpha)}\otimes {\bf X}_{\boldsymbol \alpha}\right\|\leq 2\|I-A_0\|\sup_{\boldsymbol\alpha\in \Sigma_m} \boldsymbol\rho_{\boldsymbol \alpha}.
$$
\item[(ii)]
If  $
\sum\limits_{m=1}^\infty \sup\limits_{\boldsymbol\alpha\in \Sigma_m} \boldsymbol\rho_{\boldsymbol \alpha}\leq \frac{1}{2},
$
 then
$$
\sum_{m=0}^\infty  \sup_{{\bf X}\in   \boldsymbol\rho{\bf B_n}(\cH)^-} \left\|\sum_{\boldsymbol \alpha\in \Sigma_m}   A_{(\boldsymbol \alpha)}\otimes {\bf X}_{\boldsymbol \alpha}\right\|
\leq
 \|A_0\|+\|I-A_0\|.
$$
\item[(iii)] If $\boldsymbol\rho:=(\rho_{i,j})$ with $\rho_{i,j}=r_i\in [0,1)$ for   $j\in \{1,\ldots, n_i\}$ and
$\prod_{i=1}^k (1-r_i)\geq \frac{2}{3}$, then
$$
\sum\limits_{{\bf p}\in \ZZ_+^k} \left\|\sum\limits_ {\boldsymbol\alpha\in \Lambda_{\bf p}} A_{(\boldsymbol \alpha)}\otimes   {\bf X}_{\boldsymbol \alpha}\right\|\leq
 \|A_0\|+\|I-A_0\|, \qquad {\bf X}\in \boldsymbol\rho{\bf B_n}(\cH)
$$
   In particular, the inequality  holds when  $r_1=\cdots =r_k\leq 1-\left(\frac{2}{3}\right)^{1/k}$.
\end{enumerate}
\end{corollary}
\begin{proof} Items (i) and (ii) are particular cases of Theorem \ref{CAX}.  It is easy to see that part (iii) follows from part (ii).
\end{proof}
We remark that under the conditions of Corollary \ref{im}, we have $0\leq A_0\leq I$. Due to the spectral theorem, one can easily see that $\|A_0\|+\|I-A_0\|=1$ if and only $A_0=a_0I$ for some scalar $a_0>0$.
In the particular case when $n_1=\cdots= n_k=1$, we obtain the following result for the scalar  polydisc $\DD^k$.

\begin{corollary}  \label {pdisc} Let
$$
f({\bf z})=\sum_{(p_1,\ldots, p_k)\in \ZZ_+^k}  A_{p_1,\ldots, p_k} z_1^{p_1}\cdots z_k^{p_k},\qquad {\bf z}=(z_1,\ldots, z_k)\in \DD^k,
$$
 be an operator-valued analytic function on the  polydisk  such that $\Re f(z)\leq I$ and $f(0)\geq 0$. Then
$$
\sum_{(p_1,\ldots, p_k)\in \ZZ_+^k} \| A_{p_1,\ldots, p_k}\| r_1^{p_1}\cdots r_k^{p_k}\leq \|A_0\|+\|I-A_0\|
$$
for any $r_i\in [0, 1)$ with $\prod_{i=1}^k (1-r_i)\geq \frac{2}{3}$.
\end{corollary}

We should mention that, in the particular case when $k=1$,   we recover the corresponding result for the disc   obtained by Paulsen and Singh in \cite{PS}.

\section{The Bohr radius $K_{mh}({\bf B_n})$ for the Hardy space $H^\infty({\bf B_n})$}

 In this section, using the results of  Section 3, we  obtain   estimations for the Bohr radii
     $K_{mh}({\bf B_n})$ and  $K_{mh}^0({\bf B_n})$ which  extend Boas-Khavinson  results for the scalar polydisk to the polyball.

\begin{theorem} \label{RE}  Let   $F:{\bf B_n}(\cH)\to B(\cK)\otimes_{min}B(\cH)$ be a   free holomorphic function with representation
$
F({\bf X})=\sum_{\boldsymbol \alpha\in  \FF_{n_1}^+\times\cdots \times \FF_{n_k}^+} A_{(\boldsymbol \alpha)}\otimes  {\bf X}_{\boldsymbol \alpha}
$
  such that $ F(0)=a_0I$, $a_0\geq 0$, and  $\Re F({\bf X})\leq I$ for any ${\bf X}\in {\bf B_n}(\cH)$. If $k>1$, then
  $$\sum\limits_{{\bf p}\in \ZZ_+^k} \left\|\sum\limits_ {\boldsymbol\alpha\in \Lambda_{\bf p}} A_{(\boldsymbol \alpha)} \otimes  {\bf X}_{\boldsymbol \alpha}\right\|
\leq 1, \qquad {\bf X}\in   r {\bf B_n}(\cH)^-.
$$
for any  $r\in [0,\gamma_k]$, where $\gamma_k\in \left(\frac{1}{3\sqrt{k}},1\right)$ is the   solution of the equation
  \begin{equation*}
  \sum_{m=1}^\infty \left(\begin{matrix} m+k-1\\k-1\end{matrix}\right)^{1/2} r^m=\frac{1}{2}.
  \end{equation*}
  Moreover, if $r\geq \frac{2\sqrt{\log k}}{\sqrt{k}}$, then the inequality above fails.
\end{theorem}
 \begin{proof}
Assume that  $F $ has the representation
$
F({\bf X})=\sum\limits_{{\bf p}\in \ZZ_+^k} \left(\sum\limits_ {\boldsymbol\alpha\in \Lambda_{\bf p}} A_{(\boldsymbol \alpha)}  \otimes {\bf X}_{\boldsymbol \alpha}\right)$
and    the   properties stated in the theorem. Note also that $a_0\leq 1$.
Let ${\bf Y}=(Y_1,\ldots, Y_k)\in {\bf B_n}(\cH)^-$ with $Y_i=(Y_{i,1},\ldots, Y_{i,n_i})$  and let ${\bf z}:=(z_1, \ldots, z_k)\in \DD^k$.
Since ${\bf B_n}(\cH)$ is noncommutative complete Reinhardt domain (see Proposition 1.3 from \cite{Po-automorphisms-polyball}), we have ${\bf zY}:=(z_1Y_1,\ldots, z_k Y_k)\in {\bf B_n}(\cH)$. Consequently,
$$
g(z_1,\ldots, z_k):=F({\bf zY})=\sum\limits_{{\bf p}\in \ZZ_+^k} \left(\sum\limits_ {\boldsymbol\alpha\in \Lambda_{\bf p}} A_{(\boldsymbol \alpha)}  \otimes {\bf Y}_{\boldsymbol \alpha}\right)z_1^{p_1}\cdots z_k^{p_k}
$$
is an   operator-valued  analytic function on the polydisk $\DD^k$ with the properties that $g(0)=a_0$ and $\Re g({\bf z})\leq I$ for ${\bf z}\in \DD^k$.
Denote $B_{p_1,\ldots p_k}:=\sum\limits_ {\boldsymbol\alpha\in \Lambda_{\bf p}} A_{(\boldsymbol \alpha)} \otimes {\bf Y}_{\boldsymbol \alpha}$ for all ${\bf p}=(p_1,\ldots, p_k)\in \ZZ_+^k$.
Applying  Theorem \ref{W-O}  to $g$ and the right minimal set
$$\Lambda=\Gamma_m:=\{\boldsymbol\alpha=(\alpha_1,\ldots, \alpha_k)\in \FF_{n_1}^+\times\cdots \times \FF_{n_k}^+: \  |\boldsymbol\alpha|:=|\alpha_1|+\cdots +|\alpha_k|=m\},
$$
 when $n_1=\cdots=n_k=1$, we obtain
$$
\sum_{{(p_1,\ldots, p_k)\in \ZZ_+^k}\atop {p_1+\cdots +p_k=m}} B_{p_1,\ldots p_k}^* B_{p_1,\ldots p_k} \leq 4\|I-B_0\|(I-B_0)=4(1-a_0)^2.
$$
Consequently, if ${\bf z}:=(z_1, \ldots, z_k)\in \DD^k$, we deduce that
\begin{equation*}
\begin{split}
\sum\limits_{(p_1,\ldots, p_k)\in \ZZ_+^k} &\left\|B_{p_1,\ldots p_k}z_1^{p_1}\cdots z_k^{p_k} \right\|\\
&\leq \|B_0\| +\sum_{m=1}^\infty \sum_{{(p_1,\ldots, p_k)\in \ZZ_+^k}\atop {p_1+\cdots +p_k=m}}\left\|B_{p_1,\ldots p_k}z_1^{p_1}\cdots z_k^{p_k} \right\|\\
&\leq  \|B_0\| +\sum_{m=1}^\infty \left\|\sum_{{(p_1,\ldots, p_k)\in \ZZ_+^k}\atop {p_1+\cdots +p_k=m}} B_{p_1,\ldots p_k}^* B_{p_1,\ldots p_k}\right\|^{1/2}
\left(\sum_{{(p_1,\ldots, p_k)\in \ZZ_+^k}\atop {p_1+\cdots +p_k=m}}|z_1|^{2p_1}\cdots |z_k|^{2p_k}\right)^{1/2}\\
&\leq
a_0+2(1-a_0)\sum_{m=1}^\infty\left(\sum_{i=1}^k |z_i|^2\right)^{m/2}.
\end{split}
\end{equation*}
Hence, if $\sum_{m=1}^\infty\left(\sum_{i=1}^k |z_i|^2\right)^{m/2}\leq \frac{1}{2}$, then
$\sum\limits_{(p_1,\ldots, p_k)\in \ZZ_+^k} \left\|B_{p_1,\ldots p_k}z_1^{p_1}\cdots z_k^{p_k} \right\|\leq 1$.
In particular,  if we take $z_1=\cdots =z_k=r\in [0,\frac{1}{3\sqrt{k}}]$, then
$$
\sum_{m=1}^\infty\left(\sum_{i=1}^k |z_i|^2\right)^{m/2}=\sum_{m=1}^\infty(kr^2)^{m/2}\leq \sum_{m=1}^\infty\frac{1}{3^m}=\frac{1}{2}.
$$
The results above show that
$$
\sum\limits_{{\bf p}=(p_1,\ldots, p_k)\in \ZZ_+^k} \left\| \left(\sum\limits_ {\boldsymbol\alpha\in \Lambda_{\bf p}} A_{(\boldsymbol \alpha)} \otimes {\bf Y}_{\boldsymbol \alpha}\right)r^{p_1+\cdots p_k} \right\|\leq 1,
$$
which is equivalent to
$$\sum\limits_{{\bf p}\in \ZZ_+^k} \left\|\sum\limits_ {\boldsymbol\alpha\in \Lambda_{\bf p}} A_{(\boldsymbol \alpha)} \otimes {\bf X}_{\boldsymbol \alpha}\right\|
\leq 1, \qquad {\bf X}\in   r {\bf B_n}(\cH)^-,
$$
for any $r\in \left[0,\frac{1}{3\sqrt{k}}\right]$.
 Now, note   that
$$
\varphi(r_1,\ldots, r_k):=\sum_{m=1}^\infty \left(\sum_{{(p_1,\ldots, p_k)\in \ZZ_+^k}\atop{p_1+\cdots+p_k=m}}
r_1^{2p_1}\cdots r_k^{2p_k}\right)^{1/2} < \sum_{m=1}^\infty (r_1^2+\cdots r_k^2)^{m/2},
$$
whenever $k>1$. Taking $r_1=\cdots =r_k=r$, the inequality above becomes
$$
\sum_{m=1}^\infty \left(\begin{matrix} m+k-1\\k-1\end{matrix}\right)^{1/2} r^m <
\sum_{m=1}^\infty (kr^2)^{m/2}.
$$
In particular, if $r=\frac{1}{3\sqrt{k}}$, we obtain
$$
\sum_{m=1}^\infty \left(\begin{matrix} m+k-1\\k-1\end{matrix}\right)^{1/2} \left(\frac{1}{3\sqrt{k}}\right)^m <
\frac{1}{2}.
$$
Hence,   there is a unique solution $\gamma_k> \frac{1}{3\sqrt{k}}$ of the equation
$$
  \sum_{m=1}^\infty \left(\begin{matrix} m+k-1\\k-1\end{matrix}\right)^{1/2} r^m=\frac{1}{2}.
  $$
  Consequently, $\varphi(\gamma_k,\ldots, \gamma_k)=\frac{1}{2}$.
 Using the first part of the proof when $z_1=\cdots= z_k=\gamma_k$,  we deduce  that  the inequality in the theorem holds for any  $r\in [0,\gamma_k]$.

    On the other hand,  according to a result of Boas and Khavinson \cite{BK},
if $r\geq\frac{2\sqrt{\log k}}{\sqrt{k}}$, then there is a function $f({\bf z})=\sum_{{\bf p}=(p_1,\ldots, p_k)\in \ZZ_+^k} a_{\bf p}{\bf z}^p$, ${\bf z}=(z_1,\ldots, z_k)\in \DD^k$,   in the Hardy space $H^\infty(\DD^k)$
such that  $\|f\|_\infty\leq 1$ and
$\sum_{{\bf p}\in \ZZ_+^k} r^{|{\bf p}|} |a_{\bf p}|> 1$. We can assume without loss of generality that $f(0)\geq 0$.
We remark that the free holomorphic   function
$F({\bf X}):=\sum_{{\bf p}=(p_1,\ldots, p_k)\in \ZZ_+^k} a_{\bf p}X_{1,1}^{p_1}X_{2,1}^{p_2}\cdots X_{k,1}^{p_k}$,
 where ${\bf X}=({ X}_1,\ldots, { X}_k)\in {\bf B_n}(\cH)$  and  $X_i:=(X_{i,1},\ldots, X_{i,n_i})$,
 is in the noncommutative Hardy algebra $H^\infty({\bf B_n})$ and $\|F\|_\infty=\|f\|_\infty$.
 Therefore, $F(0)=f(0)I$ and $\Re F({\bf X})\leq I$ for any ${\bf X}\in {\bf B_n}(\cH)$. Since
 $$
 \sum_{{\bf p}\in \ZZ_+^k} \sup_{{\bf X}\in r{\bf B_n}(\cH)}\left\|  |a_{\bf p}|X_{1,1}^{p_1}X_{2,1}^{p_2}\cdots X_{k,1}^{p_k}\right\|
 =
 \sum_{{\bf p}\in \ZZ_+^k}|a_{\bf p}|r^{|{\bf p}|}\|{\bf S}_{1,1}^{p_1}{\bf S}_{2,1}^{p_2}\cdots {\bf S}_{k,1}^{p_k}\|
 =\sum_{{\bf p}\in \ZZ_+^k} r^{|{\bf p}|} |a_{\bf p}|>1,
 $$
the last part of the theorem holds.
The proof is complete.
\end{proof}

We remark that the hypothesis of Theorem \ref{RE} are satisfied, in particular,  for any  bounded   free holomorphic function $F:{\bf B_n}(\cH)\to B(\cK)\otimes_{min}B(\cH)$  with $\|F\|_\infty\leq 1$ and
$ F(0)=a_0I$ for some  $a_0\geq 0$.
As a consequence, we obtain the following result concerning the Bohr radius for the  Hardy space $H^\infty({\bf B_n})$.

\begin{corollary} If $k>1$, then the   Bohr radius for   the Hardy space $H^\infty({\bf B_n})$, with respect to the multi-homogeneous expansion,  satisfies the inequalities
$$
\frac{1}{3\sqrt{k}}< K_{mh}({\bf B_n})< \frac{2\sqrt{\log k}}{\sqrt{k}}
$$
and
$$
\limsup_{k\to\infty}\frac{K_{mh}({\bf B_n})}{\frac{\sqrt{\log k}}{\sqrt{k}}}\leq 1.
$$
Moreover, if     $\gamma_k\in [0,1)$ is the   solution of the equation
  $$
  \sum_{m=1}^\infty \left(\begin{matrix} m+k-1\\k-1\end{matrix}\right)^{1/2} r^m=\frac{1}{2}
  $$
   then
   $$\frac{1}{3\sqrt{k}}<\gamma_k\leq K_{mh}({\bf B_n}).
   $$
\end{corollary}
\begin{proof}
Let $F\in H^\infty({\bf B_n})$ have the representation
$
F({\bf X})=\sum_{{\bf p}\in \ZZ_+^k} \left(\sum_ {\boldsymbol\alpha\in \Lambda_{\bf p}} a_{(\boldsymbol \alpha)}  {\bf X}_{\boldsymbol \alpha}\right)$
and assume that $\|F\|_\infty\leq 1$. Without loss of generality, we can assume that $a_0\geq 0$. Note also that $a_0\leq 1$ and    $\Re F({\bf X})\leq I$ for any ${\bf X}\in {\bf B_n}(\cH)$. Applying Theorem  \ref{RE} to $F$ one can obtain most of the results of the corollary.
 The only thing remaining   to show  is that  the inequality with the $\limsup$ holds.
To this end, note that, due to Proposition \ref{prelim},  each function in $H^\infty(\DD^k)$ can be seen as   a function  in $H^\infty({\bf B_n})$. Consequently,  we have  $K_{mh}({\bf B_n})\leq K_{mh}({\DD^k})$. Using the result
from \cite{BK} (se also  \cite{BPS}) we have
$$
\limsup_{k\to\infty}\frac{K_{mh}({\bf B_n})}{\frac{\sqrt{\log k}}{\sqrt{k}}}\leq \limsup_{k\to\infty}\frac{K_{mh}(\DD^k)}{\frac{\sqrt{\log k}}{\sqrt{k}}} \leq 1.
$$
The proof is complete.
\end{proof}

\begin{theorem} \label{RE2}  Let   $F:{\bf B_n}(\cH)\to B(\cK)\otimes_{min}B(\cH)$ be a bounded  free holomorphic function with representation
$
F({\bf X})=\sum_{\boldsymbol \alpha\in  \FF_{n_1}^+\times\cdots \times \FF_{n_k}^+} A_{(\boldsymbol \alpha)}\otimes  {\bf X}_{\boldsymbol \alpha}
$
  such that $ F(0)=0$.  If $k>1$, then
  $$\sum\limits_{{\bf p}\in \ZZ_+^k} \left\|\sum\limits_ {\boldsymbol\alpha\in \Lambda_{\bf p}} A_{(\boldsymbol \alpha)} \otimes  {\bf X}_{\boldsymbol \alpha}\right\|
\leq \|F\|_\infty, \qquad {\bf X}\in   r {\bf B_n}(\cH)^-
$$
for any  $r\in [0,t_k]$, where $t_k\in \left(\frac{1}{2\sqrt{k}},1\right)$ is the   solution of the equation
  \begin{equation*}
  \sum_{m=1}^\infty \left(\begin{matrix} m+k-1\\k-1\end{matrix}\right)^{1/2} r^m=1.
  \end{equation*}
  Moreover, if $r\geq \max \{\frac{2\sqrt{\log k}}{\sqrt{k}}, \frac{1}{\sqrt{2}}\}$, then the inequality above fails.
\end{theorem}
\begin{proof} The proof is very similar to that of Theorem \ref{RE}. We point out the differences. We use Proposition \ref{W} (instead of Theorem \ref{W-O}) to obtain
$$
\sum_{{(p_1,\ldots, p_k)\in \ZZ_+^k}\atop {p_1+\cdots +p_k=m}} B_{p_1,\ldots p_k}^* B_{p_1,\ldots p_k} \leq  1-|a_0|^2=1.
$$
Consequently, we deduce that
$$
\sum\limits_{(p_1,\ldots, p_k)\in \ZZ_+^k} \left\|B_{p_1,\ldots p_k}z_1^{p_1}\cdots z_k^{p_k} \right\|\leq \sum_{m=1}^\infty\left(\sum_{i=1}^k |z_i|^2\right)^{m/2}.
$$
In particular,  if we take $z_1=\cdots =z_k=r\in [0,\frac{1}{2\sqrt{k}}]$, then
$$
\sum_{m=1}^\infty\left(\sum_{i=1}^k |z_i|^2\right)^{m/2}=\sum_{m=1}^\infty(kr^2)^{m/2}\leq \sum_{m=1}^\infty\frac{1}{2^m}=1.
$$
This leads to
$$\sum\limits_{{\bf p}\in \ZZ_+^k} \left\|\sum\limits_ {\boldsymbol\alpha\in \Lambda_{\bf p}} A_{(\boldsymbol \alpha)} \otimes {\bf X}_{\boldsymbol \alpha}\right\|
\leq \|F\|_\infty, \qquad {\bf X}\in   r {\bf B_n}(\cH)^-,
$$
for any $r\in \left[0,\frac{1}{2\sqrt{k}}\right]$.
The rest of the  proof is very similar to that of Theorem \ref{RE}. We leave it to the reader.
\end{proof}

A simple consequence of Theorem \ref{RE2} is the following.

\begin{corollary} \label{coco} If $k>1$, the Bohr radius $K_{mh}^0({\bf B_n})$ satisfies the inequalities
$$
\frac{1}{2\sqrt{k}}< K_{mh}^0({\bf B_n})< \frac{2\sqrt{\log k}}{\sqrt{k}}.
$$
\end{corollary}

\section{The Bohr radius $K_{h}({\bf B_n})$ for the Hardy space $H^\infty({\bf B_n})$}

With respect to the homogeneous power series expansion of the elements in the Hardy space $H^\infty({\bf B_n})$, we  prove that  $K_{h}({\bf B_n})=1/3$,  extending the classical result
  to our multivariable noncommutative setting. We also obtain estimations for the Bohr radius $K_{h}^0({\bf B_n})$.

According to Exemple \ref{Ex1}, if $q\in \ZZ^+$, then
$$
\Gamma_q:=\{\boldsymbol\alpha=(\alpha_1,\ldots, \alpha_k)\in \FF_{n_1}^+\times\cdots \times \FF_{n_k}^+: \  |\boldsymbol\alpha|:=|\alpha_1|+\cdots +|\alpha_k|=q\}
$$
is a minimal set in  $\FF_{n_1}^+\times\cdots \times \FF_{n_k}^+$ and  $\{\Gamma_q\}_{q=0}^\infty$ is an   exhaustion of
 $\FF_{n_1}^+\times\cdots \times \FF_{n_k}^+$ by minimal sets.
 Any function $F$ in the Hardy algebra $H^\infty({\bf B_n})$ has a power series expansion   in terms of  {\it homogeneous} polynomials, i.e.
\begin{equation*}
F({\bf X})=\sum\limits_{q=0}^\infty \left(\sum\limits_ {\boldsymbol\alpha\in \Gamma_q} a_{(\boldsymbol \alpha)}  {\bf X}_{\boldsymbol \alpha}\right), \qquad {\bf X}\in {\bf B_n}(\cH),
\end{equation*}
for any Hilbert space $\cH$, where the convergence is in the operator norm topology.

 \begin{definition}\label{Bohr-objects}
\begin{enumerate}
\item[(i)] The {\it Bohr radius} for the polyball ${\bf B_n}$  is  denoted  by $K_{h}({\bf B_n})$ and is the largest $r\geq 0$ such that
   \begin{equation*}
\sum\limits_{q=0}^\infty \left\|\sum\limits_ {\boldsymbol\alpha\in \Gamma_q} a_{(\boldsymbol \alpha)}  {\bf X}_{\boldsymbol \alpha}\right\|
\leq \|F\|_\infty, \qquad {\bf X}\in   r {\bf B_n}(\cH)^-,
\end{equation*}
  for any   $F\in H^\infty({\bf B_n})$.
  \item[(ii)]
The {\it Bohr scaling set}  for  the polyball ${\bf B_n}$ is denoted by  $\cS_{h}({\bf B_n})$ and consists of all
$\boldsymbol\rho=(\rho_{i,j})$  with $\rho_{i,j}\geq 0$  and such that
 \begin{equation*}
\sum\limits_{q=0}^\infty \left\|\sum\limits_ {\boldsymbol\alpha\in \Gamma_q} a_{(\boldsymbol \alpha)}  {\bf X}_{\boldsymbol \alpha}\right\|
\leq \|F\|_\infty, \qquad {\bf X}\in   \boldsymbol \rho {\bf B_n}(\cH)^-,
\end{equation*}
  for any   $F\in H^\infty({\bf B_n})$.

\item [(iii)]The {\it Bohr part} of the polyball ${\bf B_n}$ is defined by \
$
\cB_{h}({\bf B_n}):=\bigcup_{\boldsymbol \rho\in \cS_{h}({\bf B_n})} \boldsymbol \rho {\bf B^-_n} .
$

\item[(iv)] The   Bohr  positive scalar  part of the polyball ${\bf B_n}$ is  the set $\Omega_{h}({\bf B_n})$   of all
    ${\bf r}:= \{r_{i,j})$ with
    $r_{i,j}\geq 0$  and such that
    $$
 \sum\limits_{q=0}^\infty \left|\sum\limits_ {\boldsymbol\alpha\in \Gamma_q} a_{(\boldsymbol \alpha)}  {\bf r}_{\boldsymbol \alpha}\right|
\leq \|F\|_\infty
$$
for any   $F\in H^\infty({\bf B_n})$.

\end{enumerate}
\end{definition}
We remark that, due to the noncommutative von Neumann inequality for the polyball \cite{Po-poisson}, we have
$$ \sup_{{\bf X}\in \boldsymbol r {\bf B_n(\cH)^-}}\left\|\sum\limits_ {\boldsymbol\alpha\in \Gamma_q} a_{(\boldsymbol \alpha)} {\bf X}_{\boldsymbol \alpha}\right\|
=\left\|\sum\limits_ {\boldsymbol\alpha\in \Gamma_q} a_{(\boldsymbol \alpha)} r^{|{\boldsymbol \alpha}|}{\bf S}_{\boldsymbol \alpha}\right\|
$$
and  the inequality in part (i) of the definition above is equivalent to
$
\sum\limits_{q=0}^\infty\left\|\sum\limits_ {\boldsymbol\alpha\in \Gamma_q} a_{(\boldsymbol \alpha)} r^{|{\boldsymbol \alpha}|}{\bf S}_{\boldsymbol \alpha}\right\|
\leq \|F\|_\infty.
$
A similar observation should be made for the inequality in  definition  (ii).
We should mention that due to the fact that
$$
\sum_{m=0}^\infty \left\|\sum_{\boldsymbol\alpha\in \Gamma_m} a_{(\boldsymbol\alpha)}   {\bf X}_{\boldsymbol\alpha}\right\|\leq
\sum\limits_{{\bf p}\in \ZZ_+^k}\left\|\sum\limits_ {\boldsymbol\alpha\in \Lambda_{\bf p}} a_{(\boldsymbol \alpha)}  {\bf X}_{\boldsymbol \alpha}\right\|,\qquad {\bf X}\in {\bf B_n}(\cH).
$$
we always have  $ K_{h}({\bf B_n})\geq  K_{mh}({\bf B_n})$ and $ K_{h}^0({\bf B_n})\geq  K_{mh}^0({\bf B_n})$. As we will see soon, the inequalities are strict, in general.

\begin{theorem} \label{FX} If $F:{\bf B_n}(\cH)\to B(\cK)\otimes_{min} B(\cH)$ is  any bounded free holomorphic function with representation
$
F({\bf X} ):=\sum_{m=0}^\infty \sum_{\boldsymbol\alpha\in \Gamma_m}  A_{(\boldsymbol \alpha)} \otimes {\bf X}_{\boldsymbol \alpha}
$
such that $ F(0)=a_0I$ for some  $a_0\in \CC$, then
  $$
\sum_{m=0}^\infty \left\|\sum_{\boldsymbol\alpha\in \Gamma_m} A_{(\boldsymbol\alpha)} \otimes  {\bf X}_{\boldsymbol\alpha}\right\| \leq \|F\|_\infty
$$
for any $ {\bf X}\in \frac{1}{3}{\bf B_n}(\cH)$. Moreover, $\frac{1}{3}$ is the best possible constant and the inequality is strict unless $F$ is a constant.
\end{theorem}
\begin{proof}
Without loss of generality, we can assume that $\|F\|_\infty$=1. Let ${\bf Y}\in {\bf B_n}(\cH)$ and $z\in \DD$. According to Proposition 1.3 from \cite{Po-automorphisms-polyball}, ${\bf B_n}(\cH)$ is a noncommutative complete Reinhardt domain. Consequently, $z{\bf Y}\in {\bf B_n}(\cH)$ for any $z\in \DD$, and
$\sup_{z\in \DD}\|F(z{\bf Y})\|\leq 1$ for any ${\bf Y}\in {\bf B_n}(\cH)$.
Since $F$ is free holomorphic on ${\bf B_n}(\cH)$, we have (see \cite{Po-automorphisms-polyball})
$
\sum_{m=0}^\infty \left\|\sum_{\boldsymbol\alpha\in \Gamma_m} A_{(\boldsymbol\alpha)} \otimes   {\bf Y}_{\boldsymbol\alpha}\right\|<\infty$,  ${\bf Y}\in {\bf B_n}(\cH).
$
Hence, we deduce that
$$
g(z):=F(z{\bf Y})=\sum_{m=0}^\infty \left(\sum_{\boldsymbol\alpha\in \Gamma_m} A_{(\boldsymbol\alpha)} \otimes   {\bf Y}_{\boldsymbol\alpha}\right)z^m,\qquad z\in \DD,
$$
is an operator-valued analytic function  on $\DD$ with $\|g\|_\infty\leq 1$ and $g(0)=a_0I_\cH$, $a_0\in \CC$. Applying Theorem \ref{Bohr2}, part (i), when $ k=1$,  to $g$ and taking into account that $0\leq |a_0|\leq 1$, we obtain
$$
\sum_{m=0}^\infty \left\|\sum_{\boldsymbol\alpha\in \Gamma_m} A_{(\boldsymbol\alpha)}  \otimes {\bf Y}_{\boldsymbol\alpha}\right\|r^m\leq |a_0|+(1-|a_0|^2)\frac{r}{1-r}\leq1, \qquad {\bf Y}\in {\bf B_n}(\cH),
$$
for any $r\in [0,\frac{1}{3}]$, which is equivalent to the inequality in the theorem.

  On the other hand, due to Proposition \ref{prelim}, we have  $ K_{h}({\bf B_n})\leq K_h(\DD)\leq\frac{1}{3}.
$
Therefore, $\frac{1}{3}$ is the best possible constant. To prove the last part of the theorem, note that if $\|F\|_\infty=1$, then using the inequality above, we deduce that $|a_0|+\frac{1}{2}(1-|a_0|^2)=1$, which implies $a_0=0$. Employing Proposition \ref{W}. we deduce that $ A_{(\boldsymbol\alpha)}=0$ for any $\boldsymbol\alpha \in \FF_{n_1}^+\times\cdots \times  \FF_{n_k}^+$ different from the identity ${\bf g}_0$. This shows that $F=a_0I$.
The proof is complete.
\end{proof}

\begin{corollary} \label{zero} If $F:{\bf B_n}(\cH)\to B(\cK)\otimes_{min} B(\cH)$ is  any bounded free holomorphic function with representation
$
F({\bf X} ):=\sum_{m=0}^\infty \sum_{\boldsymbol\alpha\in \Gamma_m}  A_{(\boldsymbol \alpha)} \otimes {\bf X}_{\boldsymbol \alpha}
$
such that $ F(0)=0$, then
  $$
\sum_{m=0}^\infty \left\|\sum_{\boldsymbol\alpha\in \Gamma_m} A_{(\boldsymbol\alpha)} \otimes  {\bf X}_{\boldsymbol\alpha}\right\| \leq \|F\|_\infty
$$
for any $ {\bf X}\in \frac{1}{2}{\bf B_n}(\cH)$.
\end{corollary}
\begin{proof} Following the proof of Theorem \ref{FX}, we have
$$
\sum_{m=0}^\infty \left\|\sum_{\boldsymbol\alpha\in \Gamma_m} A_{(\boldsymbol\alpha)}  \otimes {\bf Y}_{\boldsymbol\alpha}\right\|r^m\leq |a_0|+(1-|a_0|^2)\frac{r}{1-r}=\frac{r}{1-r}\leq 1, \qquad {\bf Y}\in {\bf B_n}(\cH),
$$
for any $r\in[0,\frac{1}{2}]$. This completes the proof.
\end{proof}

\begin{corollary}  \label{K-0} The Bohr radius $K_{h}^0({\bf B_n})$ satisfies the inequalities
$$
\max\left\{\frac{1}{2}, \sqrt{1-\left(\frac{1}{2}\right)^{1/k}}\right\}\leq K_h^0({\bf B_n})\leq \frac{1}{\sqrt{2}}.
$$
\end{corollary}
\begin{proof}  Due to Corollary \ref{zero}, we have $\frac{1}{2}\leq K_h^0({\bf B_n})$.  Using Corollary \ref{Kmh0} and the fact that   $K_{mh}^0({\bf B_n})\leq K_{h}^0({\bf B_n})$, we deduce that  $ \sqrt{1-\left(\frac{1}{2}\right)^{1/k}}\leq K_h^0({\bf B_n})$.
On the other hand, employing Proposition \ref{prelim}, we deduce the inequality  $K_{h}^0({\bf B_n})\leq K_{h}^0(\DD)$. Since   $K_{h}^0(\DD)=\frac{1}{\sqrt{2}}$ (see e.g. \cite{PPoS}), we conclude that $ K_h^0({\bf B_n})\leq \frac{1}{\sqrt{2}}$ and complete the proof.
\end{proof}

\begin{corollary} \label{Bo-homo} Let $F:{\bf B_n}(\cH)\to B(\cH)$ be a bounded free holomorphic function with representation
$
F({\bf X} ):=\sum_{m=0}^\infty \sum_{\boldsymbol\alpha\in \Gamma_m}  a_{(\boldsymbol \alpha)} {\bf X}_{\boldsymbol \alpha}.
$
Then the following statements hold:
\begin{enumerate}
\item[(i)] For any $ {\bf X}\in \frac{1}{3}{\bf B_n}(\cH)$
 $$
\sum_{m=0}^\infty \left\|\sum_{\boldsymbol\alpha\in \Gamma_m} a_{(\boldsymbol\alpha)}   {\bf X}_{\boldsymbol\alpha}\right\| \leq \|F\|_\infty,\qquad F\in H^\infty ({\bf B_n}),
$$
and $\frac{1}{3}$ is the best possible constant. Moreover, the inequality is strict unless $F$ is a constant.
\item[(ii)] $K_h({\bf B_n})=\frac{1}{3}$.
\item[(iii)] $\cS_{h}({\bf B_n})= [0,\frac{1}{3}]^{n_1+\cdots n_k}$ and $\cB_{h}({\bf B_n})=\frac{1}{3}{\bf B_n^-}$.
\item[(iv)] The Bohr positive scalar part   for the polyball $\Omega_h({\bf B_n})$ consists of all tuples of positive numbers
    ${\bf r}:=(r_1,\ldots, r_k)$, $r_i:=(r_{i,1},\ldots, r_{i,n_i})$ with the property that
    $r_{i,j}\geq 0$ and $\|r_i\|_2\leq \frac{1}{3}$ for any $i\in \{1,\ldots, k\}$.
\end{enumerate}
\end{corollary}
\begin{proof} Note that items (i) and (ii)  hold due to Theorem \ref{FX}.
We need to prove (iii). If $\boldsymbol\rho=(\rho_{i,j})$ is in $ [0,\frac{1}{3}]^{n_1+\cdots n_k}$ and ${\bf X}\in \boldsymbol\rho {\bf B_n}(\cH)^-$, then the noncommutative von Neumann inequality and item (i) imply
$$
\sum_{m=0}^\infty \left\|\sum_{\boldsymbol\alpha\in \Gamma_m} a_{(\boldsymbol\alpha)}   {\bf X}_{\boldsymbol\alpha}\right\|
\leq \sum_{m=0}^\infty \left\|\sum_{\boldsymbol\alpha\in \Gamma_m} a_{(\boldsymbol\alpha)}  \boldsymbol\rho_{\boldsymbol \alpha} {\bf S}_{\boldsymbol\alpha}\right\|
\leq \sum_{m=0}^\infty \left\|\sum_{\boldsymbol\alpha\in \Gamma_m} a_{(\boldsymbol\alpha)}   {\bf S}_{\boldsymbol\alpha}\right\|\frac{1}{3^q}
\leq \|F\|_\infty.
$$
This shows that $ [0,\frac{1}{3}]^{n_1+\cdots n_k}\subseteq \cS_{h}({\bf B_n})$. If we assume that there is $\boldsymbol\rho=(\rho_{i,j})\in \cS_{h}({\bf B_n})$ with at least one $\rho_{i,j}>\frac{1}{3}$, then,  due to the fact that $H^\infty(\DD)$ is
  isometrically   embedded in $H^\infty({\bf B_n})$ (see Proposition \ref{prelim}), we deduce that $K_h(\DD)>\frac{1}{3}$, which is a contradiction. Therefore, we must have  $\cS_{h}({\bf B_n})= [0,\frac{1}{3}]^{n_1+\cdots n_k}$ which clearly implies $\cB_{h}({\bf B_n})=\frac{1}{3}{\bf B_n^-}$.

 To prove item (iv), note that applying item (i) to ${\bf X}=\{X_{ij}\}$ with $X_{i,j}=r_{i,j}I_\cH$, where $r_{i,j}\geq 0$ and $\|r_i\|_2\leq \frac{1}{3}$ for any $i\in \{1,\ldots, k\}$, we deduce that ${\bf r}\in \Omega_h({\bf B_n})$.
 Now, assume that ${\bf r}:=(r_1,\ldots, r_k)$, $r_i:=(r_{i,1},\ldots, r_{i,n_i})$ with the property that
    $r_{i,j}\geq 0$ and   at least one $r_i$ satisfies the inequality $\|r_i\|_2>\frac{1}{3}$. For simplicity, we assume that $\|r_1\|_2>\frac{1}{3}$.
    As in the proof of Theorem 4.5 from \cite{PPoS}, we can construct a bounded free holomorphic function on the unit ball $[B(H)^{n_1}]_1$,
    $f(X_1)=\sum_{m=0}^\infty \sum_{\alpha_1\in \FF_{n_1}^+, |\alpha_1|=m} c_{\alpha_1} X_{1,\alpha_1}$, with the property that $\|f\|_\infty=1$ and
    $$
    \sum_{m=0}^\infty \left|\sum_{\alpha_1\in \FF_{n_1}^+, |\alpha_1|=m} c_{\alpha_1} r_{1,\alpha_1}\right| >1.
    $$
Define $F(X_1,\ldots, X_k):=f(X_1)$ on the polyball ${\bf B_n}$, and note that $F\in H^\infty({\bf B_n})$ and  $\|F\|_\infty=1$. The free holomorphic function $F$ has the representation
    $
F({\bf X} ):=\sum_{m=0}^\infty \sum_{\boldsymbol\alpha\in \Gamma_m}  b_{(\boldsymbol \alpha)} {\bf X}_{\boldsymbol \alpha},
$
where $b_{(\boldsymbol \alpha)}=0$ unless $(\boldsymbol \alpha)=(\alpha_1, g_0^2,\ldots, g_0^k)$, in which case  $b_{(\alpha_1, g_0^2,\ldots, g_0^k)}:=c_{\alpha_1}$.
Since
$$
\sum_{m=0}^\infty \left|\sum_{\boldsymbol\alpha\in \Gamma_m} b_{(\boldsymbol\alpha)}   {\bf r}_{\boldsymbol\alpha}\right|=\sum_{m=0}^\infty \left|\sum_{\alpha_1\in \FF_{n_1}^+, |\alpha_1|=m} a_{\alpha_1} r_{1,\alpha_1}\right|>1,
$$
we deduce that  ${\bf r}\notin \Omega_h({\bf B_n})$.
The proof is complete.
\end{proof}

\begin{theorem} \label{Omega} If $F:{\bf B_n}(\cH)\to B(\cH)$ is a bounded free holomorphic function with representation
$
F({\bf X} ):=\sum_{m=0}^\infty \sum_{\boldsymbol\alpha\in \Gamma_m}  a_{(\boldsymbol \alpha)} {\bf X}_{\boldsymbol \alpha},
$
then
$$
\sum_{m=0}^\infty \left\|\sum_{\boldsymbol\alpha\in \Gamma_m} a_{(\boldsymbol\alpha)}   {\bf X}_{\boldsymbol\alpha}\right\| \leq \Omega(r)\|F\|_\infty,\qquad  {\bf X}\in r{\bf B_n}(\cH)^-,
$$
where  $\Omega(r):=\min\left\{ M(r), \left( \frac{1}{\sqrt{1-r^2}}\right)^k\right\}$ and
$$
M(r):=\begin{cases} 1& \ \text{ if } 0\leq r\leq \frac{1}{3},\\
\frac{4r^2+(1-r)^2}{4r (1-r)}& \ \text{ if } \frac{1}{3}< r <1.
\end{cases}
$$
\end{theorem}
\begin{proof} We can assume that $\|F\|_\infty\leq 1$. We saw in the proof of Theorem \ref{FX}
that
$$
\sum_{m=0}^\infty \left\|\sum_{\boldsymbol\alpha\in \Gamma_m} a_{(\boldsymbol\alpha)}   {\bf X}_{\boldsymbol\alpha}\right\| \leq |a_0|+(1-|a_0|^2)\frac{r}{1-r}
$$
for any ${\bf X}\in r{\bf B_n}(\cH)^-$. Since $0\leq |a_0|\leq 1$, as in the proof of Proposition \ref{Cro}, when $c=\frac{r}{1-r}$,  we can prove that
$\sup \{x+ (1-x^2)c: 0\leq x\leq 1\} \leq M(r)$, where $M(r)$ is given in the theorem.
On the other hand, due to Theorem \ref{Bo-Bo}, if $\boldsymbol\rho:=(\rho_{i,j})$ with $\rho_{i,j}=r_i\in [0,1)$ for   $j\in \{1,\ldots, n_i\}$, then
$$
\sum_{m=0}^\infty \left\|\sum_{\boldsymbol\alpha\in \Gamma_m} a_{(\boldsymbol\alpha)}   {\bf X}_{\boldsymbol\alpha}\right\|\leq
\sum\limits_{{\bf p}\in \ZZ_+^k}\left\|\sum\limits_ {\boldsymbol\alpha\in \Lambda_{\bf p}} a_{(\boldsymbol \alpha)}  {\bf X}_{\boldsymbol \alpha}\right\|
\leq
 \prod_{i=1}^k (1-r^2_i)^{-1/2}\|F\|_\infty,
$$
for any ${\bf X}\in \boldsymbol\rho {\bf B_n(\cH)}^-$. Taking $r_1=\cdots=r_k=r$, we complete the proof.
\end{proof}

Let $F\in H^\infty({\bf B_n})$ have the representation
$
F({\bf X} ):=\sum_{m=0}^\infty \sum_{\boldsymbol\alpha\in \Gamma_m}  a_{(\boldsymbol \alpha)} {\bf X}_{\boldsymbol \alpha}
$
and let
$$
\cM(F,r):=\sum_{m=0}^\infty r^m \left\|\sum_{\boldsymbol\alpha\in \Gamma_m}  a_{(\boldsymbol \alpha)} {\bf S}_{\boldsymbol \alpha}\right\|
$$
be the associated majorant series. Define
$$
m_{\bf B_n}(r):=\sup \frac{\cM(F,r)}{\|F\|_\infty}, \qquad r\in [0,1),
$$
where the supremum is taken over all $F\in H^\infty({\bf B_n})$  with $F$ not identically $0$.
Due to Corollary \ref{Bo-homo}, we have $m_{\bf B_n}(r)=1$ if $r\in[0,\frac{1}{3}]$.  On the other hand,  Theorem \ref{Omega}, shows that
$m_{\bf B_n}(r)\leq \Omega(r)$ for $r\in [0,1)$.
The precise value of $m_{\bf B_n}(r)$ as $\frac{1}{3}<r<1$ remains unknown, in general.

\section{Fej\' er and  Bohr  inequalities  for multivariable polynomials with operator coefficients}

In this section, we   obtain  analogues of Carath\' eodory's   inequality,   and Fej\' er  and Egerv\' ary-Sz\' azs inequalities   for   free holomorhic functions  with  operator coefficients and  positive  real parts on the polyball. These results are use to provide  an analogue of Landau's   inequality    and  Bohr type inequalities  when the norm is replaced by the numerical radius of  an operator.

The classical numerical radius of an operator $T\in B(\cH)$ is defined by
$$
\omega(T):=\sup\{|\left<Th,h\right>| : \ h\in\cH, \|h\|=1\}.
$$
Here are some of its basic properties that will be used in what follows.
\begin{enumerate}
\item[(i)] $\omega (T_1+T_2)\leq \omega(T_1)+\omega(T_2)$ for any  $T_1,T_2\in B(\cH)$.
\item[(ii)] $\omega(\lambda T)=|\lambda| \omega (T)$ for any $\lambda\in \CC$.
\item[(iii)] $\omega(U^* TU)=\omega(T)$ for any unitary operator $U$.
\item[(iv)] $\omega(X^* TX)\leq \|X\|^2\omega(T)$ for any operator $X:\cK\to \cH$.
    \item[(v)] $\omega(T\otimes I_\cE)=\omega(T)$ for any separable Hilbert space $\cE$.
        \item[(vi)] The numerical radius is continuous in the operator norm topology.
\end{enumerate}

For each $p\in \ZZ$, we set $p^+:=\max\{p, 0\}$ and $p^-:=\max\{-p, 0\}$.
 A function $F$ with operator-valued coefficients in $B(\cK)$, where $\cK$ is separable Hilbert space, is called {\it free $k$-pluriharmonic} on the abstract polyball ${\bf B_n}$    if it has the form
\begin{equation} \label{F(X)}
F({\bf X})= \sum_{p_1\in \ZZ}\cdots \sum_{p_k\in \ZZ} \sum_{{\alpha_i,\beta_i\in \FF_{n_i}^+, i\in \{1,\ldots, k\}}\atop{|\alpha_i|=p_i^-, |\beta_i|=p_i^+}}A_{(\alpha_1,\ldots,\alpha_k;\beta_1,\ldots, \beta_k)}
\otimes {\bf X}_{1,\alpha_1}\cdots {\bf X}_{k,\alpha_k}{\bf X}_{1,\beta_1}^*\cdots {\bf X}_{k,\beta_k}^*,
\end{equation}
where the multi-series converge in the operator norm topology for any  ${\bf X}=(X_1,\ldots, X_k)\in {\bf B_n}(\cH)$, with $X_i:=(X_{i,1},\ldots, X_{i,n_i})$, and any Hilbert space $\cH$.
According to \cite{Po-pluriharmonic-polyball}, the order of the series in the definition above is irrelevant. Note that any free holomorphic function on the polyball is $k$-pluriharmonic and so is the real part of a free holomorphic function.

In what follows, we obtain an analogue of Carath\' eodory's   inequality \cite{Ca}, and Fej\' er  \cite{Fe} and Egerv\' ary-Sz\' azs inequalities \cite{ES} for   free holomorhic functions  with  positive  real parts on the polyball  and  operator coefficients.

\begin{theorem} \label{Fejer} Let $m\in \NN\cup \{\infty\}$ and let
$$
f({\bf X}):=\sum_{q=1}^m\sum_{\boldsymbol\alpha\in \Gamma_q} A_{(\boldsymbol \alpha)}^*\otimes {\bf X}_{\boldsymbol \alpha}^*+A_0\otimes I+
\sum_{q=1}^m\sum_{\boldsymbol\alpha\in \Gamma_q} A_{(\boldsymbol \alpha)}\otimes {\bf X}_{\boldsymbol \alpha},\qquad {\bf X}\in {\bf B_n}(\cH),
$$
be a positive $k$-pluriharmonic  function on the polyball ${\bf B_n}$ with coefficients in $B(\cK)$. If $m\in \NN$, then for each $q\in \{1,\ldots, m\}$,
$$
\omega \left(\sum_{\boldsymbol\alpha\in \Gamma_q} A_{(\boldsymbol \alpha)}\otimes {\bf X}_{\boldsymbol \alpha}\right)\leq\omega \left(\sum_{\boldsymbol\alpha\in \Gamma_q} A_{(\boldsymbol \alpha)}\otimes {\bf S}_{\boldsymbol \alpha}\right)
\leq \|A_0\|\cos \frac{\pi}{\left[\frac{m}{q}\right]+2},\qquad {\bf X}\in {\bf B_n}(\cH),
$$
where  ${\bf S}$ is the universal model of the polyball and $[x]$ is the integer part of $x$.
If $m=\infty$, then
$$
\omega \left(\sum_{\boldsymbol\alpha\in \Gamma_q} A_{(\boldsymbol \alpha)}\otimes {\bf S}_{\boldsymbol \alpha}\right)
\leq \|A_0\|.
$$
\end{theorem}
\begin{proof}
According to the multivariable dilation theory for regular polyballs \cite{Po-Berezin-poly}, each element ${\bf X}=\{X_{i,j}\}$ in the polyball ${\bf B_n}(\cH)$,  has a dilation  $\{{\bf S}_{i,j}\otimes I_\cM\}$ for some separable Hilbert space $\cM$ with the property  that $X_{i,j}^*=({\bf S}_{i,j}^*\otimes I_\cM)|_\cH$ for any $i\in \{1,\ldots, k\}$ and $j\in \{1,\ldots, n_i\}$.
Using   the properties of the numerical radius, we obtain
$$
\omega \left(\sum_{\boldsymbol\alpha\in \Gamma_q} A_{(\boldsymbol \alpha)}\otimes {\bf X}_{\boldsymbol \alpha}\right)\leq \omega \left(\sum_{\boldsymbol\alpha\in \Gamma_q} A_{(\boldsymbol \alpha)}\otimes {\bf S}_{\boldsymbol \alpha}\right).$$
Assume now that $m\in \NN$. Due to Theorem 2.29 from \cite{Po-unitary}, we have the following operatorial version of the Fej\' er and Egerv\' ary-Sz\' azs inequalities. If $\{C_p\}_{p=0}^m\subset B(\cK)$, $m\geq 1$, and
$$
\sum_{1\leq p\leq m} C_p^*\bar z^p +C_0 +\sum_{1\leq p\leq m} C_pz^p \geq 0, \qquad z\in \DD,
$$
then
$$
\omega (C_p)\leq \|C_0\|\cos \frac{\pi}{\left[\frac{m}{p}\right]+2}, \qquad p\in \{1,\ldots, m\}.
$$
Since ${\bf B_n}$ is a noncommutative complete Reinhardt domain (see \cite{Po-automorphisms-polyball}), we deduce that  the tuple  $z{\bf S}$ is in $ {\bf B_n}(F^2(H_{n_1})\otimes \cdots \otimes F^2(H_{n_k}))$ for any $z\in \DD$. Consequently,
 $$
\psi(z):=\sum_{q=1}^m\left(\sum_{\boldsymbol\alpha\in \Gamma_q} A_{(\boldsymbol \alpha)}^*\otimes {\bf S}_{\boldsymbol \alpha}^*\right)\bar z^q+A_0\otimes I+
\sum_{q=1}^m\left(\sum_{\boldsymbol\alpha\in \Gamma_q} A_{(\boldsymbol \alpha)}\otimes {\bf S}_{\boldsymbol \alpha}\right)z^p
$$
is a  positive operator-valued harmonic function on the disc $\DD$. Applying  the above-mentioned result to $\psi$,  we deduce the inequality in the theorem.

Now, we consider the case $m=\infty$. Setting
$$
M(r,N):=2\sum_{q=N+1}^\infty \left\|\sum_{\boldsymbol\alpha\in \Gamma_q} A_{(\boldsymbol \alpha)}\otimes r^{|\boldsymbol\alpha|}{\bf S}_{\boldsymbol \alpha}\right\|, \qquad r\in [0,1), N\in \NN,
$$
 and using the noncommutative von Neumann inequality for the polyball \cite{Po-poisson}, we deduce that
 $$
\varphi(z):=\sum_{q=1}^N\left(\sum_{\boldsymbol\alpha\in \Gamma_q} A_{(\boldsymbol \alpha)}^*\otimes {\bf S}_{\boldsymbol \alpha}^*\right) r^q\bar z^q+\left(A_0\otimes I+ M(r,N)I\right) +
\sum_{q=1}^N\left(\sum_{\boldsymbol\alpha\in \Gamma_q} A_{(\boldsymbol \alpha)}\otimes {\bf S}_{\boldsymbol \alpha}\right)r^qz^p
$$
is a  positive operator-valued harmonic function on the disc $\DD$. As in the first part of the proof, if $q\in \NN$ and $N\geq q$,  we obtain
 $$
\omega \left(\sum_{\boldsymbol\alpha\in \Gamma_q} A_{(\boldsymbol \alpha)}\otimes r^q{\bf S}_{\boldsymbol \alpha}\right)
\leq \|A_0\otimes I+ M(r,N)I\|\cos \frac{\pi}{\left[\frac{N}{q}\right]+2}.
$$
 Since $M(r,N)\to 0$ as $N\to \infty$, passing to the limit in the inequality above, we deduce that $$\omega \left(\sum_{\boldsymbol\alpha\in \Gamma_q} A_{(\boldsymbol \alpha)}\otimes r^q{\bf S}_{\boldsymbol \alpha}\right)
\leq \|A_0\|,  \qquad r\in [0,1).
$$
 Since the numerical radius is continuous in the operator norm topology, taking $r\to 1$, we complete the proof.
\end{proof}

When the  positive $k$-pluriharmonic  functions on the polyball ${\bf B_n}$ have scalar coefficients, we obtain the following consequence of Theorem \ref{Fejer}.

\begin{corollary}  \label{Fejer-cor} Let  $m\in \NN\cup \{\infty\}$ and let
$$
g({\bf X}):=\sum_{q=1}^m\sum_{\boldsymbol\alpha\in \Gamma_q} \bar a_{(\boldsymbol \alpha)}{\bf X}_{\boldsymbol \alpha}^*+a_0 I+
\sum_{q=1}^m\sum_{\boldsymbol\alpha\in \Gamma_q} a_{(\boldsymbol \alpha)} {\bf X}_{\boldsymbol \alpha},\qquad {\bf X}\in {\bf B_n}(\cH),
$$
be a positive $k$-pluriharmonic  function on the polyball.   Then for each $q\in \{1,\ldots, m\}$,
$$
\sup_{\boldsymbol\lambda=\{\lambda_{i,j}\}\in  {\bf B_n}(\CC)^-}\left|\sum_{\boldsymbol\alpha\in \Gamma_q} a_{(\boldsymbol \alpha)}{\boldsymbol\lambda}_{\boldsymbol \alpha}\right|\leq a_0\cos \frac{\pi}{\left[\frac{m}{q}\right]+2}.
$$
In particular, when $q=1$, we obtain
$$\sum_{i=1}^k\left(\sum_{j=1}^{n_i} \left|\left(\frac{\partial \widetilde g}{\partial z_{i,j}}\right)(0)\right|^2\right)^{1/2}\leq a_0\cos \frac{\pi}{m+2},
$$
 where  $\widetilde g$ is the scalar representation of $g$, i.e. $\widetilde g({\bf z}):=g({\bf z})$ for  ${\bf z}=\{z_{i,j}\}\in (\CC^{n_1})_1\times \cdots \times (\CC^{n_k})_1$.
\end{corollary}

 A simple consequence of Theorem \ref{Fejer} and Corollary \ref{Fejer-cor} is the following result concerning    holomorphic functions with positive real parts on the polydisk.
 In particular cases, we recover  Fej\' er and Egerv\' ary-Sz\' azs inequalities, as well as Carath\' eodory's inequality.

\begin{corollary} If  $m\in \NN\cup \{\infty\}$ and  $f({\bf z} )=\sum\limits_{{{\bf p} \in \ZZ_+^k},  {|{\bf p}|\leq m}} a_{\bf p} {\bf z}^{\bf p}$ is a holomorphic function on the polydisk   and $\Re f(z)\geq 0$, $z\in \DD^k$, then
$$
\sup_{\boldsymbol  \xi\in \TT^k}\left|\sum_{{\bf p} \in \ZZ_+^k, |{\bf p}|=q} a_{\bf p}{\boldsymbol\xi}^{\bf p}\right|\leq 2\Re a_0\cos \frac{\pi}{\left[\frac{m}{q}\right]+2}. \qquad q\in \{1,\ldots, m\}.
$$
In particular,
\begin{enumerate}
\item[(i)]
if $m\in \NN$ and $k=1$, then we recover  Fej\' er and Egerv\' ary-Sz\' azs inequalities, i.e.
$$|a_{p_1}|\leq 2\Re a_0\cos \frac{\pi}{\left[\frac{m}{q}\right]+2}, \qquad 1\leq p_1\leq m;
$$
\item[(ii)]
if $m=\infty$ and  $k=1$, then we recover the Carath\' eodory's inequality, i.e.
 $$|a_{p_1}|\leq 2\Re a_0, \qquad p_1\in \NN.
 $$
 \end{enumerate}
\end{corollary}

We refer the reader to the paper by Kor\' anyi and Puk\' anszky \cite{KP} for a characterization of  the holomorphic functions  with positive real parts on the polydisk.

  In the next theorem, the first inequality can be seen as an analogue of Landau's   inequality \cite{LG}  for   free holomorhic functions   on the polyball  and  operator coefficients, while the second inequality is a Bohr type result where the norms are replaced by the numerical radius of operators.

\begin{theorem} \label{OM}  Let  $m\in \NN\cup \{\infty\}$ and let  $F({\bf X}):= \sum\limits_{q=1}^m\sum\limits_{\boldsymbol\alpha\in \Gamma_q} A_{(\boldsymbol \alpha)} \otimes {\bf X}_{\boldsymbol \alpha}
$
be a free holomorphic function on the polyball with $F(0)\geq 0$ and $\Re F({\bf X})\leq I$ for  ${\bf X}\in {\bf B_n}$. Then,   for each $q\in \{1,\ldots, m\}$,
$$
\omega \left(\sum_{\boldsymbol\alpha\in \Gamma_q} A_{(\boldsymbol \alpha)}\otimes {\bf X}_{\boldsymbol \alpha}\right)\leq
\omega \left(\sum_{\boldsymbol\alpha\in \Gamma_q} A_{(\boldsymbol \alpha)}\otimes {\bf S}_{\boldsymbol \alpha}\right)
\leq 2\|I-A_0\|\cos \frac{\pi}{\left[\frac{m}{q}\right]+2}
$$
and
$$
\sum_{q=0}^m \sup_{{\bf X}\in r {\bf B_n}(\cH)^-}\omega\left(\sum_{\boldsymbol\alpha\in \Gamma_q} A_{(\boldsymbol\alpha)} \otimes {\bf X}_{\boldsymbol\alpha}\right)= \sum_{q=0}^m  \omega\left(\sum_{\boldsymbol\alpha\in \Gamma_q} A_{(\boldsymbol\alpha)} \otimes {\bf S}_{\boldsymbol\alpha}\right)r^q\leq \|A_0\|+\|I-A_0\|
$$
for any $r\in [0,t_m]$, where $t_m\in (0,1)$ is the solution of the equation
$$
\sum_{q=1}^m t^q \cos \frac{\pi}{\left[\frac{m}{q}\right]+2}=\frac{1}{2}.
$$
Moreover, the sequence $\{t_m\}$ is  strictly decreasing and converging  to $\frac{1}{3}$.
When $m=\infty$, we have $t_\infty=\frac{1}{3}$.
\end{theorem}
\begin{proof} The Landau type inequality  is a consequence of Theorem \ref{Fejer}. We prove the second part of the theorem.
Let $r\in [0,t_m]$, where $t_m\in (0,1)$ is the solution of the equation mentioned above. Note that the first part of the theorem implies
\begin{equation*}
\begin{split}
 \sum_{q=0}^m  \omega\left(\sum_{\boldsymbol\alpha\in \Gamma_q} A_{(\boldsymbol\alpha)} \otimes {\bf S}_{\boldsymbol\alpha}\right)r^q &\leq
 \omega(A_0)+ 2\|I-A_0\|\sum_{q=1}^m t^q \cos \frac{\pi}{\left[\frac{m}{q}\right]+2}\\
 &\leq \|A_0\|+\|I-A_0\|.
\end{split}
\end{equation*}
A close look at the function $\varphi_m(t):=\sum_{q=1}^m t^q \cos \frac{\pi}{\left[\frac{m}{q}\right]+2}$, $t\in [0,1]$, reveals that $\varphi_m$ is strictly increasing and, consequently,   there is a unique solution $t_m\in (0,1)$  of the the equation $\varphi_m(t)=\frac{1}{2}$. Moreover, since $\varphi_m(t)<\varphi_{m+1}(t)<\sum_{q=1}^\infty t^q$ for $t\in [0,1)$ and $\sum_{q=1}^\infty \frac{1}{3^q}=\frac{1}{2}$, we have $t_m>\frac{1}{3}$ and  the sequence $\{t_m\}$ is  strictly decreasing. Since the sequence $\varphi_m$ is uniformly convergent to $\xi(t):=\sum_{q=1}^\infty t^q$  on any interval $[0,\delta]$ with $\delta\in (0,1)$, one can easily see that $t_m\to \frac{1}{3}$ as $m\to \infty$. Therefore, when $m=\infty$, we have $t_\infty=\frac{1}{3}$.
The proof is complete.
\end{proof}

Here is the numerical radius version of Bohr's inequality for  free holomorphic functions on polyballs.

\begin{corollary} \label{omco} Let  $m\in \NN\cup \{\infty\}$ and
 let  $f({\bf X}):= \sum\limits_{q=1}^m\sum\limits_{\boldsymbol\alpha\in \Gamma_q} a_{(\boldsymbol \alpha)} {\bf X}_{\boldsymbol \alpha}
$
be a free holomorphic function with $f(0)\geq 0$ and $\Re f({\bf X})\leq I$ on the polyball ${\bf B_n}$. Then
$$
  \sum_{q=0}^m  \omega\left(\sum_{\boldsymbol\alpha\in \Gamma_q} a_{(\boldsymbol\alpha)} {\bf S}_{\boldsymbol\alpha}\right)r^q\leq 1
$$
for any $r\in [0,t_m]$, where the sequence $\{t_m\}$  is defined in Theorem \ref{OM}. In particular, the result holds if $f$ is a free holomorphic function with $\|f\|_\infty\leq 1$. Moreover, we have
$$
\sum_{q=0}^m \sup_{{\bf z}\in r{\bf B_n}(\CC)^-}\left|\sum_{\boldsymbol\alpha\in \Gamma_q} a_{(\boldsymbol\alpha)} {\bf z}_{\boldsymbol\alpha}\right|\leq 1
$$
for any $r\in [0,t_m]$. When $m=\infty$, we have $t_\infty=\frac{1}{3}$.
\end{corollary}

Due to Proposition \ref{prelim}, Theorem \ref{OM}, and Corollary \ref{omco}, we obtain the following result for the polydisk.

\begin{corollary} \label{Co-po} Let  $m\in \NN\cup \{\infty\}$ and let $f({\bf z} )=\sum\limits_{{{\bf p} \in \ZZ_+^k},  {|{\bf p}|\leq m}} a_{\bf p} {\bf z}^{\bf p}$ be a holomorphic function on the polydisk     such that  $f(0)\geq 0$  and $\Re f({\bf z})\leq 1$ for  ${\bf z}\in \DD^k$.
Then, for each $q\in \{1,\ldots, m\}$,
$$
\sup_{\boldsymbol\xi\in \TT^k}\left|\sum_{{\bf p}\in \ZZ_+^k, |{\bf p}|=q}a_{\bf p} \boldsymbol\xi^{\bf p}\right|\leq  2(1-|a_0|)\cos \frac{\pi}{\left[\frac{m}{q}\right]+2}
$$
and
$$
\sum_{q=0}^m \sup_{\boldsymbol\xi\in \TT^k}\left|\sum_{{\bf p}\in \ZZ_+^k, |{\bf p}|=q}a_{\bf p} \boldsymbol\xi^{\bf p}\right|r^q\leq 1
$$
for any $r\in [0,t_m]$, where the sequence $\{t_m\}$  is defined in Theorem \ref{OM}.
\end{corollary}

We remark the following particular cases of Corollary \ref{Co-po}.
\begin{enumerate}
\item[(i)] The results holds for any holomorphic function $f$ on the polydisk with $f(0)\geq 0$ and  $\|f\|_\infty\leq 1$.
     \item[(ii)] If $m=\infty$, then $t_\infty=\frac{1}{3}$ is the best positive constant in the Bohr type inequality (see \cite{Ai}).
          \item[(iii)]  If  $m=\infty$, $k=1$, and $n_1=1$, we recover Bohr's inequality (see \cite{Bo}).
          \end{enumerate}

\begin{corollary} \label{La} Let $f({\bf z} )=\sum\limits_{{{\bf p} \in \ZZ_+^k},  {|{\bf p}|\leq m}} a_{\bf p} {\bf z}^{\bf p}$ be a holomorphic function of degree  $m\in \NN\cup \{\infty\}$ on the polydisc $\DD^k$ such that $\Re f({\bf z})\leq 1$ for  ${\bf z}\in \DD^k$.

\begin{enumerate}
\item[(i)] If  $f(0)\geq 0$,
then
$$
\left|\left(\frac{\partial f}{\partial z_1}\right)(0)\right|+\cdots + \left|\left(\frac{\partial f}{\partial z_k}\right)(0)\right|
\leq 2(1-f(0)) \cos \frac{\pi}{m +2}.
$$
\item[(ii)]
If $m=\infty$  and ${\bf a}=(a_1,\ldots, a_k)\in \DD^k$ is such that  $f({\bf a})\geq 0$, then
$$
\sum_{i=1}^k (1-|a_i|^2)\left|\left(\frac{\partial f}{\partial z_i}\right)({\bf a})\right|
\leq  2(1-f({\bf a})).
$$
\end{enumerate}
\end{corollary}
\begin{proof} Part (i) is obtained from Corollary \ref{Co-po} when $q=1$. When $m=\infty$, the inequality   becomes
\begin{equation}\label{infinite}
\left|\left(\frac{\partial f}{\partial z_1}\right)(0)\right|+\cdots + \left|\left(\frac{\partial f}{\partial z_k}\right)(0)\right|
\leq 2(1-f(0)).
\end{equation}
Part (ii) is an extension of this inequality and can be obtained as follows. Assume that ${\bf a}=(a_1,\ldots, a_k)\in \DD^k$ is such that  $f({\bf a})\geq 0$.
For each $i\in \{1,\ldots, k\}$, let $\varphi_{a_i}$ be  the automorphism of $\DD$ given by $\varphi_{a_i}(z):=\frac{z-a_i}{1-\bar a_i z}$ for  $z\in \DD$.
Define
$$
g({\bf z}):=f\left(-\varphi_{a_1}(z_1),\ldots, -\varphi_{a_k}(z_k)\right),\qquad {\bf z}=(z_1,\ldots, z_k)\in \DD^k.
$$
Note that $g$ is holomorphic on the polydisc such that $g(0)=f({\bf a})\geq 0$ and
$\Re g({\bf z})\leq 1$ for  ${\bf z}\in \DD^k$. Applying inequality \eqref{infinite} to $g$, we obtain
$$
\sum_{i=1}^k  \left|\left(\frac{\partial f}{\partial z_i}\right)({\bf a})\varphi_{a_i}'(0)\right|
\leq  2(1-g(0)).
$$
Now, taking into account that $\varphi_{a_i}'(0)=1-|a_i|^2$, we complete the proof.
\end{proof}

\section{Harnack   inequalities for free $k$-pluriharmonic functions}

In this section, we provide Harnack type inequalities for positive free $k$-pluriharmonic function  with operator coefficients  on polyballs.

 A bounded linear operator $A\in B(\cK\otimes \bigotimes_{i=1}^k F^2(H_{n_i}))$ is called {\it $k$-multi-Toeplitz} with respect to the universal model
 ${\bf R}:=({\bf R}_1,\ldots, {\bf R}_k)$,  where  ${\bf R}_i:=({\bf R}_{i,1},\ldots,{\bf R}_{i,n_i})$, if
 $$
 (I_\cK\otimes {\bf R}_{i,s}^*)A(I_\cK\otimes {\bf R}_{i,t})=\delta_{st}A,\qquad s,t\in \{1,\ldots, n_i\},
 $$
  for every $i\in\{1,\ldots, k\}$.
 Let     $\boldsymbol{\cT_{\bf n}}$ be the set of all    $k$-multi-Toeplitz operators  on $\cK\otimes\bigotimes_{i=1}^k F^2(H_{n_i})$.
 In \cite{Po-pluriharmonic-polyball}, we proved that
\begin{equation*}\begin{split}
\boldsymbol{\cT_{\bf n}}&=\text{\rm span} \{f^*g:\ f, g\in B(\cK)\otimes_{min} \boldsymbol\cA_{\bf n}\}^{- \text{\rm SOT}}\\
&=\text{\rm span} \{f^*g:\ f, g\in B(\cK)\otimes_{min} \boldsymbol\cA_{\bf n}\}^{- \text{\rm WOT}},
\end{split}
\end{equation*}
where $\boldsymbol\cA_{\bf n}$ is the polyball algebra.
A function $F$ with operator-valued coefficients in $B(\cK)$, where $\cK$ is separable Hilbert space, is called {\it free $k$-pluriharmonic} on the abstract polyball ${\bf B_n}$    if it has the form \eqref{F(X)}.
 In particular, for any $r\in [0,1)$,
$$F(r{\bf S}):= \sum_{p_1\in \ZZ}\cdots \sum_{p_k\in \ZZ} \sum_{{\alpha_i,\beta_i\in \FF_{n_i}^+, i\in \{1,\ldots, k\}}\atop{|\alpha_i|=p_i^-, |\beta_i|=p_i^+}}A_{(\alpha_1,\ldots,\alpha_k;\beta_1,\ldots, \beta_k)}
\otimes r^{|\boldsymbol\alpha|+ |\boldsymbol\beta|}{\bf S}_{1,\alpha_1}\cdots {\bf S}_{k,\alpha_k}{\bf S}_{1,\beta_1}^*\cdots {\bf S}_{k,\beta_k}^*
$$
  is convergent in the operator
 norm topology to a $k$-multi-Toeplitz operator, with respect to the right universal model ${\bf R}=\{{\bf R}_{i,j}\}$, which is in the operator space  $\text{\rm span} \{f^*g:\ f, g\in B(\cK)\otimes_{min} \boldsymbol\cA_{\bf n}\}^{-\|\cdot\|}$.

In \cite{Po-unitary}, we introduced  a joint numerical radius for $n$-tuples of of operators $(T_1,\ldots, T_n)\in B(\cH)^n$ which turned out to be related to the classical numerical radius by the relation
\begin{equation}\label{omfo}
w(T_1,\ldots, T_n)=\omega (T_1^*\otimes S_1+\cdots + T_n^*\otimes S_n),
\end{equation}
where $S_1,\ldots, S_n$ are the left creation operators on the full Fock space
$F^2(H_n)$. The joint numerical radius has similar properties to those mentioned, in Section 6, for the classical numerical radius.

\begin{theorem} \label{H-ineq}
 Let $F$ be a positive free $k$-pluriharmonic function  with operator coefficients   and of   degree $m_i\in \NN\cup \{\infty\}$, $i\in \{1,\ldots, k\}$, with respect to the variables $ (X_{i,1}, \ldots, X_{i,n_i})$.   If $\boldsymbol\rho:=(\rho_1,\ldots, \rho_k)\in [0,1)^k$, then
 $$
   \left\|F({\bf X})\right\|\leq \|F(0)\|
   \prod_{i=1}^k \left(1+2\sum_{p_i=1}^{m_i}\rho_i^{p_i} \cos \frac{\pi}{\left[\frac{m_i}{p_i}\right]+2}\right)\leq \|F(0)\|\prod_{i=1}^k \frac{1+\rho_i}{1-\rho_i}
 $$
 for any ${\bf X}\in \boldsymbol\rho{\bf B}_{\bf n}(\cH)^-$, where $[x]$ is the integer part of $x$.
 \end{theorem}

 \begin{proof} First assume that $m_i\in \NN$.  Let  $A_{(\alpha_1,\ldots,\alpha_k;\beta_1,\ldots, \beta_k)}\in B(\cK)$, where   $\alpha_i,\beta_i\in \FF_{n_i}^+$, $|\alpha_i|=q_i^-, |\beta_i|=q_i^+$, and $-m_i\leq q_i\leq m_i$ be the coefficients in the representation of  $F$.
Then, for any $r\in [0,1)$,
$$
F(\boldsymbol\rho{\bf S})= \sum_{{i\in \{1,\ldots, k\}}\atop{-m_i\leq q_i\leq m_i}}  \sum_{{\alpha_i,\beta_i\in \FF_{n_i}^+, |\beta_i|=q_i^+, |\alpha_i|=q_i^-}}A_{(\alpha_1,\ldots,\alpha_k;\beta_1,\ldots, \beta_k)}
\otimes \left({\prod_{i=1}^k\rho_i^{|\alpha_i|+|\beta_i|}}\right)
{S}_{1,\alpha_1}{S}_{1,\beta_1}^*\otimes \cdots \otimes {S}_{k,\alpha_k}{ S}_{k,\beta_k}^*,
$$
where $S_{i,j}$ are the left creation operators acting on the full Fock space $F^2(H_{n_i})$, as defined in Section 1. We also use the notation $S_i=(S_{i,1},\ldots, S_{i, n_i})$.
  We remark that    $F(\boldsymbol\rho{\bf S})$  is a positive $k$-multi-Toeplitz operator and
$$
F(\boldsymbol\rho{\bf S})= \sum_{\alpha_k\in \FF_{n_k}^+, 1\leq |\alpha_k|\leq m_k}
C_{(\alpha_k)}^*\otimes \rho_k^{|\alpha_k|}S_{k,\alpha_k}^*+ C_0\otimes I_{F^2(H_{n_k})}+
\sum_{\alpha_k\in \FF_{n_k}^+, 1\leq |\alpha_k|\leq m_k}
C_{(\alpha_k)}\otimes \rho_k^{|\alpha_k|}S_{k,\alpha_k}.
$$
Note that $F(\boldsymbol\rho{\bf S})$ is also a positive $1$-multi-Topeplitz operator with respect to the right creation operators $R_k=(R_{k,1},\ldots, R_{k, n_k})$, with coefficients  in $B(\cK)\otimes_{min} B(F^2(H_{n_1}))\otimes_{min} \cdots \otimes_{min} B(F^2(H_{n_{k-1}}))$.
According to Theorem 2.29 from \cite{Po-unitary}, for each $p_k\in \{1,\ldots, m_k\}$, the joint numerical radius satisfies the inequality
\begin{equation}
\label{joint}
\rho_k^{p_k}w\left( C_{(\alpha_k)}:\ |\alpha_k|=p_k\right)\leq \|C_0\| \cos \frac{\pi}{\left[\frac{m_i}{p_k}\right]+2}.
\end{equation}
 Note that  $C_0\otimes I_{F^2(H_{n_k})}=F(\rho_1S_1,\ldots, \rho_{k-1}S_{k-1}, 0)$  is a positive $(k-1)$-multi-Toeplitz operator.
Using the properties of the numerical radius, we deduce that
\begin{equation}\label{omeom}
\omega\left(F(\boldsymbol\rho{\bf S})\right)\leq \omega\left(C_0\otimes I_{F^2(H_{n_k})}\right)+2\sum_{p_k=1}^{m_k}\rho_k^{p_k}\omega\left( \sum_{\alpha_k\in \FF_{n_k}^+, |\alpha_k|= p_k}
C_{(\alpha_k)}\otimes S_{k,\alpha_k}
\right).
\end{equation}
Let $Y_1,\ldots, Y_{n_k^{p_k}}$ be  an enumeration of the operators $C_{(\alpha_k)}^*$, where $\alpha_k\in \FF_{n_k}^+$ and $|\alpha_k|=p_k$, and let $\widetilde S_1,\ldots, \widetilde S_{n_k^{p_k}}$ be the left creation operators on the full Fock space $F^2(H_{n_k^{p_k}})$ with $n_k^{p_k}$ generators.
According to relation \eqref{omfo}, we have
$$
w\left( C_{(\alpha_k)}:\ |\alpha_k|=p_k\right)=\omega \left( \sum_{j=1}^{n_k^{p_k}}Y_j^*\otimes \widetilde S_j\right).
$$
On the other hand, since $\left[S_{k,\alpha_k}:\ \alpha_k\in \FF_{n_k}^+, |\alpha_k|=p\right]$ is a pure row isometry, it is unitarily equivalent to
$\left[\widetilde S_j\otimes I_\cG: \ j\in \{1,\ldots, n_k^{p_k}\}\right]$ for some separable Hilbert space $\cG$, due to the Wold type decomposition \cite{Po-isometric} for row isometries. Consequently, we obtain
$$
\omega \left( \sum_{j=1}^{n_k^{p_k}}Y_j^*\otimes \widetilde S_j\right)=
\omega\left( \sum_{\alpha_k\in \FF_{n_k}^+, |\alpha_k|= p_k}
C_{(\alpha_k)}\otimes S_{k,\alpha_k}
\right).
$$
Hence and  using relations \eqref{joint} and \eqref{omeom}, we deduce that
\begin{equation}
\label{omom}
\omega\left(F(\rho_1 S_1,\ldots \rho_k S_k)\right)\leq \omega\left(F(\rho_1 S_1,\ldots \rho_{k-1} S_{k-1}, 0)\right)
\left[ 1+2\sum_{p_k=1}^{m_k}\rho_k^{p_k} \cos \frac{\pi}{\left[\frac{m_k}{p_k}\right]+2}\right].
\end{equation}
Let us show that this inequality remains true even when
   $m_k=+\infty$. Indeed, in this case    we have $0\leq F(\boldsymbol\rho {\bf S})\leq G(\boldsymbol\rho {\bf S})$, where
$$  G(\boldsymbol\rho {\bf S}):=\sum_{\alpha_k\in \FF_{n_k}^+, 1\leq |\alpha_k|\leq q_k}
C_{(\alpha_k)}^*\otimes \rho_k^{|\alpha_k|}S_{k,\alpha_k}^*+ (C_0\otimes I_{F^2(H_{n_k})}+ M_{q_k}(\boldsymbol \rho)I)+
\sum_{\alpha_k\in \FF_{n_k}^+, 1\leq |\alpha_k|\leq q_k}
C_{(\alpha_k)}\otimes \rho_k^{|\alpha_k|}S_{k,\alpha_k}
$$
and
$$
M_{q_k}(\boldsymbol \rho):=\left\|\sum_{\alpha_k\in \FF_{n_k}^+|\alpha_k|> q_k}
C_{(\alpha_k)}^*\otimes \rho_k^{|\alpha_k|}S_{k,\alpha_k}^*+
\sum_{\alpha_k\in \FF_{n_k}^+,  |\alpha_k|> q_k}
C_{(\alpha_k)}\otimes \rho_k^{|\alpha_k|}S_{k,\alpha_k}\right\|
$$
Note that $M_{q_k}(\boldsymbol \rho)\to 0$ as $q_k\to \infty$. As in the proof above, but written for $G(\boldsymbol\rho {\bf S})$, and taking $q_k\to \infty$, we obtain
 \begin{equation*}
 \begin{split}
\omega\left(F(\rho_1 S_1,\ldots \rho_k S_k)\right)&\leq \omega\left(F(\rho_1 S_1,\ldots \rho_{k-1} S_{k-1}, 0)\right)
\left[ 1+2\sum_{p_k=1}^\infty\rho_k^{p_k} \right]\\
&\leq \omega\left(F(\rho_1 S_1,\ldots \rho_{k-1} S_{k-1}, 0)\right) \frac{1+\rho_k}{1-\rho_k},
\end{split}
\end{equation*}
which   can be obtained from \eqref{omom} passing to the limit as $m_k\to \infty$. Therefore, the inequality \eqref{omom} is valid if $m_k\in \NN\cup \{\infty\}$.
Since $F(\rho_1 S_1,\ldots \rho_{k-1} S_{k-1}, 0)$ is a positive  $k-1$ multi-Toeplitz operator with coefficients  in
$B(\cE)\otimes_{min} B(F^2(H_{n_1}))\otimes_{min} \cdots \otimes_{min} B(F^2(H_{n_{k-2}}))$, as above, we obtain the inequality
$$
\omega\left(F(\rho_1 S_1,\ldots \rho_{k-1} S_{k-1}, 0)\right)\leq \omega\left(F(\rho_1 S_1,\ldots \rho_{k-2} S_{k-2}, 0, 0)\right)
\left[ 1+2\sum_{p_{k-1}=1}^{m_{k-1}}\rho_{k-1}^{p_{k-1}} \cos \frac{\pi}{\left[\frac{m_{k-1}}{p_{k-1}}\right]+2}\right].
$$
Continuing this process, and putting together  the inequalities obtained, we deduce that
$$
\omega\left(F(\boldsymbol\rho{\bf S})\right)=
\omega\left(F(\rho_1 S_1,\ldots \rho_{k} S_{k})\right)\leq
\omega(F(0))\prod_{i=1}^k \left(1+2\sum_{p_i=1}^{m_i}\rho_i^{p_i} \cos \frac{\pi}{\left[\frac{m_i}{p_i}\right]+2}\right).
$$
According to the multivariable dilation theory \cite{Po-Berezin-poly}, each element ${\bf X}=\{X_{i,j}\}$ in the polyball ${\bf B_n}(\cH)$,  has a dilation  $\{{\bf S}_{i,j}\otimes I_\cM\}$ for some separable Hilbert space $\cM$ such that $X_{i,j}^*=({\bf S}_{i,j}^*\otimes I_\cM)|_\cH$ for any $i\in \{1,\ldots, k\}$ and $j\in \{1,\ldots, n_i\}$.
Using to the properties of the numerical radius, we obtain
$$
\omega\left(F(\boldsymbol\rho{\bf X})\right)=\omega\left(F(\rho_1 X_1,\ldots \rho_{k} X_{k})\right)\leq
\omega\left(F(\rho_1 S_1,\ldots \rho_{k} S_{k})\right)=\omega\left(F(\boldsymbol\rho{\bf S})\right),
$$
which combined with the latter inequality and the fact that the numerical radius of a positive operator coincides with its norm, implies the inequality of the theorem. The proof is complete.
\end{proof}

\begin{corollary}
 Let $F$ be a positive free $k$-pluriharmonic function  with scalar coefficients   and of   degree $m_i\in \NN\cup \{\infty\}$, $i\in \{1,\ldots, k\}$,  with respect to the variables $ (X_{i,1}, \ldots, X_{i,n_i})$, then
 $$
   F({\bf X})\leq F(0)
   \prod_{i=1}^k \left(1+2\sum_{p_i=1}^{m_i}\rho_i^{p_i} \cos \frac{\pi}{\left[\frac{m_i}{p_i}\right]+2}\right)\leq F(0)\prod_{i=1}^k \frac{1+\rho_i}{1-\rho_i}
 $$
 for any ${\bf X}\in \boldsymbol\rho{\bf B}_{\bf n}(\cH)^-$    and  $\boldsymbol\rho:=(\rho_1,\ldots, \rho_k)\in [0,1)^k$, where $[x]$ is the integer part of $x$.
 \end{corollary}

       %

\end{document}